\documentclass[11pt,twoside, leqno]{article}

\usepackage{amssymb}
\usepackage{amsmath}
\usepackage{amsthm}

\allowdisplaybreaks

\def\rr{{\mathbb R}}
\def\rn{{{\rr}^n}}

\def\zz{{\mathbb Z}}
\def\nn{{\mathbb N}}
\def\hh{{\mathbb H}}

\def\cc{{\mathbb C}}

\def\cx{{\mathcal X}}
\def\bbg{{\mathbb G}}

\def\cd{{\mathcal D}}

\def\cg{{\mathcal G}}
\def\cl{{\mathcal L}}

\def\cb{{\mathcal B}}
\def\ca{{\mathcal A}}
\def\ccc{{\mathcal C}}

\def\fz{\infty}
\def\az{\alpha}
\def\supp{{\mathop\mathrm{\,supp\,}}}

\def\loc{{\mathop\mathrm{\,loc\,}}}

\def\lz{\lambda}
\def\dz{\delta}

\def\ez{\epsilon}
\def\ezl{{\epsilon_1}}
\def\ezz{{\epsilon_2}}

\def\kz{\kappa}
\def\bz{\beta}

\def\gz{{\gamma}}

\def\vz{\varphi}
\def\tz{\theta}
\def\sz{\sigma}

\def\wz{\widetilde}

\def\hs{\hspace{0.3cm}}

\def\ls{\lesssim}
\def\gs{\gtrsim}

\def\fin{{\mathop\mathrm{fin}}}

\def\ati{{\mathrm {AOTI}}}

\def\diam{{\mathop\mathrm{\,diam\,}}}

\def\dint{\displaystyle\int}

\def\dfrac{\displaystyle\frac}

\def\r{\right}
\def\lf{\left}

\newtheorem{thm}{Theorem}[section]
\newtheorem{lem}{Lemma}[section]
\newtheorem{prop}{Proposition}[section]
\newtheorem{rem}{Remark}[section]

\newtheorem{defn}{Definition}[section]

\numberwithin{equation}{section}

\begin{document}

\arraycolsep=1pt

\title{{\vspace{-5cm}\small\hfill\bf Trans. Amer. Math. Soc., to appear}\\
\vspace{4cm}\Large\bf Localized Hardy Spaces $H^1$ Related to
Admissible Functions on RD-Spaces
and Applications to Schr\"odinger
Operators \footnotetext{\hspace{-0.35cm} {\it 2000 Mathematics Subject
Classification}. {Primary 42B30; Secondary 42B20, 42B25, 42B35.}
\endgraf{\it Key words and phrases.} Space of homogeneous type,
Heisenberg group, connected and simply connected Lie group, admissible
function, Hardy space, maximal function, atom, Schr\"odinger operator, reverse
H\"older inequality, Riesz transform.
\endgraf
The first author is supported by the National
Natural Science Foundation (Grant No. 10871025) of China.}}
\author{Dachun Yang and Yuan Zhou}
\date{ }
\maketitle

\begin{center}
\begin{minipage}{13.5cm}\small
{\noindent{\bf Abstract.}
Let ${\mathcal X}$ be an RD-space, which means
that ${\mathcal X}$ is a space of homogenous type
in the sense of Coifman and Weiss with the additional property that
a reverse doubling property holds in ${\mathcal X}$.
In this paper, the authors first introduce the notion of
admissible functions $\rho$
and then develop a theory of localized
Hardy spaces $H^1_\rho ({\mathcal X})$
associated with $\rho$,
which includes several maximal function
characterizations of $H^1_\rho ({\mathcal X})$,
the relations between $H^1_\rho ({\mathcal X})$
and the classical Hardy space $H^1({\mathcal X})$ via constructing a
kernel function related to $\rho$,
the atomic decomposition characterization of $H^1_\rho ({\mathcal X})$,
and the boundedness of certain localized singular
integrals on $H^1_\rho({\mathcal X})$
via a finite atomic decomposition characterization
of some dense subspace of $H^1_\rho ({\mathcal X})$.
This theory has a wide range of applications.
Even when this theory is applied, respectively, to
the Schr\"odinger operator or
the degenerate Schr\"odinger operator on $\rn$,
or the sub-Laplace Schr\"odinger operator on Heisenberg
groups or connected and simply connected nilpotent Lie groups,
some new results are also obtained.
The Schr\"odinger operators considered here are associated with nonnegative
potentials satisfying the reverse H\"older inequality.}
\end{minipage}
\end{center}

\vspace{0.2cm}

\section{Introduction\label{s1}}

The theory of Hardy spaces on the Euclidean
space $\rn$ plays an important role in various fields of analysis
and partial differential equations; see, for examples,
\cite{sw,fs72,clms,s93,g08}. One of the most important applications
of Hardy spaces is that they are good substitutes of Lebesgue
spaces when $p\in(0,\,1]$. For example,
when $p\in(0,\,1]$, it is well-known that Riesz
transforms are not bounded on $L^p(\rn)$, however, they are
bounded on Hardy spaces. A localized version of Hardy spaces
on $\rn$ was first introduced by Goldberg \cite{g}. These
classical Hardy spaces are essentially related to the Laplace
operator $\Delta\equiv-\sum_{j=1}^n(\frac{\partial}{\partial x_j})^2$ on $\rn$.

On the other hand, the studies of Schr\"odinger operators with
nonnegative potentials satisfying
the reverse H\"older inequality obtain an increasing interest;
see, for example, \cite{f83,z99,s95,l99,d98,dz99,ks00a,ks00b,
dz02,dz04,d05,dy05b,dgttz,lll,aa}.
In particular, Fefferman \cite{f83}, Shen \cite{s95} and Zhong \cite{z99} established
some basic results, including estimates of the fundamental solutions and the
boundedness on Lebesgue spaces of Riesz transforms, for the
Schr\"odinger operator $\cl\equiv\Delta+V$ on $\rn$
with $n\ge3$ and the nonnegative potential $V$ satisfying
the reverse H\"older inequality.
Lu \cite{l96} extended part of these
results to the sub-Laplace Schr\"odinger operator
on stratified groups,
and Li \cite{l99} on connected and simply connected nilpotent
Lie groups. Kurata and Sugano \cite{ks00a}
extended some of these results
to the degenerate Schr\"odinger operator on $\rn$ with $n\ge3$.
On the other hand, Dziuba\'nski and Zienkiewicz \cite{dz99}
first characterized the Hardy space $H^1_\cl(\rn)$ for Schr\"odinger operators
via atoms, the maximal function defined by the semigroup
generated by $\cl$ and the Riesz transforms $\nabla\cl^{-1/2}$,
which were further generalized by C. Lin, H. Liu and Y. Liu \cite{lll}
to Heisenberg groups. Also, Duong and Yan \cite{dy05b} established
the Lusin-area function and molecular characterizations of Hardy
spaces $H^1_\cl(\rn)$ associated to the operator $\cl$ with heat kernel bounds,
which includes the Schr\"odinger operator with
nonnegative potential as an example.
Dziuba\'nski \cite{d05} further obtained the atomic characterization
and the maximal function characterization of the semigroup
generated by $\cl$ for Hardy spaces $H^1_\cl(\rn)$
associated with the degenerate Schr\"odinger operator $\cl$
on $\rn$ via a theory of Hardy spaces on spaces of homogeneous
type with the additional assumption that the measure of any ball
is equivalent to its radius in \cite{cw77, ms79b, u80}.

Recently, a theory of Hardy spaces on so-called RD-spaces
were established in \cite{hmy06,hmy2,gly1,gly2}.
A space  $\cx$ of homogenous type in the sense of Coifman and Weiss is called
 an RD-space if $\cx$ has the additional property that
a reverse doubling property holds in $\cx$ (see \cite{hmy2}).
It is well-known that a connected space of homogeneous type
is an RD-space. Typical examples of RD-spaces
include Euclidean spaces, Euclidean spaces with
weighted measures satisfying the doubling property,
Heisenberg groups, Lie groups of polynomial growth (\cite{v88, vsc92})
and the boundary of an unbounded model
polynomial domain in $\cc^2$ (\cite{ns04, ns06}), or
more generally, Carnot-Carath\'eodory spaces with
doubling measures (\cite{nsw85, hmy2}). Throughout this
paper, we only consider those RD-spaces with infinity total
measures.

Motivated by the properties of nonnegative
potentials satisfying the reverse H\"older inequality
in aforementioned Schr\"odinger operators,
in this paper, we first introduce a class of admissible functions
$\rho$ on $\cx$. Via establishing some basic properties of $\rho$, we develop
a theory of Hardy spaces $H^1_\rho (\cx)$
associated to admissible functions $\rho$,
which includes several maximal function
characterizations of $H^1_\rho (\cx)$,
the relations between $H^1_\rho (\cx)$
and the classical Hardy space $H^1(\cx)$ via
constructing a kernel function
related to $\rho$, the atomic decomposition
characterization of $H^1_\rho (\cx)$,
and the boundedness of certain localized singular
integrals on $H^1_\rho (\cx)$
via a finite atomic decomposition characterization
of some dense subspace of $H^1_\rho (\cx)$.
Since these results hold for any admissible
function $\rho$ and any RD-space $\cx$, they have a wide range
of applications. Moreover, even when this theory is applied, respectively, to
the Schr\"odinger operator or
the degenerate Schr\"odinger operator on $\rn$,
or the sub-Laplace Schr\"odinger operator on Heisenberg
groups or connected and simply connected nilpotent Lie groups,
we also obtain some new results. Precisely,
this paper is organized as follows.

In Section 2, we first recall some notation and notions from \cite{hmy2}.
Then we introduce the notions of admissible functions $\rho$,
localized Hardy spaces $H^1_\rho (\cx)$ defined by the
grand maximal functions, and atomic Hardy spaces $H^{1,\,q}_{\rho}(\cx)$.
Some properties of admissible functions are also presented,
which are used through the whole paper.
We also recall some results on classical Hardy spaces
$H^1(\cx)$ from \cite{hmy06, hmy2, gly1, gly2}.

One key step of this paper is
to construct a kernel function on
$\cx\times\cx$ associated to any given admissible function
$\rho$ in Proposition \ref{p3.1} below
by subtly exploiting some ideas originally from Coifman \cite{djs}.
A suitable variant of this kernel function actually yields an approximation
of the identity related to $\rho$. This may be very useful
in establishing a theory of Besov and Triebel-Lizorkin spaces
including Hardy spaces $H^p(\cx)$ when $p\le 1$ but near
to $1$ and fractional Sobolev spaces; see \cite{hmy06,hmy2}.
Using this kernel function, in Section \ref{s3} of this paper,
we establish the relations between $H^1_\rho (\cx)$
and $H^1(\cx)$ (see Theorem \ref{t3.1} below), and as an application,
we further obtain an atomic decomposition characterization of $H^1_\rho (\cx)$
via $(1,\,q)_{\rho}$-atoms with $q\in(1,\,\fz]$
(see Theorem \ref{t3.2} (i) below). Moreover, for certain dense subspace
of $H^1_\rho (\cx)$, we establish its finite atomic
decomposition characterization
via $(1,\,q)_{\rho}$-atoms with $q<\fz$
 and continuous $(1,\,\fz)_{\rho}$-atom  (see Theorem \ref{t3.2} (ii) below).
As an application of this result, we establish a general boundedness
criterion for sublinear operators on $H^1_\rho (\cx)$
via atoms (see Proposition \ref{p3.2} below), and then
we obtain the boundedness on $H^1_\rho (\cx)$
of certain localized singular integrals (see Proposition \ref{p3.3} below),
which is useful in establishing the boundedness of Riesz
transforms related to Schr\"odinger operators in Section \ref{s5}.

In Section 4, we establish a radial maximal function characterization of
$H^1_\rho (\cx)$; see Theorem \ref{t4.1} below.
For the sake of applications,
we also characterize $H^1_\rho (\cx)$
via a variant of the radial maximal functions, which is closely
related to the considered admissible function $\rho$;
see Theorem \ref{t4.2} below. We should point out that
the method used to obtain the radial maximal function characterization of
$H^1_\rho (\cx)$ is totally different
from the method used by Dziuba\'nski and Zienkiewicz in \cite{d98,dz99,dz02,d05}
to obtain a similar result on $\rn$.
The method in \cite{d98,dz99,dz02,d05} strongly depends on
an existing theory of localized Hardy spaces $h^1$,
on $\rn$ or on spaces of homogeneous
type with the additional assumption that the measure of any ball
is equivalent to its radius,
in the sense of Goldberg \cite{g}. We successfully avoid this via the
discrete Calder\'on reproducing formula from
\cite{hmy2,gly1} and a subtle split of dyadic cubes
of Christ in \cite{ch}.

In Section \ref{s5}, we apply the results obtained in Sections \ref{s3}
and \ref{s4}, respectively, to the Schr\"odinger operator or the
degenerate Schr\"odinger operator on $\rn$, the
sub-Laplace Schr\"odinger operator on Heisenberg groups or on
connected and simply connected nilpotent Lie groups.
The nonnegative potentials of these Schr\"odinger operators are assumed to
satisfy the reverse H\"older inequality.
Even for these special cases, our results
further complement the results in \cite{dz99,dz02,d05,lll}.
Especially for the sub-Laplace Schr\"odinger operator on
connected and simply connected nilpotent
Lie groups, Theorem \ref{t5.1} through Theorem \ref{t5.4} below seem unknown before.

Moreover, in forthcoming papers, we will develop a dual theory for
$H^1_\rho (\cx)$ and will also apply these results
obtained in this paper to the sub-Laplace Schr\"odinger operator
with nonnegative potentials satisfying the reverse H\"older inequality
on the boundary of an unbounded model
polynomial domain in $\cc^2$ appeared in \cite{ns04,ns06}.

We finally make some conventions. Throughout this paper, we always
use $C$ to denote a positive constant that is independent of the main
parameters involved but whose value may differ from line to line.
Constants with subscripts, such as $C_1$, do not change in
different occurrences. If $f\le Cg$, we then write $f\ls g$ or
$g\gs f$; and if $f\ls g\ls f$, we then write $f\sim g$.
We also denote $\max\{\bz,\,\gz\}$ and $\min\{\bz,\,\gz\}$,
respectively, by $\bz\vee\gz$ and $\bz\wedge\gz$. For any set $E\subset\cx$,
set $E^\complement\equiv(\cx\setminus E)$.

\section{Preliminaries\label{s2}}

We first recall the notions of spaces of homogeneous type in the
sense of Coifman and Weiss \cite{cw71,cw77} and RD-spaces
in \cite{hmy2}.

\begin{defn}\rm\label{d2.1}
Let $(\cx,\, d)$ be a metric space with a
regular Borel measure $\mu$
such that all balls defined by $d$ have finite and positive measure.
For any $x\in \cx $ and $r>0$, set the ball $B(x,r)\equiv\{y\in \cx :\
d(x,y)<r\}.$

(i) The triple $(\cx,\,d,\,\mu)$ is called a
    space of homogeneous type if there exists
    a constant $C_1\ge 1$ such that for all $x\in \cx $ and $r>0$,
    $$\mu(B(x, 2r))\le C_1\mu(B(x,r))\ (\mathrm{\it doubling\ property}).$$

(ii) The triple $(\cx,\,d,\,\mu)$ is called an RD-space if  there exist
 constants $0<\kz\le n$ and $C_2\ge1$ such that for all
    $x\in \cx$, $0<r<\diam(\cx)/2$ and $1\le\lz<\diam(\cx)/(2r)$,
     \begin{equation}\label{2.1}
     (C_2)^{-1}\lz^\kz\mu(B(x,r))\le\mu(B(x,\lz r))\le C_2\lz^ n\mu(B(x,r)),
        \end{equation}
    where $\diam(\cx)=\sup_{x,\,y\in\cx}d(x,y)$.
\end{defn}

\begin{rem}\label{r2.1}\rm
(i) Obviously, an RD-space is a space of homogeneous type.
Conversely, a space of homogeneous type automatically
satisfies the second inequality of \eqref{2.1}.
Moreover, it was proved in \cite[Remark 1.1]{hmy2}
that if $\mu$ is doubling, then $\mu$ satisfies \eqref{2.1}
if and only if there exist constants $a_0>1$ and $ C_0>1$ such that for
all $x\in\cx$ and $0<r<\diam(\cx)/a_0$,
$$\mu(B(x, a_0r))\ge  C_0\mu(B(x,r))\quad
(\mathrm{\rm reverse\ doubling\ property})$$ (If $a_0=2$, this is
the classical reverse doubling condition), and equivalently,
for all $x\in\cx$ and $0<r<\diam(\cx)/a_0$, $(B(x, a_0r)\setminus
B(x, r))\neq\emptyset$, which, as pointed out to us by the referee,
is known in the topology as uniform perfectness.
For more equivalent characterizations of RD-spaces, see \cite{yz09}.

(ii) Let $d$ be a quasi-metric, which means that there exists $A_0\ge1$
such that for all $x$, $y$, $z\in\cx$,
$d(x, y)\le A_0(d(x, z)+d(z, y))$.
Recall that Mac\'ias and Segovia \cite[Theorem 2]{ms79a} proved that
there exists an equivalent quasi-metric $\wz d$ such that
all balls corresponding to $\wz d$ are open in the topology
induced by $\wz d$, and there exist constants $\wz A_0>0$ and $\tz\in(0, 1)$
such that for all $x$, $y$, $z\in\cx$,
\begin{equation*}
\lf|\wz d(x, z)-\wz d(y, z)\r|
\le \wz A_0\lf[\wz d(x, y)\r]^\tz\lf[\wz d(x, z)+\wz d(y, z)\r]^{1-\tz}.
\end{equation*}
It is known that the approximation of the identity
as in Definition \ref{d2.3} below also exists for $\wz d$; see \cite{hmy2}.
Obviously, all results in this section and Sections \ref{s3} and \ref{s4}
are invariant on equivalent
quasi-metrics. From these facts, it follows that all
conclusions of this section and Sections \ref{s3} and \ref{s4}
are still valid for quasi-metrics (especially, for so-called $d$-spaces
of Triebel; see \cite[p.\,189]{t06}).
\end{rem}

Throughout the whole paper, we always
assume that $\cx$ is an RD-space and $\mu(\cx)=\fz$.
In what follows, for any $x,\,y\in\cx$ and $r\in(0,\,\fz)$,
we set $V_r(x)\equiv\mu(B(x,\,r))$ and
$V(x,\,y)\equiv\mu(B(x,\,d(x,\,y)))$.

\subsection{Admissible functions}\label{s2.1}

We first introduce the notion of admissible functions.

\begin{defn}\rm\label{d2.2}
A positive function $\rho$ on $\cx$ is called admissible if there exist
positive constants $C_3$ and $k_0$ such that for all $x,\,y\in\cx$,
\begin{equation}\label{2.2}
\rho(y)\le C_3[\rho(x)]^{1/(1+k_0)}[\rho(x)+d(x,\,y)]^{k_0/(1+k_0)}.
\end{equation}
\end{defn}

Obviously, if $\rho$ is a constant function, then $\rho$ is admissible.
Another non-trivial class of admissible functions is given by the well-known reverse
H\"older class $\cb_q(\cx,\,d,\,\mu)$ (see, for example, \cite{g73,m72,s95} for
its definition on $\rn$,
and \cite{st} for its definition on spaces of homogenous type).
Recall that a nonnegative potential $U$ is said to belong to
$\cb_q(\cx,\,d,\,\mu)$ (for short, $\cb_q(\cx)$) with $q\in(1,\,\fz]$
if there exists a positive constant $C$
such that for all balls $B$,
$$\lf\{\frac1{\mu(B)}\int_B[U(y)]^q\,d\mu(y)\r\}^{1/q}
\le C\frac1{\mu(B)}\int_BU(y)\,d\mu(y)$$
with the usual modification when $q=\fz$.
It was proved in \cite[pp.\,8-9]{st} that if $U\in\cb_q(\cx)$ for some $q\in(1,\,\fz]$
and the measure $U(z)\,d\mu(z)$ has the doubling property,
then $U$ is an $\ca_p(\cx,\,d,\,\mu)$-weight for some $p\in[1,\,\fz)$
in the sense of Muckenhoupt,
and also $U\in \cb_{q+\ez}(\cx)$ for some $\ez>0$.
Here it should be pointed out that,  generally, $U\in\cb_q(\cx)$ cannot
imply the doubling property of $U(z)\,d\mu(z)$,
but when $\mu(B(x,\,r))$ increases continuous respect to $r$ for all $x\in\cx$,
$U\in\cb_q(\cx)$ does imply the doubling property
of $U(z)\,d\mu(z)$ by \cite[Theorem 17]{st}.
We also refer the reader to \cite{ma} for other
conditions to guarantee the doubling property of $U(z)\,d\mu(z)$.
Following \cite{s95}, for all $x\in\cx$, set
\begin{equation}\label{2.3}
\rho(x)\equiv\sup\lf\{r>0:\ \frac{r^2}{V_r(x)}
\int_{B(x,\,r)}U(y)\,d\mu(y)\le1\r\},
\end{equation}
where we recall that  $V_r(x)\equiv\mu(B(x,\,r))$
for all $x\in\cx$ and $r>0$.
Then we have the following conclusion.

\begin{prop}\label{p2.1}
Let $q\in(1\vee (n/2),\,\fz]$ and  $U\in \cb_q(\cx)$.
If the measure $U(z)\,d\mu(z)$ has the doubling property,
then $\rho$ as in \eqref{2.3} is an admissible function.
\end{prop}

\begin{proof}
For any fixed $y\in\cx$ and $0<r<R<\fz$, by the H\"older inequality,
$U\in\cb_q(\cx)$ and the doubling property of $\mu$, we have
\begin{eqnarray}
\label{2.4}\frac{r^2}{V_r(y)}\int_{B(y,\,r)}U(z)\,d\mu(z)
&&\ls r^2\lf\{\frac1{V_r(y)}\int_{B(y,\,r)}[U(z)]^q\,d\mu(z)\r\}^{1/q}\\
&&\ls r^2\lf[\frac{V_R(y)}{V_r(y)}\r]^{1/q}
\frac1{V_R(y)}\int_{B(y,\,R)}U(z)\,d\mu(z)\nonumber\\
&&\ls\lf(\frac rR\r)^{2-n/q}\frac{R^2}{V_R(y)}\int_{B(y,\,R)}U(z)\,d\mu(z).\nonumber
\end{eqnarray}

By the assumption that  $U(z)\,d\mu(z)$ has the doubling property,
so there exist positive constants $C$
and $n_1>\{(\kz-n/q)\vee 0\}$
such that for all $\lz>1$, $r>0$ and $x\in\cx$,
\begin{equation}\label{2.5}
\int_{B(x,\,\lz r)}U(z)\,d\mu(z)\le C\lz^{n_1} \int_{B(x,\, r)}U(z)\,d\mu(z).
\end{equation}
By \eqref{2.4} and the fact that $q>n/2$, there exists at least one $r>0$
such that
$$\frac{r^2}{V_r(y)}\int_{B(y,\,r)}U(z)\,d\mu(z)\le1$$
and
$$\lim_{R\to\fz}\frac{R^2}{V_R(y)}\int_{B(y,\,R)}U(z)\,d\mu(z)=\fz,$$
which imply that $0<\rho(y)<\fz$.
Thus, from \eqref{2.5}, it further follows that
\begin{equation}\label{2.6}
    \frac{[\rho(y)]^2}{V_{\rho(y)}(y)}\int_{B(y,\,\rho(y))}U(z)\,d\mu(z)\sim 1.
\end{equation}

Now we prove that $\rho$ satisfies \eqref{2.2}.
For any fixed $x,\,y\in\cx$,
if $d(x,\,y)<\rho(y)$, then by the doubling property of $\mu$ and \eqref{2.6}, we have
\begin{eqnarray*}
\frac{[\rho(y)]^2}{V_{\rho(y)}(x)}\int_{B(x,\,\rho(y))}U(z)\,d\mu(z)
&&\sim \frac{[\rho(y)]^2}{V_{\rho(y)}(y)}\int_{B(y,\,\rho(y))}U(z)\,d\mu(z)\sim1.
\end{eqnarray*}
This together with \eqref{2.4} implies that $\rho(y)\sim\rho(x)$
and hence, \eqref{2.2} holds in this case.
If $d(x,\,y)\ge\rho(y)$, then there exists $j\in\nn$ such that
 $2^{j-1}\rho(y)\le d(x,\,y)<2^j\rho(y)$.
Thus for any integer $k>j$, if we choose
$r_k\equiv2^{j-k}\rho(y)\in(0,\,\rho(y))$, then by \eqref{2.4}, \eqref{2.5} and \eqref{2.6}, we have
\begin{eqnarray*}
\frac{r_k^2}{V_{r_k}(x)}\int_{B(x,\,r_k)}U(z)\,d\mu(z)
&&\ls 2^{kn/q}
\frac{r_k^2}{V_{2^kr_k}(x)}\int_{B(x,\,2^kr_k)}U(z)\,d\mu(z)\\
&&\ls 2^{kn/q}2^{2(j-k)}
\frac{[\rho(y)]^2}{V_{2^j\rho(y) }(y)}\int_{B(y,\,2^j\rho(y))}U(z)\,d\mu(z)\\
&&\ls 2^{-k(2-n/q)}2^{-j(\kz-n_1-2)}.
\end{eqnarray*}
Notice that $q>n/2$ and $n_1>\kz-n/q$ imply that
$n_1+2-\kz> 2-n/q>0$. Let $k$ be the maximal positive integer no more than
$1 +j(n_1+2-\kz)/(2-n/q).$
Then
$$\frac{r_k^2}{V_{r_k}(x)}\int_{B(x,\,r_k)}U(z)\,d\mu(z)\ls 1,$$
which together with \eqref{2.4} implies that
$$\rho(x)\gs r_1\sim 2^{j-k}\rho(y)
\sim 2^{-j\{(n_1+2-\kz)/(2-n/q)-1\}}\rho(y).$$
Let $k_0\equiv (n_1+2-\kz)/(2-n/q)-1$. Then $k_0>0$ and
$$\rho(y)\ls[\rho(x)]^{1/(1+k_0)} [2^j\rho(y)]^{k_0/(1+k_0)}
\ls[\rho(x)]^{1/(1+k_0)}[d(x,\,y)]^{k_0/(1+k_0)},$$
which also implies that \eqref{2.2} holds in this case and hence,
completes the proof of Proposition \ref{p2.1}.
\end{proof}

We now establish some properties of admissible functions.

\begin{lem}\label{l2.1}
Let $\rho$ be an admissible function. Then

(i) for any $\wz C>0$, there exists a positive constant $C$, depending on $\wz C$,
such that if $d(x,\,y)\le \wz C\rho(x)$,
then
$C^{-1}\rho(y)\le \rho(x)\le C\rho(y);$

(ii) there exists a positive constant $C$ such that for all
$x, \, y\in\cx$,
$$C^{-1}[\rho(x)+d(x,\,y)]\le \rho(y)+d(x,\,y)\le C[\rho(x)+d(x,\,y)];$$

(iii) there exists a positive constant $C_4$ such that for all
$x, \, y\in\cx$,
$$\rho(y)\ge C_4[\rho(x)]^{1+k_0}[\rho(x)+d(x,\,y)]^{-k_0}.$$
\end{lem}

\begin{proof}
If $d(x,\,y)\le \wz C\rho(x)$, then
by \eqref{2.2}, $\rho(y)\ls\rho(x)$.
By \eqref{2.2} with exchanging $x$ and $y$ again, we have
$\rho(x)\ls [\rho(y)]^{1/(1+k_0)}[\rho(x)]^{k_0/(1+k_0)}$, which implies that
$\rho(x)\ls\rho(y)$. Thus (i) holds.

To prove (ii), if $\rho(x)\le d(x,\,y)$, then it is easy to see that
$\rho(x)+d(x,\,y)\ls \rho(y)+d(x,\,y);$ if
$\rho(x)>d(x,\,y)$, then by (i), $\rho(y)\sim \rho(x)$, which implies that
$$\rho(x)+d(x,\,y)\sim \rho(y)+d(x,\,y).$$ By symmetry, we have (ii).

To prove (iii), by \eqref{2.2} exchanging $x$ and $y$,  and (ii), we have
$$\rho(x)\ls[\rho(y)]^{1/(1+k_0)}[\rho(x)+d(x,\,y)]^{k_0/(1+k_0)},$$
which gives (iii). This finishes the proof of Lemma \ref{l2.1}.
\end{proof}

For each $m\in\zz$, let
$\cx_m\equiv\{x\in\cx:\ 2^{-(m+1)/2}<\rho(x)/8\le2^{-m/2}\}.$
Then, obviously, $\cx=\cup_{m\in\zz}\cx_m$.
Moreover, using some ideas from \cite{dz99} on $\rn$,
we have the following results.

\begin{lem}\label{l2.2}
There exists a positive constant $C_5$ such that for all $R\ge 2$ and
$m,\,m'\in\zz$, if $x\in\cx_m$ and $(\cx_{m'}\cap
B(x,\,2^{-m/2}R))\ne\emptyset$, then $|m'-m|\le C_5\log R$.
\end{lem}

\begin{proof}
If $x\in\cx_m$ and $y\in (\cx_{m'}\cap B(x,\,2^{-m/2}R))$, then by
\eqref{2.2} and Lemma \ref{l2.1} (iii), we have
$$R^{-k_0}2^{-m/2}\ls\rho(y)\ls R^{k_0/(1+k_0)}2^{-m/2},$$
which implies that $$R^{-k_0}2^{-m/2}\ls 2^{-m'/2}\ls
R^{k_0/(1+k_0)}2^{-m/2},$$
namely, $R^{-k_0}\ls 2^{(m-m')/2}\ls
R^{k_0/(1+k_0)}$. Thus, $|m'-m|\ls \log R$, which completes
the proof of Lemma \ref{l2.2}.
\end{proof}

\begin{lem} \label{l2.3}
There exist positive constant $C$ and subset
 $\{x_{(m,\,k)}:\ x_{(m,\,k)}\in\cx_m\}_{m\in\zz,\,k}$
such that for all $R\ge2$ and $m\in\zz$, $\cx_m\subset\lf[\cup_{k}
B(x_{(m,\,k)},\,2^{-m/2})\r]$ and $$\sharp\{(m',\,k'):\
(B(x_{(m,\,k)},\,R2^{-m/2})\cap
B(x_{(m',\,k')},\,R2^{-m'/2}))\ne\emptyset\}\le R^{C},$$
where $\sharp E$ denotes the cardinality of any set $E$.
\end{lem}

\begin{proof}
For each fixed $m\in\zz$, since $\cx_m\subset \lf[\cup_{x\in\cx_m}
B(x,\,\frac152^{-m/2})\r],$ using
the standard $5$-covering theorem (see, for example, Theorem 1.2 in \cite{h99}), we obtain a
subset $\{x_{(m,\,k)}\}_{k}$ of $\cx_m$ such that
$$\cx_m\subset\lf\{ \bigcup_{x\in\cx_m}
B\lf(x,\,\frac152^{-m/2}\r)\r\}
\subset\lf\{ \bigcup_{k} B(x_{(m,\,k)},\,2^{-m/2})\r\},$$
and  $\{B(x_{(m,\,k)},\,\frac152^{-m/2})\}_{k}$ are disjointed.

Assume that $(B(x_{(m,\,k)},\,R2^{-m/2})\cap B(x_{(m',\,k')},\,R2^{-m'/2}))\ne\emptyset$.
If $m\le m'$, then
$$B(x_{(m,\,k)},\,2R2^{-m/2})\cap\cx_{m'}\ne\emptyset;$$
and if $m>m'$, then $(B(x_{(m',\,k')},\,2R2^{-m'/2})\cap\cx_{m})\ne\emptyset.$
Thus, by Lemma \ref{l2.2},
\begin{equation}\label{2.7}
|m-m'|\le C_5\log (2R).
\end{equation}

Moreover, for any fixed $m'$, if $y\in B(x_{(m',\,k')},\,R2^{-m'/2})$, then
\begin{eqnarray*}
d(x_{(m,\,k)},\,y)&&\le d(x_{(m,\,k)},\,x_{(m',\,k')})+d(x_{(m',\,k')},\,y)
\le [R+(2R)^{1+C_5/2}]2^{-m/2}.
\end{eqnarray*}
 This implies that
\begin{eqnarray*}
 B(x_{(m',\,k')},\,R2^{-m'/2})
 \subset B(x_{(m,\,k)},\,[R+(2R)^{1+C_5/2}]2^{-m/2})
 \subset B(x_{(m',\,k')},\,R^{4+2C_5}2^{-m'/2}).
\end{eqnarray*}
Set $\wz C\equiv 4+C_5/2$.
Observe that by the doubling property of $\mu$, we have
\begin{eqnarray*}
\mu\lf(B\lf(x_{(m',\,k')},\,\frac152^{-m'/2}\r)\r)
&&\ge \frac1{C_2}
\lf(\frac{5R^{\wz C}2^{-m/2}}{2^{-m'/2}}\r)^{-n}
\mu(B(x_{(m',\,k')},\,R^{\wz C}2^{-m'/2}))\\
&&\ge R^{C'} \mu(B(x_{(m,\,k)},\,[R+(2R)^{1+C_5/2}]2^{-m/2}))
\end{eqnarray*}
for some positive constant $C'$ independent of $R,\,m,\,m'$ and $k$ .
Thus, for fixed $m'$, by the disjointness of
$\{B(x_{(m',\,k')},\,\frac152^{-m'/2})\}_{k'}$,
we have
$$\sharp\{k':\
(B(x_{(m,\,k)},\,R2^{-m/2})\cap
B(x_{(m',\,k')},\,R2^{-m'/2}))\ne\emptyset\}\le R^C$$
for some positive constant $C$ independent of $R,\,m,\,m'$ and $k$.
This together with \eqref{2.7} implies that
$$\sharp\{(m',\,k'):\
(B(x_{(m,\,k)},\,R2^{-m/2})\cap
B(x_{(m',\,k')},\,R2^{-m'/2}))\ne\emptyset\}\le C_5R^C \log(2R),$$
which completes the proof of Lemma \ref{l2.3}.
\end{proof}

In what follows, we set
\begin{eqnarray}
\eta\in\ccc^1(\rr),&&\ \eta(t)\in[0,\,1]\  \mathrm{for\ all}\ t\in\rr, \label{2.8}\\
 &&\eta(t)=1\ \mathrm{when}\ |t|\le 1\ \mathrm{and} \ \eta(t)=0\ \mathrm{when}
 \ |t|\ge 2.\nonumber
\end{eqnarray}

\begin{lem}\label{l2.4}
There exist constant $C>0$ and functions
$\{\psi_{(m,\,k)}\}_{m\in\zz,\,k}$ such that

(i) $\supp \psi_{(m,\,k)}\subset B(x_{(m,\,k)},\,\rho(x_{m,\,k})/2)$ and
$0\le \psi_{(m,\,k)}(x)\le1$ for all $x\in\cx$;

(ii) $|\psi_{(m,\,k)}(x)-\psi_{(m,\,k)}(y)|\le Cd(x,\,y)
[\rho(x_{m,\,k})]^{-1}$ for all $x,\,y\in\cx$;

(iii) $\sum_{m\in\zz,\,k}\psi_{(m,\,k)}(x)=1$ for all $x\in\cx$.
\end{lem}

\begin{proof}
Let $\eta$ be as in \eqref{2.8}.
For each $m\in\zz$ and $k$, and all $x\in\cx$, set
$\eta_{(m,\,k)}(x)\equiv\eta(2^{m/2}d(x_{(m,\,k)},\,x))$ and
$$\psi_{(m,\,k)}(x)\equiv\frac{\eta_{(m,\,k)}(x)}
{\sum_{m'\in\zz,\,k'}\eta_{(m',\,k')}(x)}.$$
Then it is easy to show that
$\{\psi_{(m,\,k)}\}_{m\in\zz,\,k}$ satisfies (i) through (iii),
which completes the proof of Lemma \ref{l2.4}.
\end{proof}

In what follows, we always simply denote  $\psi_{(m,\,k)}$ and
$B(x_{(m,\,k)},\,\rho(x_{(m,\,k)})/2)$, respectively, by $\psi_\az$
and $B_\az$.

\subsection{Hardy spaces $H^1(\cx)$ and their localized variants\label{s2.2}}

The following notion of
approximations of the identity
on RD-spaces was first introduced in \cite{hmy2},
whose existence was given in Theorem 2.6 of \cite{hmy2}.
Recall that $V_r(x)\equiv\mu(B(x,\,r))$
and $V(x,\,y)\equiv\mu(B(x,\,d(x,\,y)))$ for all $x,\,y\in\cx$ and $r>0$.

\begin{defn}\rm\label{d2.3}
Let $\ez_1\in(0,\,1]$, $\ez_2>0$ and $\ez_3>0$. A sequence
$\{S_k\}_{k\in\zz}$ of bounded
linear integral operators on $L^2(\cx)$ is called an
approximation of the identity of order $(\ez_1,\,\ez_2,\,\ez_3)$
(for short, $(\ez_1,\,\ez_2,\,\ez_3)$-$\ati$), if there
exists a positive constant $C_6$ such that for all $k\in\zz$ and
$x,\, x',\, y,\,y'\in\cx$, $S_k(x,y)$, the integral kernel of
$S_k$, is a measurable function from $\cx\times\cx$ into $\cc$
satisfying
\begin{enumerate}
\vspace{-0.3cm}
\item[(i)] $|S_k(x,y)|\le C_6\frac 1{V_{2^{-k}}(x)+V(x,\,y)}
[\frac{2^{-k}}{2^{-k}+d(x,\,y)}]^{\ez_2};$
\vspace{-0.3cm}
\item[(ii)] $|S_k(x,y)-S_k(x',y)|\le C_6 [\frac{d(x,\,x')}{2^{-k}+d(x,\,y)}]^\ezl
\frac 1{V_{2^{-k}}(x)+V(x,\,y)}[\frac{2^{-k}}{2^{-k}+d(x,\,y)}]^{\ez_2}$
for $d(x,x')\le[2^{-k}+d(x,\,y)]/2;$
\vspace{-0.3cm}
\item[(iii)] Property (ii) also holds with $x$ and $y$
interchanged;
\vspace{-0.3cm}
\item[(iv)] $|[S_k(x,y)-S_k(x,y')]-[S_k(x',y)-S_k(x',y')]|\le C_6
[\frac{d(x,\,x')}{2^{-k}+d(x,\,y)}]^\ezl
[\frac{d(y,\,y')}{2^{-k}+d(x,\,y)}]^\ezl$
\newline
$\times\frac 1{V_{2^{-k}}(x)+V(x,\,y)}
[\frac{2^{-k}}{2^{-k}+d(x,\,y)}]^{\ez_3}$ for
$d(x,x')\le[2^{-k}+d(x,\,y)]/3$
and $d(y,y')\le[2^{-k}+d(x,\,y)]/3;$
\vspace{-0.3cm}
\item[(v)] $\int_\cx S_k(x,z)\,d\mu(z)=1=\int_\cx S_k(z,y)\,d\mu(z)$.
\end{enumerate}
\end{defn}

\begin{rem}\label{r2.2}\rm
(i) In \cite{hmy2}, for any $N>0$, it was proved that
there exists $(1,\,N,\,N)$-$\ati$ $\{S_k\}_{k\in\zz}$
with bounded support in the sense that
$S_k(x,y)=0$ when $d(x,\,y)>\wz C2^{-k}$,
where $\wz C$ is a fixed positive constant independent of $k$.
In this case, $\{S_k\}_{k\in\zz}$ is called a $1$-$\ati$
with bounded support; see \cite{hmy2}.

(ii) If a sequence $\{\wz S_t\}_{t>0}$ of bounded
linear integral operators on $L^2(\cx)$ satisfies (i) through (v)
of Definition \ref{d2.3}
with  $2^{-k}$ replaced by $t$, then
$\{\wz S_t\}_{t>0}$ is called a continuous
approximation of the identity of order $(\ez_1,\,\ez_2,\,\ez_3)$
(for short, continuous $(\ez_1,\,\ez_2,\,\ez_3)$-$\ati$).
For example,
if $\{S_k\}_{k\in\zz}$ is an $(\ez_1,\,\ez_2,\,\ez_3)$-$\ati$ and if we set
$\wz S_t(x,\,y)\equiv S_k(x,\,y)$ for $t\in(2^{-k-1},\,2^{-k}]$ with $k\in\zz$,
then  $\{\wz S_t\}_{t>0}$ is a continuous $(\ez_1,\,\ez_2,\,\ez_3)$-$\ati$.

(iii) If $S_k$ (resp. $\wz S_t$) satisfies (i), (ii), (iii) and (v) of
Definition \ref{d2.3}, then $S_kS_k $ (resp. $\wz S_t\wz S_t$)
satisfies the conditions (i) through (v) of
Definition \ref{d2.3}; see \cite{hmy06}.
\end{rem}

The following spaces of test functions play an important role in the theory of
function spaces on space of homogeneous type; see \cite{hmy06, hmy2}.

\begin{defn}\rm\label{d2.4}
Let $x\in\cx$, $r>0$, $\bz\in(0,\,1]$ and $\gz>0$. A function $f$
on $\cx$ is said to belong to the space of test functions, $\cg(x,\,r,\,\bz,\,\gz)$,
if there exists a positive constant $C_f$ such that
\begin{enumerate}
    \vspace{-0.3cm}
    \item[(i)] $|f(y)|\le C_f\frac 1{V_r(x)+V(x,\,y)}
    [\frac r{r+d(x,\,y)}]^\gz$ for all $y\in\cx$;
    \vspace{-0.3cm}
    \item[(ii)] $|f(y)-f(y')|\le C_f[\frac {d(y,\,y')}{r+d(x,\,y)}]^\bz
    \frac 1{V_r(x)+V(x,\,y)}[\frac r{r+d(x,\,y)}]^\gz$
    for all $y,\, y'\in \cx $ satisfying that $d(y,\,y')\le [r+d(x,\,y)]/2$.
    \vspace{-0.3cm}
\end{enumerate}
Moreover, for any $f\in \cg(x,\,r,\,\bz,\,\gz)$,   its norm is defined by
$$\|f\|_{\cg(x,\,r,\,\bz,\,\gz)}\equiv\inf\lf\{C_f:\, (i)\  and \ (ii)\
 hold \r\}.$$
\end{defn}

It is easy to see that $\cg(x,\,r,\,\bz,\,\gz)$ is a Banach space.
Let $\ez\in(0,\,1]$ and $\bz,\,\gz\in(0,\,\ez]$.
For applications, we further define the space $\cg_0^\ez(x,\,r,\,\bz,\,\gz)$ to be
the completion of the set $\cg(x,\,r,\,\ez,\,\ez)$ in $\cg(x,\,r,\,\bz,\,\gz)$.
For $f\in\cg_0^\ez(x,\,r,\,\bz,\,\gz)$,
define $\|f\|_{\cg_0^\ez(x,\,r,\,\bz,\,\gz)}\equiv\|f\|_{\cg(x,\,r,\,\bz,\,\gz)}$.
Let
$\lf(\cg_0^\ez(x,\,r,\,\bz,\,\gz)\r)'$
be the set of all continuous linear functionals on
$\cg_0^\ez(x,\,r,\,\bz,\,\gz)$,
and as usual, endow $\lf(\cg_0^\ez(x,\,r,\,\bz,\,\gz)\r)'$
with the weak $\ast$-topology.
Throughout the whole
paper, we fix $x_1\in\cx$ and write
$\cg(\bz,\,\gz)\equiv\cg(x_1,\,1,\,\bz,\,\gz)$,
and $\lf(\cg_0^\ez(\bz,\,\gz)\r)'\equiv\lf(\cg_0^\ez(x_1,\,1,\,\bz,\,\gz)\r)'$.

The following results concerning
approximations of the identity
were proved in \cite[Proposition 2.7]{hmy2}
and Lemma 3.5 through Lemma 3.7 and Proposition 3.8 in \cite{gly1}.

\begin{lem}\label{l2.5}
Let $\ez_1\in(0,\,1]$,  $\ez_2,\,\ez_3>0$, $\ez\in(0,\,\ezl\wedge\ezz)$ and
$\{S_k\}_{k\in\zz}$ be an $(\ez_1,\,\ez_2,\,\ez_3)$-$\ati$.

(i) If $p\in[1,\fz]$, then $\{S_k\}_{k\in\zz}$ is
a sequence of bounded operators on $L^p(\cx)$ uniformly in $k$.
Moreover, for any $p\in[1,\fz)$ and $f\in L^p(\cx)$,
$\|S_k(f)-f\|_{L^p(\cx)}\to0$ as $k\to\fz$.

(ii) If $\bz,\,\gz\in(0,\,\ez)$, then $\{S_k\}_{k\in\zz}$
is a sequence of bounded operators on $\cg^\ez_0(\bz,\,\gz)$
uniformly in $k$. Moreover, for any $f\in\cg^\ez_0(\bz,\,\gz)$,
$\|S_k(f)-f\|_{\cg^\ez_0(\bz,\,\gz)}\to0$ as $k\to\fz$;
for any $f\in(\cg^\ez_0(\bz,\,\gz))'$,
$S_k(f)$ converges to $f$ in the weak $\ast$-topology of
$(\cg^\ez_0(\bz,\,\gz))'$ as $k\to\fz$.
\end{lem}

\begin{defn}\rm\label{d2.5}
Let $\ez_1\in(0,\,1]$,\,$\ez_2,\,\ez_3>0$, $\ez\in(0,\,\ez_1\wedge\ez_2)$ and
$\{S_k\}_{k\in\zz}$ be an $(\ez_1,\,\ez_2,\,\ez_3)$-$\ati$.
Let $\rho$ be an admissible function.
For any $\bz,\,\gz\in(0,\,\ez)$,
$f\in (\cg_0^\ez(\bz,\,\gz))'$ and $x\in\cx$,
define

\noindent(i) the radial maximal function $S^+(f)$ by
$ S^+(f)(x)\equiv\sup_{k\in\zz} |S_k(f)(x)|;$

\noindent(ii) the radial maximal function $ S^+_{\rho}(f)$
associated to $\rho$ by
$$S^+_{\rho}(f)(x)\equiv\sup_{\{k\in\zz,\,2^{-k}<\rho(x)\}} |S_k(f)(x)|;$$

\noindent(iii) the grand maximal function $G^{(\ez,\,\bz,\,\gz)}(f)$ by
$$ G^{(\ez,\,\bz,\,\gz)}(f)(x)\equiv\sup\lf\{|\langle f,\,\varphi\rangle|:\,
\varphi\in\cg^\ez_0(\bz,\,\gz),\
\|\vz\|_{\cg(x,\,r,\,\bz,\,\gz)}\le1\ \mathrm{for\ some}\ r>0\r\};$$

\noindent(iv)  the grand maximal function
$G^{(\ez,\,\bz,\,\gz)}_{\rho}(f)$ associated to $\rho$ by
$$G^{(\ez,\,\bz,\,\gz)}_{\rho}(f)(x)\equiv\sup\lf\{|\langle f,\,\varphi\rangle|:\,
\varphi\in\cg^\ez_0(\bz,\,\gz),\
\|\vz\|_{\cg(x,\,r,\,\bz,\,\gz)}\le1\ \mathrm{for\ some}\ r\in (0,\rho(x))\r\}.$$
\end{defn}

When there exists no ambiguity, we simply write
$G^{(\ez,\,\bz,\,\gz)}(f)$ and $G^{(\ez,\,\bz,\,\gz)}_{\rho}(f)$
as $G(f)$ and $G_{\rho}(f)$, respectively.
Notice that $\|S_k(x,\,\cdot)\|_{\cg^\ez_0(x,\,2^{-k},\,\bz,\,\gz)}\le C_6$ for all
$x\in\cx$ and $\bz,\,\gz\in(0,\,\ez)$.
It is easy to see that for all $x\in\cx$,
$S^+_{\rho}(f)(x)\le S^+(f)(x)\le C_6G(f)(x)$
and
\begin{equation}\label{2.9}
S^+_{\rho}(f)(x)\le C_6G_{\rho}(f)(x)\le C_6G(f)(x).
\end{equation}

\begin{defn}\rm\label{d2.6}
Let $\ez\in(0,\,1)$, $\bz,\,\gz\in(0,\,\ez)$ and $\rho$ be an admissible function.

(i) The Hardy space $H^1(\cx)$ is defined by
$$H^1(\cx)\equiv\lf\{f\in(\cg^\ez_0(\bz,\,\gz))': \
\|f\|_{H^1(\cx)}\equiv\|G(f)\|_{L^1(\cx)}<\fz\r\}.$$

(ii) The Hardy space $H^1_\rho (\cx)$ associated to
$\rho$ is defined by
$$H^1_\rho (\cx)\equiv\lf\{f\in(\cg^\ez_0(\bz,\,\gz))': \
\|f\|_{H^1_\rho (\cx)}
\equiv\|G_{\rho}(f)\|_{L^1(\cx)}<\fz\r\}.$$
\end{defn}

\begin{defn}\rm\label{d2.7}
Let $q\in(1,\,\fz]$.

\noindent (i) A measurable function $a$ is called a $(1,\,q)$-atom
associated to the ball $B(x,\,r)$ if

(A1) $\supp a\subset B(x,\,r)$ for some $x\in\cx$ and $r>0$,

(A2) $\|a\|_{L^q(\cx)}\le [\mu(B(x,\,r))]^{1/q-1}$,

(A3) $\int_\cx a(x)\,d\mu(x)=0$.

\noindent (ii) A measurable function $a$ is called a
$(1,\,q)_{\rho}$-atom
associated to the ball $B(x,\,r)$ if $r< \rho(x)$ and
$a$ satisfies (A1) and (A2),
and when $r<\rho(x)/4$, $a$ also satisfies (A3).
\end{defn}

\begin{defn}\rm\label{d2.8}
Let  $\ez\in(0,\,1)$, $\bz,\,\gz\in(0,\,\ez)$ and $q\in(1,\,\fz]$.

(i) The space $H^{1,\,q}(\cx)$  is defined to be the
set of all $f=\sum_{j\in\nn}\lz_ja_j$ in $(\cg^\ez_0(\bz,\,\gz))'$,
where $\{a_j\}_{j\in\nn}$ are $(1,\,q)$-atoms and $\{\lz_j\}_{j\in\nn}\subset\cc$
such that $\sum_{j\in\nn}|\lz_j|<\fz$. For any $f\in H^{1,\,q}(\cx)$,
define $\|f\|_{H^{1,\,q}(\cx)}\equiv\inf\{\sum_{j\in\nn}|\lz_j|\}$, where
the infimum is taken over all the above decompositions of $f$.

(ii) The space $H^{1,\,q}_\fin(\cx)$ is defined to be the
set of all $f=\sum_{j=1}^N\lz_ja_j$,
where $N\in\nn$, $\{\lz_j\}_{j\in\nn}\subset\cc$,
and $\{a_j\}_{j=1}^N$ are $(1,\,q)$-atoms
when $q<\fz$ or continuous $(1,\,\fz)$-atoms when $q=\fz$.
For any $f\in H^{1,\,q}_\fin(\cx)$,
define $\|f\|_{H^{1,\,q}_\fin(\cx)}
\equiv\inf\{\sum_{j=1}^N|\lz_j|\}$, where
the infimum is taken over all the above finite decompositions of $f$.

(iii) The space $H^{1,\,q}_{\rho}(\cx)$ is defined
as in (i) with $(1,\,q)$-atoms replaced by $(1,\,q)_{\rho}$-atoms.

(iv) The space $H^{1,\,q}_{\rho,\,\fin}(\cx)$ is defined
as in (ii) with
$(1,\,q)$-atoms replaced by $(1,\,q)_{\rho}$-atoms.
\end{defn}

The atomic Hardy spaces $H^{1,\,q}(\cx)$ were originally introduced in \cite{cw77}.
Moreover, in \cite{gly1,gly2}, the following results were established.

\begin{thm}\label{t2.1}
(i)  Let $\ez\in(0,\,1)$ and $\bz,\,\gz\in(0,\,\ez)$. Then
the following are equivalent:
(\rm{a}) $f\in H^1(\cx)$;  (\rm{b}) $f\in(\cg^\ez_0(\bz,\,\gz))'$ and
  $\|S^+(f)\|_{L^1(\cx)}<\fz$; (\rm{c})  $f\in H^{1,\,q}(\cx)$
with $q\in(1,\,\fz]$. Moreover, for any fixed $q\in(1,\,\fz]$ and all $f\in H^1(\cx)$,
$$\|f\|_{H^1(\cx)}\sim\|S^+(f)\|_{L^1(\cx)}\sim \|f\|_{H^{1,\,q}(\cx)}.$$

(ii) If $q\in(1,\,\fz]$, then for all $f\in H^{1,\,q}_\fin(\cx)$,
 $\|f\|_{H^{1,\,q}_\fin(\cx)}\sim\|f\|_{H^1(\cx)}$.
\end{thm}

We finally point out that by Definitions \ref{d2.6} and \ref{d2.8} above, the
spaces $H^1(\cx)$, $H^1_\rho(\cx)$, $H^{1,\,q}(\cx)$ and $H^{1,\,q}_\rho(\cx)$
seem to depend on the choices of $\ez\in(0,\,1)$ and $\bz,\,\gz\in(0,\,\ez)$.
However, in Remark \ref{r3.1} below, we show that all these spaces
are independent of the choices of $\ez\in(0,\,1)$
and $\bz,\,\gz\in(0,\,\ez)$, which is the reason why we omit the
parameters $\ez,$ $\bz$ and $\gz$ when mentioning them.

\section{Atomic decomposition characterizations of $H^1_\rho (\cx)$}\label{s3}

We begin with the following relations concerning the Hardy spaces in
Definition \ref{d2.6} and Definition \ref{d2.8}
and the Lebesgue space $L^1(\cx)$. Recall that the symbol $\subset$
means continuous embedding.

\begin{lem}\label{l3.1}
Let $q\in(1,\,\fz]$. Then

(i) $H^{1,\,q}(\cx)\subset H_{\rho}^{1,\,q}(\cx)\subset H^1_\rho (\cx)
\subset L^1(\cx)$;

(ii) $H_{\rho}^{1,\,q}(\cx)=H_{\rho}^{1,\,\fz}(\cx)$
with equivalent norms independent of $\rho$.
\end{lem}

\begin{proof}
To see $H^{1,\,q}(\cx)\subset H_{\rho}^{1,\,q}(\cx)$,
we only need to prove that if $a$ is a $(1,\,q)$-atom supported in
$B(x_0,\,r_0)$ with $ r_0\ge\rho(x_0)$, then $a\in H^{1,\,q}_{\rho}(\cx)$.
In fact, by Lemma \ref{l2.4}, we write
$a\equiv\sum_{\az}\psi_\az a$ pointwise. Recall
that $\{\psi_\az\}_\az$ is as in Lemma \ref{l2.4}.
From Lemma \ref{l2.3}, it is easy to see that $a=\sum_\az\psi_\az a$ holds in
$(\cg_0^\ez(\bz,\,\gz))'$ with $\ez,\,\bz,\,\gz$ as in Definition \ref{d2.8}.
Let
$$\lz_\az\equiv[\mu(B_\az)]^{1-1/q}\|\psi_\az a\|_{L^q(\cx)}.$$
If $\lz_\az=0$, set $a_\az\equiv0$; if $\lz_\az\ne0$,
set $a_\az\equiv(\lz_\az)^{-1}\psi_\az a$. Notice that
by Lemma \ref{l2.1} (i), if $(B_\az\cap B(x_0,\,r_0))\ne\emptyset$,
then $B_\az\subset B(x_0,\,Cr_0)$.
Thus, $a_\az$ is a $(1,\,q)_{\rho}$-atom associated to the ball
$B_\az\equiv B(x_\az,\,\rho(x_\az)/2)$, and by the
H\"older inequality and Lemma \ref{l2.3}, we have
$$\sum_\az\lz_\az\ls
\|a\|_{L^q(\cx)}\lf[\sum_{\az}\mu(B_\az)\r]^{1/q'}
\ls[\mu(B(x_0,\,r_0))]^{1/q-1}[\mu(B(x_0,\,Cr_0))]^{1/q'}\ls1.$$
This means that $a\in H^{1,\,q}_{\rho}(\cx)$ and
$\|a\|_{H^{1,\,q}_{\rho}(\cx)}\ls1.$ Thus,
$H^{1,\,q}(\cx)\subset H^{1,\,q}_{\rho}(\cx)$.

To prove $H_{\rho}^{1,\,q}(\cx)\subset H^1_\rho (\cx)$,
by the definition of $G_{\rho}$,
it suffices to prove
that for all $(1,\,q)_{\rho}$-atoms $a$,
$\|G_{\rho}(a)\|_{L^1(\cx)}\ls1$.
In fact, if $a$ is a $(1,\,q)$-atom, then it is known that
$\|G_{\rho}(a)\|_{L^1(\cx)}\ls\|G(a)\|_{L^1(\cx)}\ls1$.
If $\int_\cx a(x)\,d\mu(x)\ne0$, assuming that $\supp a\subset B(x_0,\,r_0)$, then
$\rho(x_0)/4\le r_0<\rho(x_0)$.
Since $G_{\rho}(a)\ls M(a)$,
where $M$ is the Hardy-Littlewood maximal operator on $\cx$,
then $G_{\rho}$ is bounded on $L^q(\cx)$ (see \cite{cw71,cw77}); then by the
H\"older inequality, we have
$$\|G_{\rho}(a)\|_{L^1(B(x_0,\,4\rho(x_0))}\ls
[\mu(B(x_0,\,\rho(x_0)))]^{1/q'}\|a\|_{L^q(B(x_0,\,4\rho(x_0)))}\ls1,$$
where and in what follows, for any set $E\subset\cx$, we write
$$\|f\|_{L^q(E)}\equiv \lf\{\int_E|f(x)|^q\,d\mu(x)\r\}^{1/q}.$$
For $x\notin B(x_0,\,4\rho(x_0))$, since for any
$\psi\in \cg_0^\ez(\bz,\,\gz)$ with $\|\psi\|_{\cg^\ez_0(x,\,r,\,\bz,\,\gz)}\le1$ and
$r<\rho(x)$, we have
$$\lf|\int_\cx a(y)\psi(y)\,d\mu(y)\r|\ls \int_\cx |a(y)|\frac1{V(x,\,y)}
\lf[\frac{\rho(x)}{d(x,\,y)}\r]^\gz\,d\mu(y)
\ls \frac 1{V(x,\,x_0)}
\lf[\frac{\rho(x)}{d(x,\,x_0)}\r]^\gz,$$
which together with \eqref{2.2}
implies that
$$G_{\rho}(a)(x)\ls \frac 1{V(x,\,x_0)}
\lf[\frac{\rho(x)}{d(x,\,x_0)}\r]^\gz\ls \frac 1{V(x,\,x_0)}
\lf[\frac{\rho(x_0)}{d(x,\,x_0)}\r]^{\gz/(1+k_0)}.$$
Thus, $\|G_{\rho}(a)\|_{L^1(\cx\setminus B(x_0,\,4\rho(x_0)))}\ls1$ and, therefore,
$\|G_{\rho}(a)\|_{L^1(\cx)}\ls1.$ This shows that
$H^{1,\,q}_{\rho}(\cx)\subset H^1_\rho (\cx)$.

To prove $H^1_\rho (\cx)
\subset L^1(\cx)$, assume that $f\in H^1_\rho (\cx)$.
By \eqref{2.9} and Definition \ref{d2.8},
we have that $\|S^+_{\rho}(f)\|_{L^1(\cx)}\ls
 \|f\|_{H^1(\cx)}$, which means that
$\{S_k(f)\chi_{\{2^{-k}<\rho(\cdot)\}}\}_{k\in\zz}$ is a bounded set in $L^1(\cx)$.
Thus, by the proof of \cite[Theorem III. C. 12]{w91},
$\{S_k(f)\chi_{\{2^{-k}<\rho(\cdot)\}}\}_{k\in\zz}$ is
relatively weakly compact in $L^1(\cx)$. This together with
the Eberlein-$\rm\breve{S}$mulian theorem (see, for example,
\cite[Theorem II. C. 3]{w91}) implies that there exist a subsequence
$\{S_{k_j}(f)\chi_{\{2^{-k_j}<\rho(\cdot)\}}\}_{j\in\nn}$ of
$\{S_k(f)\chi_{\{2^{-k}<\rho(\cdot)\}}\}_{k\in\zz}$ and a measurable
function $g\in L^1(\cx)$ such that
$\{S_{k_j}(f)\chi_{\{2^{-k_j}<\rho(\cdot)\}}\}_{j\in\nn}$ weakly converges to $g$
in $L^1(\cx)$ and hence in $(\cg_0^\ez(\bz,\gz))'$ with
$\ez,$ $\bz$ and $\gz$ as in Definition \ref{d2.6}. From this, it is easy to follow that
$$\|g\|_{L^1(\cx)}\le\|S^+_{\rho}(f)\|_{L^1(\cx)}\ls
 \|f\|_{H^1_\rho (\cx)}.$$
Denote by $\cg_{0,b}^\ez(\bz,\gz)$ the set of functions
in $\cg_0^\ez(\bz,\gz)$ with bounded support.
For any $\psi\in\cg_{0,b}^\ez(\bz,\gz)$,
assume that $\supp \psi\subset B(x_1,\,r)$.
By Lemma \ref{l2.1} (iii), there exists $j_1\in\nn$ such that
$2^{-j_1}\le \inf_{y\in B(x_1,\,r)}\rho(y)$.
Therefore, by Lemma \ref{l2.5} (ii),
$$\langle g,\psi\rangle=\lim_{j\to\fz}\langle S_{k_j}(f)\chi_{2^{-k_j}
<\rho(\cdot)},\,\psi\rangle
= \lim_{j\to\fz}\langle S_{k_j}(f),\,\psi\rangle=\langle f,\,\psi\rangle.$$
On the other hand, it is easy to show that $\cg_{0,b}^\ez(\bz,\gz)$
is dense in $\cg_0^\ez(\bz,\gz)$, which further implies
that $f=g$ in $(\cg_0^\ez(\bz,\gz))'$. In this sense,
we say $f\in L^1(\cx)$. Thus (i) holds.

To prove (ii), by Definition \ref{d2.8}, obviously,
$H_{\rho}^{1,\,\fz}(\cx)\subset H_{\rho}^{1,\,q}(\cx)$
and the inclusion is continuous.
Conversely, it suffices to prove that if $a$ is any $(1,\,q)_{\rho}$-atom
supported in $B\equiv B(x_0,\,r_0)$, then
$a\in H_{\rho}^{1,\,\fz}(\cx)$ and
$\|a\|_{H^{1,\,\fz}_{\rho}(\cx)}\ls1$.
In fact, if $a$ is a $(1,\,q)$-atom, then by (i) of this
lemma and Theorem \ref{t2.1} (i),
$a\in H^{1,\,\,q}(\cx)=H^{1,\,\fz}(\cx)\subset H^{1,\,\fz}_{\rho}(\cx)$,
and $\|a\|_{H^{1,\,\fz}_{\rho}(\cx)}\ls \|a\|_{H^{1,\,\fz}(\cx)}\ls1$.
If $\int_\cx a(x)\,d\mu(x)\ne0$, then $[a-a_B\chi_{B}]/2$ is a $(1,\,q)$-atom, where
$a_B\equiv\frac1{\mu(B)}\int_Ba(y)\,d\mu(y).$
Thus $a-a_B\chi_{B}\in H^{1,\,\fz}_{\rho}(\cx)$ with
$\|a-a_B\chi_{B}\|_{H^{1,\,\fz}_{\rho}(\cx)}\ls1$.
Since $|a_B|\le[\mu(B)]^{-1}$ and $\rho(x_0)/4\le r<\rho(x_0)$, we know that
$a_B\chi_B$ is a $(1,\,\fz)_{\rho}$-atom. Thus, $a\in H^{1,\,\fz}_{\rho}(\cx)$ and
$\|a\|_{H^{1,\,\fz}_{\rho}(\cx)}\ls1$.
This gives (ii), which completes the proof of Lemma \ref{l3.1}.
\end{proof}

\begin{rem}\label{r3.1}\rm
(i) Observe that in Definition \ref{d2.8}, if $\sum_{j\in\nn}\lz_ja_j$
converges to $f$ in $(\cg^\ez_0(\bz,\,\gz))'$, where
$\{\lz_j\}_{j\in\nn}\subset\cc$ such that $\sum_{j\in\nn}|\lz_j|<\fz$
and $\{a_j\}_{j\in\nn}$ are $(1,\,q)$-atoms or $(1,\,q)_\rho$-atoms,
then by Lemma \ref{l3.1}, $\sum_{j\in\nn}\lz_ja_j$ also converges
to $\wz f$ in $L^1(\cx)$,
where $f$ and $\wz f$ coincide in $(\cg^\ez_0(\bz,\,\gz))'$.
By identifying $f$ with $\wz f$, if we replace the distribution space
$(\cg^\ez_0(\bz,\,\gz))'$ in Definition \ref{d2.8} with $L^1(\cx)$,
we still obtain the same atomic Hardy spaces, which further implies that
the spaces $H^{1,\,q}(\cx)$ and $H^{1,\,q}_{\rho}(\cx)$
are independent of the choices of $\ez\in(0,\,1)$ and $\bz,\,\gz\in(0,\,\ez)$.
This is the reason why we omit the parameters $\ez,\,\bz,\,\gz$,
when we mention the atomic Hardy spaces $H^{1,\,q}(\cx)$ and
$H^{1,\,q}_{\rho}(\cx)$.

(ii) Notice that Theorem \ref{t2.1} shows that $H^{1,\,q}(\cx)=H^1(\cx)$.
By (i) of this remark, we know that the spaces $H^{1}(\cx)$,
whose definitions seem to depend on the choices of $\ez\in(0,\,1)$
and $\bz,\,\gz\in(0,\,\ez)$, are actually equivalent.
Thus, the space $H^{1}(\cx)$ is independent of the choices
of $\ez\in(0,\,1)$ and $\bz,\,\gz\in(0,\,\ez)$.

(iii) Similarly, if we can prove $H^{1,\,q}_\rho (\cx)=H^1_\rho(\cx)$,
then the space  $H^{1}_\rho(\cx)$ is also independent of the choices
of $\ez\in(0,\,1)$ and $\bz,\,\gz\in(0,\,\ez)$. We do prove
$H^{1,\,q}_\rho (\cx)=H^1_\rho(\cx)$ in Theorem \ref{t3.2}
below without using the fact that the space $H^{1}_\rho(\cx)$
is independence of the choices of $\ez\in(0,\,1)$ and $\bz,\,\gz\in(0,\,\ez)$.
\end{rem}

To obtain an atomic decomposition characterization of $H^1_\rho (\cx)$,
we first construct a kernel function on $\cx\times\cx$
by subtly developing some ideas of Coifman presented in \cite{djs}
(see also \cite{hmy2}).

\begin{prop}\label{p3.1}
Let $\rho$ be an admissible function.
There exist a nonnegative function
$K_\rho$ on $\cx\times\cx$ and a positive constant $C$ such that

(i) $K_\rho(x,\,y)=0$ if $d(x,\,y)>C[\rho(x)\wedge\rho(y)]$
and $K_\rho(x,\,y)\le C\frac1{V_{\rho(x)}(x)+V_{\rho(y)}(y)}$
for all $x,\,y\in\cx$;

(ii) $K_\rho(x,\,y)=K_\rho(y,\,x)$ for all $x,\,y\in\cx$;

(iii) $|K_\rho(x,\,y)-K_\rho(x,\,y')|\le C
\frac{d(y,\,y')}{\rho(x)}\frac1{V_{\rho(x)}(x)+V_{\rho(y)}(y)}$
for all $x,\, y,\,y'\in\cx$ with $d(y,\,y')\le[\rho(x)+d(x,\,y)]/2$;

(iv) $|[K_\rho(x,\,y)-K_\rho(x,\,y')]-[K_\rho(x',\,y)-K_\rho(x',\,y')]|\le C
\frac{d(x,\,x')}{\rho(x)}\frac{d(y,\,y')}{\rho(x)}
\frac1{V_{\rho(x)}(x)+V_{\rho(y)}(y)}$
for all $x,\,x', \, y,\,y'\in\cx$ with
$d(x,\,x')\le[\rho(y)+d(x,\,y)]/3$ and
$d(y,\,y')\le[\rho(x)+d(x,\,y)]/3$;

(v) $\int_\cx K_\rho(x,\,y)\,d\mu(x)=1$ for all $y\in\cx$.
\end{prop}

\begin{proof}
Let $\eta$ be as in \eqref{2.8} and $h(t)\equiv\eta(2t)$ for all
$t\in\rr$.
For any locally integrable function $f$ on $\cx$ and $u\in\cx$, define
\begin{equation}\label{3.1}
T_\rho (f)(u)
\equiv\int_\cx h\lf(\frac{d(u,\,w)}{\rho(w)}\r)f(w)\,d\mu(w)
\end{equation}
and
\begin{equation*}
\wz T_\rho (f)(u)
\equiv\int_\cx h\lf(\frac{d(u,\,w)}{\rho(u)}\r)f(w)\,d\mu(w).
\end{equation*}

We first claim that for any fixed constant $\wz C>1$,
if $d(x,\,u)\le \wz C\rho(x)$, then
\begin{equation}\label{3.2}
T_\rho (1)(u)\sim V_{\rho(x)}(x)\sim\wz T_\rho(1)(u),
\end{equation}
where the equivalent constants depend only on $\wz C$,\, $C_2$ and $C_3$.
To see that, notice that for any $u\in\cx$, by \eqref{2.8} and \eqref{2.1},
it is  easy to see that $V_{\rho(u)}(u)\sim \wz T_\rho (1)(u).$
Since $\rho(w)\sim\rho(u)$ for all
$w\in \cx$ with $d(x,\,u)<\rho(w)$ via Lemma \ref{l2.1} (i),
by \eqref{2.8} and the doubling property of $\mu$,
we also have
$T_\rho (1)(u)\sim V_{\rho(u)}(u)$.
Moreover, if $d(x,\,u)\le \wz C\rho(x)$, then by Lemma \ref{l2.1} (i),
we have $\rho(x)\sim\rho(u)$, which together with \eqref{2.1}
implies that
 $V_{\rho(x)}(x)\sim V_{\rho(x)}(u)\sim V_{\rho(u)}(u)$.
This shows the above claim \eqref{3.2}.

Moreover, for all $z\in \cx$, since $h(\frac{d(z,\,w)}{\rho(z)})\ne0$ implies that
$d(z,\,w)<\rho(z)$, by \eqref{3.2}, we have
\begin{equation}\label{3.3}
\wz T_\rho\lf(\frac1{T_\rho(1)}\r)(z)
=\int_\cx h\lf(\frac{d(z,\,w)}{\rho(z)}\r)\frac1{T_\rho(1)(w)}\,d\mu(w)
\sim\wz T_\rho(1)(z)\frac1{V_{\rho(z)}(1)}\sim1.
\end{equation}

For all $x,\,y\in\cx$, define
\begin{equation}\label{3.4}
K_\rho(x,\,y)\equiv\dfrac1{T_\rho(1)(x)}
\dint_\cx h\lf(\dfrac{d(x,\,z)}{\rho(z)}\r)
\dfrac1{\wz T_\rho\lf(\frac1{T_\rho(1)}\r)(z)}
h\lf(\dfrac{d(z,\,y)}{\rho(z)}\r)\,d\mu(z)\dfrac1{T_\rho(1)(y)}.
\end{equation}
Then $K_\rho$ satisfies (i) through (v) of Proposition \ref{p3.1}.

It is easy to see that $K_\rho(x,\,y)=K_\rho(y,\,x)$ for all $x,\,y\in\cx$, which
yields (ii).

To see (v), by \eqref{3.1}, we have
\begin{eqnarray*}
\int_\cx K_\rho(x,\,y)\,d\mu(x)
&&=\dint_\cx\lf\{\dint_\cx\dfrac1{T_\rho(1)(x)}
h\lf(\frac{d(x,\,z)}{\rho(z)}\r)\,d\mu(x)\r\}\\
&&\quad\quad \times
\dfrac1{\wz T_\rho\lf(\frac1{T_\rho(1)}\r)(z)}
h\lf(\dfrac{d(z,\,y)}{\rho(z)}\r)\,d\mu(z)
\dfrac1{T_\rho(1)(y)}\\
&&=\dint_\cx h\lf(\dfrac{d(z,\,y)}{\rho(z)}\r)\,d\mu(z)
\dfrac1{T_\rho(1)(y)}=1.
\end{eqnarray*}

To prove (i), notice that
$h(d(x,\,z)/\rho(z))h(d(z,\,y)/\rho(z))\ne0$ implies that
$d(x,\,z)\le \rho(z)$ and $d(z,\,y)\le \rho(z)$,
which together with Lemma \ref{l2.1} (i) yields that
$\rho(x)\sim \rho(z)\sim \rho(y)$ and $d(x,\,y)\ls [\rho(x)\wedge\rho(y)]$.
From these estimates, \eqref{3.3} and \eqref{3.4}, it follows that
$$K_\rho(x,\,y)\ls\frac1{V_{\rho(x)}(x)}
\int_{d(x,\,z)\ls \rho(x)}\,d\mu(z)\frac1{V_{\rho(y)}(y)}
\ls \frac1{V_{\rho(x)}(x)}\ls
\frac1{V_{\rho(x)}(x)+V_{\rho(y)}(y)}.$$
Moreover, by $\eqref{3.4}$, it is easy to see
that $K_\rho(x,\,y)\ne0$ if and only if
$$h(d(x,\,z)/\rho(x))h(d(z,\,y)/\rho(z))\ne0$$ for some $z$,
which implies that $d(x,\,y)\ls \rho(x)$.
Thus
$\supp K_\rho(x,\,\cdot)\subset B(x,\,C\rho(x))$, which
 establishes (i).

To obtain (iii), if $d(x,\,z)<\rho(z)$, $d(y,\,y')<[\rho(x)+d(x,\,y)]/2$
and $d(z,\,y)<\rho(z)$ or $d(z,\,y')<\rho(z)$, by Lemma \ref{l2.1} (i),
we then have
$d(y,\,y')\ls\rho(x)+\rho(z)\ls \rho(x)$,
$d(x,\,y)\ls\rho(x)$ and $\rho(x)\sim \rho(z)\sim \rho(y)\sim \rho(y')$.
By this and \eqref{3.3}, we have
\begin{eqnarray}\label{3.5}
&&\lf|\frac1{T_\rho(1)(y)}h\lf(\frac{d(z,\,y)}{\rho(z)}\r)-
\frac1{T_\rho(1)(y')}h\lf(\frac{d(z,\,y')}{\rho(z)}\r)\r| \\
&&\quad\le \frac1{T_\rho(1)(y)}\lf|h\lf(\frac{d(z,\,y)}{\rho(z)}\r)-
h\lf(\frac{d(z,\,y')}{\rho(z)}\r)\r|+
\frac{|T_\rho(1)(y)-T_\rho(1)(y')|}{T_\rho(1)(y)T_\rho(1)(y')}\nonumber\\
&&\quad\ls\frac1{V_{\rho(x)}(x)}\frac{d(y,\,y')}{\rho(x)}
+\frac1{[V_{\rho(x)}(x)]^2}\int_\cx\lf|h\lf(\frac{d(w,\,y)}{\rho(w)}\r)-
h\lf(\frac{d(w,\,y')}{\rho(w)}\r)\r|\,d\mu(w)\nonumber\\
&&\quad\ls\frac1{V_{\rho(x)}(x)}\frac{d(y,\,y')}{\rho(x)}
+\frac1{[V_{\rho(x)}(x)]^2}\int_{|w-y|<\rho(w)\ {\rm or}\ |w-y'|<\rho(w)}
\frac{d(y,\,y')}{\rho(w)}\,d\mu(w)\nonumber\\
&&\quad\ls\frac1{V_{\rho(x)}(x)}\frac{d(y,\,y')}{\rho(x)}
+\frac1{[V_{\rho(x)}(x)]^2}\frac{d(y,\,y')}{\rho(x)}
\int_{|w-y|\ls\rho(x)\ {\rm or}\ |w-y'|\ls \rho(x)}
d\mu(w)\nonumber\\
&&\quad\ls\frac1{V_{\rho(x)}(x)}\frac{d(y,\,y')}{\rho(x)}.\nonumber
\end{eqnarray}
This together with \eqref{3.3} implies that
for all $x,\,y,\,y'\in\cx$ with $d(y,\,y')<[\rho(x)+d(x,\,y)]/2$,
\begin{eqnarray*}
&&|K_\rho(x,\,y)-K_\rho(x,\,y')|\\
&&\quad\le\dfrac1{T_\rho(1)(x)}
\dint_\cx h\lf(\dfrac{d(x,\,z)}{\rho(z)}\r)
\dfrac1{\wz T_\rho\lf(\frac1{T_\rho(1)}\r)(z)}\\
&&\quad\quad\quad\times\lf|\dfrac1{T_\rho(1)(y)}
h\lf(\dfrac{d(z,\,y)}{\rho(z)}\r)-
\dfrac1{T_\rho(1)(y')}
h\lf(\dfrac{d(z,\,y')}{\rho(z)}\r)\r|\,d\mu(z)
\ls\frac1{V_{\rho(x)}(x)}\frac{d(y,\,y')}{\rho(x)},
\end{eqnarray*}
which shows (iii).

To prove (iv), for all $x,\,x', \, y,\,y'\in\cx$ with
$d(x,\,x')\le[\rho(y)+d(x,\,y)]/3$ and
$d(y,\,y')\le[\rho(x)+d(x,\,y)]/3$,
by \eqref{3.5}, we have
\begin{eqnarray*}
&&|[K_\rho(x,\,y)-K_\rho(x,\,y')]-[K_\rho(x',\,y)-K_\rho(x',\,y')]|\\
&&\quad\ls
\dint_\cx \lf|\dfrac1{T_\rho(1)(x)}h\lf(\dfrac{d(x,\,z)}{\rho(z)}\r)
-\dfrac1{T_\rho(1)(x')}h\lf(\dfrac{d(x',\,z)}{\rho(z)}\r)\r|\\
&&\quad\quad\quad\times\dfrac1{\wz T_\rho\lf(\frac1{T_\rho(1)}\r)(z)}
\lf|\dfrac1{T_\rho(1)(y)}
h\lf(\dfrac{d(z,\,y)}{\rho(z)}\r)-
\dfrac1{T_\rho(1)(y')}
h\lf(\dfrac{d(z,\,y')}{\rho(z)}\r)\r|\,d\mu(z)\\
&&\quad\ls\frac{d(x,\,x')}{\rho(y)}\frac{d(y,\,y')}{\rho(x)}
\frac1{V_{\rho(x)}(x)+V_{\rho(y)}(y)},
\end{eqnarray*}
which yields (iv) and hence, completes the proof of Proposition \ref{p3.1}.
\end{proof}

\begin{thm}\label{t3.1}
Let $\rho$ be an admissible function and
$K_\rho $ as in Proposition \ref{p3.1}.
If $f\in H^1_\rho (\cx)$, then $f-K_\rho (f)\in H^1(\cx)$,
where $$K_\rho(f)(x)=\int_\cx K_\rho(x,\,y)f(y)\,d\mu(y)$$
for all $x\in\cx$.
Moreover, there exists a positive constant $C$ such that
for all $f\in H^1_\rho (\cx)$,
$$\|f-K_\rho (f)\|_{H^1(\cx)}\le C\|f\|_{H^1_\rho (\cx)}.$$
\end{thm}

\begin{proof}
Let $\{S_k\}_{k\in\zz}$ be a $1$-$\ati$ with bounded support as
in Remark \ref{r2.2} (i) and  $f\in H^1_\rho (\cx)$. Then
\begin{eqnarray*}
\|f-K_\rho(f)\|_{H^1(\cx)}&&\le
\lf\|\sup_{\{k:\,2^{-k}<\rho(\cdot)\}}|S_k(f)(\cdot)|\r\|_{L^1(\cx)}
+\lf\|\sup_{\{k:\,2^{-k}<\rho(\cdot)\}}|S_k(K_\rho (f))(\cdot)|\r\|_{L^1(\cx)}\\
&&\quad+\lf\|\sup_{\{k:\,2^{-k}\ge\rho(\cdot)\}}
|S_k(f)(\cdot)-S_k(K_\rho (f))(\cdot)|\r\|_{L^1(\cx)}\\
&&\equiv I_1+I_2+I_3.
\end{eqnarray*}

By \eqref{2.9}, it is easy to see that
$$I_1\le\|S^+_{\rho}(f)\|_{L^1(\cx)}
\ls\|G_{\rho}(f)\|_{L^1(\cx)}\sim\|f\|_{H^1_\rho (\cx)}.$$
By Lemma \ref{l3.1}, $f\in L^1(\cx)$, and moreover, for all $x\in\cx$,
\begin{eqnarray*}
S_k(K_\rho (f))(x)
=\int_\cx\int_\cx S_k(x,\,z)K_\rho (z,\,y)f(y)\,d\mu(z)\,d\mu(y).
\end{eqnarray*}
For all $x,\,y\in\cx$, let
 $$\vz(x,\,y)\equiv\int_\cx S_k(x,\,z)K_\rho (z,\,y)\,d\mu(z).$$
To obtain that
$I_2\ls\|f\|_{H^1_\rho (\cx)}$, by the definition of
$H^1_\rho (\cx)$, it suffices to prove that
$\vz(x,\,\cdot)\in\cg(\ez,\ez)$ and
$\|\vz(x,\,\cdot)\|_{\cg^\ez_0(x,\,\ell,\,\ez,\,\ez)}\ls1$
for some $0<\ell<\rho(x)$ and $\ez\in(0,\,1)$ as in Definition \ref{d2.6}.
To this end, notice that by Proposition \ref{p3.1} (i),
$K_\rho (z,\,y)\ne0$ if and only if $d(y,\,z)\ls \rho(z)$.
Then if $2^{-k}<\rho(x)$, $d(x,\,z)\ls 2^{-k}$
and $d(z,\,y)<\rho(z)$,  by Lemma \ref{l2.1} (i), we have
\begin{equation}\label{3.6}\rho(z)\sim \rho(y)\sim \rho(x)\end{equation}
and $d(x,\,z)\ls \rho(x)$,
which further implies that $\supp \vz(x,\,\cdot)\subset B(x,\,C\rho(x))$.
Also, by Proposition \ref{p3.1} (i) and \eqref{3.6},
$|\vz(x,\,y)|\ls[V_{\rho(x)}(x)]^{-1}$. On the other hand,
if $2^{-k}<\rho(x)$, $d(x,\,z)\ls 2^{-k}$,
$d(x,\,y)\ls \rho(x)$ and $d(y',\,y)\le[\rho(x)+d(x,\,y)]/2$,
by Lemma \ref{l2.1} (i), we have $\rho(x)\sim \rho(z)$,
$d(x,\,z)\ls \rho(x)$ and
$d(y',\,y)\ls \rho(z)+d(z,\,y)$. Therefore,
$$|K_\rho (z,\,y)-K_\rho (z,\,y')|\ls
\frac1{V_{\rho(x)}(x)}\frac{d(y,\,y')}{\rho(x)}.$$
This implies that
$$|\vz(x,\,y)-\vz(x,\,y')|\le
\int_\cx |S_k(x,\,z)||K_\rho (z,\,y)-K_\rho (z,\,y')|\,d\mu(z)
\ls \frac1{V_{\rho(x)}(x)}\frac{d(y,\,y')}{\rho(x)}.$$
Thus, letting $\ell\equiv\rho(x)/2$, we then have $\vz(x,\,\cdot)\in\cg(\ez,\,\ez)$ and
$\|\vz\|_{\cg(x,\,\ell,\,\ez,\,\ez)}\ls1$, which is desired.

To estimate $I_3$, by Proposition \ref{p3.1} (v),
we write
\begin{eqnarray*}
&&S_k(f)(x)-S_k(K_\rho (f))(x)\\
&&\quad=\int_\cx\lf\{\int_\cx [S_k(x,\,z)-
 S_k(x,\,u)]K_\rho (u,\,z)\,d\mu(u)\r\} f(z)\,d\mu(z).
 \end{eqnarray*}
For all $x,\,z\in\cx$, let $$\psi(x,\,z)\equiv\int_\cx [S_k(x,\,z)-
 S_k(x,\,u)]K_\rho (u,\,z)\,d\mu(u).$$

Notice that by Proposition \ref{p3.1} (i), if $K_\rho (u,\,z)\ne0$,
then $d(u,\,z)< C'\rho(u)$, and
by Remark \ref{r2.2} (i), if $S_k(x,\,u)\ne0$, then
$d(x,\,u)< \wz C2^{-k}$.
We first claim that
there exists a positive constant $\wz C_0\ge 2\wz C$ such that if
$d(x,\,u)\ge \wz C_02^{-k}$ and $d(u,\,z)< C'\rho(u)$, then
$d(x,\,z)\ge \wz C2^{-k}$ and hence, $S_k(x,\,z)=0$.

In fact, by \eqref{2.2} and $\rho(x)\le 2^{-k}$, we have
\begin{eqnarray*}
d(x,\,z)&&\ge d(x,\,u)-d(u,\,z)\ge d(x,\,u)-C'\rho(u)\\
&&\ge d(x,\,u)-C'C_3[\rho(x)]^{1/(1+k_0)}[\rho(x)+d(x,\,u)]^{k_0/(1+k_0)}\\
&&\ge \lf\{1-C'C_3(\wz C_0)^{-1/(1+k_0)} (1+1/\wz C_0)^{k_0/(1+k_0)}\r\}d(x,\,u).
\end{eqnarray*}
Choosing $\wz C_0$ large enough such that $\wz C_0\ge 2\wz C$ and
$$\lf\{1-C'C_3(\wz C_0)^{-1/(1+k_0)} (1+1/\wz C_0)^{k_0/(1+k_0)}\r\}\ge 1/2,$$
we then have $d(x,\,z)\ge \wz C 2^{-k}$.

Thus $\psi(x,\,z)\ne0$ only when there exists an $u\in\cx$ such
that $d(x,\,u)< \wz C_02^{-k}$ and $d(u,\,z)< C'\rho(u)$.
Based on this observation, set
$$W_1\equiv\{u\in\cx:\
d(z,\,u)\le [2^{-k}+d(z,\,x)]/2,\,d(z,\,u)< C'\rho(u),\,
d(x,\,u)<\wz C_02^{-k}\}$$
and
$$W_2\equiv\{u\in\cx:\
d(z,\,u)>[2^{-k}+d(z,\,x)]/2,\,d(z,\,u)<C'\rho(u),\,
d(x,\,u)<\wz C_02^{-k}\}.$$
Then
\begin{eqnarray*}
|\psi(x,\,z)|\le\lf[\int_{W_1}+\int_{W_2}\r]
|S_k(x,\,z)-S_k(x,\,u)|K_\rho (u,\,z)\,d\mu(u)
\equiv I_4+I_5.
\end{eqnarray*}
If $u\in W_1$, then $\rho(u)\sim \rho(z)$ and
$$d(x,\,z)\le d(x,\,u)+d(z,\,u)<\wz C_02^{-k}+[2^{-k}+d(x,\,z)]/2
\le (\wz C_0+1/2)2^{-k}+d(x,\,z)/2,$$ which implies that
$d(x,\,z)\ls 2^{-k}$. Therefore,
noticing $2^{-k}\gs \rho(x)+ d(x,\,z)\sim \rho(z)+d(x,\,z)$ via Lemma \ref{l2.1} (iii),
by the regularity of $S_k$ and Proposition \ref{p3.1} (i), we have
$$I_4\ls2^k\rho(z)\frac{\mu(W_1)}
{[V_{2^{-k}}(x)+V_{2^{-k}}(z)]V_{\rho(z)}(z)}\ls\frac{\rho(z)}{d(x,\,z)+\rho(z)}\frac1
{V(x,\,z)+V_{\rho(z)}(z)}.$$
If $u\in W_2$, then $\rho(u)\sim \rho(z)$ and $d(x,\,z)\le 2d(z,\,u)\ls \rho(z)$,
which implies that $\rho(x)\sim \rho(z)$. Hence by Proposition \ref{p3.1} (i) and (v),
and Definition \ref{d2.3} (i), we obtain
$$I_5\ls\frac1
{V_{\rho(z)}(z)}\ls\frac{\rho(z)}{d(x,\,z)+\rho(z)}\frac1
{V(x,\,z)+V_{\rho(z)}(z)}.$$
Thus,
\begin{eqnarray*}
|\psi(x,\,z)|\ls\frac{\rho(z)}{d(x,\,z)+\rho(z)}\frac1
{V(x,\,z)+V_{\rho(z)}(z)},
\end{eqnarray*}
which implies that
\begin{eqnarray*}
&&|S_k(f)(x)-S_k(K_\rho (f))(x)|\ls\int_\cx
\frac{\rho(z)}{d(x,\,z)+\rho(z)}\frac1
{V(x,\,z)+V_{\rho(z)}(z)} |f(z)|\,d\mu(z).
 \end{eqnarray*}
Therefore, by Lemma \ref{l3.1} (i), we have
\begin{eqnarray*}
I_3
&&\le\int_\cx\int_\cx \frac{\rho(z)}{\rho(z)+d(x,\,z)}
\frac1{V_{\rho(z)}(z)+V(x,\,z)}|f(z)|\,d\mu(z)\,d\mu(x)
\ls \|f\|_{L^1(\cx)}\ls\|f\|_{H^1_\rho (\cx)},
 \end{eqnarray*}
which completes the proof of Theorem \ref{t3.1}.
\end{proof}

\begin{thm}\label{t3.2}
Let $\rho$ be an admissible function and $q\in(1,\,\fz]$.
Then

\noindent(i)
$H^1_\rho (\cx)
=H^{1,\,q}_{\rho}(\cx)$ with equivalent norms;

\noindent(ii) $\|\cdot\|_{H^1_\rho (\cx)}$
and $\|\cdot\|_{H^{1,\,q}_{\rho,\,\fin}(\cx)}$
are equivalent norms on $H^{1,\,q}_{\rho,\,\fin}(\cx)$.
\end{thm}

\begin{proof}
We first show (i). Let $f\in H^1_\rho (\cx)$. Then by Theorem \ref{t3.1},
$f-K_\rho(f)\in H^1(\cx)$. By the atomic decomposition of $H^1(\cx)$
in Theorem \ref{t2.1} (i), there exist
$\{\lz_j\}_{j\in\nn}\subset\cc$ and $(1,\,q)$-atoms $\{a_j\}_{j\in\nn}$
such that $f-K_\rho(f)=\sum_{j\in\nn}\lz_ja_j$ in $L^1(\cx)$, and
$$\sum_{j\in\nn}|\lz_j|\ls\|f-K_\rho(f)\|_{H^1(\cx)},$$
 which together
with Theorem \ref{t3.1} implies that
$\sum_{j\in\nn}|\lz_j|\ls\|f\|_{H^1_\rho (\cx)}.$

Now we decompose $K_\rho(f)$ as a summation of
$(1,\,q)_{\rho}$-atoms.
Let $\{\psi_\az\}_\az$ be as in Lemma \ref{l2.4} and
$\lz_\az\equiv[\mu(B_\az)]^{1-1/q}\|\psi_\az K_\rho (f)\|_{L^q(\cx)}.$
If $\lz_\az=0$, set $a_\az\equiv 0$; if $\lz_\az>0$, set
$a_\az\equiv(\lz_\az)^{-1}\psi_\az K_\rho (f).$
Obviously, $a_\az$ is a $(1,\,q)_{\rho}$-atom, and
\begin{equation}\label{3.7}
K_\rho (f)=\sum_{\az} \lz_\az a_\az.
\end{equation}
Recall that $\supp \psi_\az\subset B_\az\equiv B(x_\az, \rho(x_\az)/2)$.
Notice that by Lemma \ref{l2.1} (i), it is easy to
see that if $x\in B_\az$, then $V_{\rho(x)}(x)\sim\mu(B_\az)$.
This, together with $\supp K_\rho(x,\,\cdot)\subset B(x,\,C\rho(x))$,
Lemma \ref{l2.3} and Lemma \ref{l3.1} (i), yields that
\begin{eqnarray*}
\sum_{\az}\lz_\az&&\ls \sum_{\az}\mu(B_\az)
\|f\chi_{B(x_\az,\,C\rho(x_\az))}\|_{L^1(\cx)}
\sup_{x\in B_\az}\frac 1{V_{\rho(x)}(x)}
\ls\|f\|_{L^1(\cx)}\ls\|f\|_{H^1_\rho (\cx)}.
\end{eqnarray*}
On the other hand, assume that $\supp a_j\subset B(x_j,\,r_j)$. If $r_j<\rho(x_j)$,
then $a_j$ is a $(1,\,q)_{\rho}$-atom.
If $r_j\ge\rho(x_j)$, then by Lemma \ref{l2.3}, there exist finite many
$\az_j$ such that $(B_{\az_j}\cap B(x_j,\,r_j))\ne\emptyset,$ namely,
$a_j=\sum_{\az_j}\psi_{\az_j} a_j$ is a finite summation.
Let $$\lz_{j,\,\az_j}\equiv[\mu(B_{\az_j})]^{1-1/q}\|\psi_{\az_j} a_j\|_{L^q(\cx)}.$$
If $\lz_{j,\,{\az_j}}=0$, then set $a_{j,\,\az_j}\equiv0$;
if $\lz_{j,\,{\az_j}}>0$, then set
$a_{j,\,\az_j}\equiv(\lz_{j,\,\az_j})^{-1}\psi_{\az_j} a_j.$
Then $a_{j,\,{\az_j}}$ is a $(1,\,q)_{\rho}$-atom,
and by the H\"older inequality, Lemmas \ref{l2.1} and  \ref{l2.4}, we have
\begin{eqnarray*}
\sum_{{\az_j}}|\lz_{j,\,{\az_j}}|&&
\ls\lf\{\sum_{B_{\az_j}\cap B(x_j,\,r_j)\ne\emptyset}
\mu(B_{\az_j})\r\}^{1-1/q}
\lf\{\sum_{B_{\az_j}\cap B(x_j,\,r_j)\ne\emptyset}
\|\psi_{\az_j}a_j\|^q_{L^q(\cx)}\r\}^{1/q}\\
&&\ls[\mu(B(x_j,\,r_j))]^{1/q-1}\|a_j\|_{L^q(\cx)}\ls1.
\end{eqnarray*}
Thus,
\begin{equation}\label{3.8}
f=\sum_{r_j<\rho(x_j)}\lz_ja_j+\sum_{r_j\ge \rho(x_j)}\sum_{\az_j}\lz_j
\lz_{j,\,{\az_j}}a_{j,\,{\az_j}}+\sum_{\az}\lz_\az a_\az,
\end{equation}
where $\{a_j\}_{r_j<\rho(x_j)}$,
$\{a_{j,\,\az_j}\}_{r_j\ge \rho(x_j),\,\az_j}$
and $\{a_\az\}_{\az}$ are $(1,\,q)_{\rho}$-atoms and
$$\sum_{r_j<\rho(x_j)}|\lz_j|+
\sum_{r_j\ge \rho(x_j)}\sum_{\az_j}|\lz_j
\lz_{j,\,{\az_j}}|+
\sum_{\az}\lz_\az\ls\sum_{j}|\lz_j|+\sum_{\az}\lz_\az\ls
 \|f\|_{H^1_\rho (\cx)}.$$
That is, $f\in H^{1,\,q}_{\rho}(\cx)$ and
$\|f\|_{H^{1,\,q}_{\rho}(\cx)}\ls\|f\|_{H^1_\rho (\cx)}$,
which together with Lemma \ref{l3.1} implies (i).

To prove (ii), if $f\in H^{1,\,q}_{\rho,\,\fin}(\cx)$,
then $f\in L^q(\cx)$ with bounded support when $q<\fz$ and
$f\in\ccc_c(\cx)$ when $q=\fz$, and so is $K_\rho(f)$
by (i), (ii) and (iii) of Proposition \ref{p3.1}. By Proposition \ref{p3.1} (v),
$f-K_\rho(f)\in H^{1,\,q}_\fin(\cx)$.
From Theorem \ref{t2.1} (ii), it follows that there exist $N\in\nn$,
$\{\lz_j\}_{j=1}^N\subset\cc$,
and $(1,\,q)_{\rho}$-atoms $\{a_j\}_{j=1}^N$ when $q<\fz$
and continuous $(1,\,\fz)_{\rho}$-atoms $\{a_j\}_{j=1}^N$
when $q=\fz$ such that $f-K_\rho(f)=\sum_{j=1}^Na_j$ and
$\sum_{j=1}^N|\lz_j|\ls\|f-K_\rho(f)\|_{H^1(\cx)}$, which together with
Theorem \ref{t3.1} implies that
$\sum_{j=1}^N|\lz_j|\ls\|f\|_{H^1_\rho (\cx)}$.
Observe that by Lemma \ref{l2.3}, \eqref{3.7} in this case is a
finite summation of $(1,\,q)_{\rho}$-atoms when $q<\fz$
and continuous $(1,\,\fz)_{\rho}$-atoms when $q=\fz$. This together with
the above argument in the proof of (i) implies that \eqref{3.8}
in this case is also a finite summation
of $(1,\,q)_{\rho}$-atoms when $q<\fz$
and continuous $(1,\,\fz)_{\rho}$-atoms when $q=\fz$, and
$$\|f\|_{H^{1,\,q}_{\rho,\,\fin}(\cx)}\ls\sum_{r_j<\rho(x_j)}|\lz_j|+
\sum_{r_j\ge \rho(x_j)}\sum_{\az_j}|\lz_j
\lz_{j,\,{\az_j}}|+
\sum_{\az}\lz_\az\ls
 \|f\|_{H^1_\rho (\cx)}.$$
 On the other hand, obviously,
 $\|f\|_{H^1_\rho (\cx)}\ls \|f\|_{H^{1,\,q}_{\rho,\,\fin}(\cx)},$
which completes the proof of Theorem \ref{t3.2}.
\end{proof}

We point out that an interesting application of finite atomic
decomposition characterizations as in Theorem \ref{t3.2}
is to obtain a general criterion for the boundedness of
certain sublinear operators on Hardy spaces via atoms;
see \cite{msv,yz1,gly1} for similar results.

Let $B$ be a Banach space with the norm $\|\cdot\|_\cb$ and $\mathcal Y$ be a linear space. An
operator $T$ from $\mathcal Y$ to $B$ is called
$B$-sublinear if for all $f,\, g\in\mathcal Y$
and numbers $\lz,\,\nu\in\cc$, we have
$$\|T(\lz f+\nu g)\|_B\le |\lz|\|T(f)\|_B+|\nu|\|T(g)\|_B$$
and $\|T(f)-T(g)\|_B\le\|T(f-g)\|_B$; see \cite{yz1}.
Obviously, if $T$ is linear, then $T$ is $B$-sublinear.
Moreover, if $B\equiv L^r(\cx)$ with $r\ge1$,
$T$ is sublinear in the classical sense and $T(f)\ge 0$
for all $f\in\mathcal Y$, then $T$ is also
$B$-sublinear.
Using Theorem \ref{t3.2} (ii),
we immediately obtain the following result.

\begin{prop}\label{p3.2}
Let $\rho$ be an admissible function, $q\in(1,\,\fz)$, $B$ be a Banach space and
$T$ be a B-sublinear operator from $H^{1,\,q}_{\rho,\,\fin}(\cx)$ to $B$.
If \begin{equation}\label{3.9}
\sup\{\|T(a)\|_B:\ a\ is\ any\ (1,\,q)_{\rho}\mbox{-} atom\}<\fz
\end{equation}
for some $q\in(1,\,\fz)$, or
\begin{equation}\label{3.10}
    \sup\{\|T(a)\|_B:\ a\  is\  any\ continuous\
    (1,\,\fz)_{\rho}\mbox{-}atom \}<\fz,
\end{equation}
then $T$ uniquely extends to a bounded $B$-sublinear operator
from $H^1_\rho (\cx)$ to $B$.
\end{prop}

\begin{proof}
For any $f\in H^{1,\,q}_{\rho,\,\fin}(\cx)$,
by Theorem \ref{t3.2} (ii),
there exist an $N\in\nn$, $\{\lz_j\}_{j=1}^N\subset\cc$, and
$(1, \,q)_{\rho}$-atoms $\{a_j\}_{j=1}^N$ when $q<\fz$
and continuous $(1,\,\fz)_{\rho}$-atoms when $q=\fz$
such that
$f=\sum_{j=1}^N\lz_ja_j$ pointwise and
$\sum_{j=1}^N|\lz_j|\ls\|f\|_{H^1_\rho (\cx)}$. Then
by the assumption \eqref{3.9}, we have that
$\|T(f)\|_B\ls\sum_{j=1}^N
|\lz_j| \ls\|f\|_{H^1_\rho (\cx)}.$
Since $H^{1,\,q}_{\rho,\,\fin}(\cx)$
is dense in ${H^1_\rho (\cx)}$, a density argument
gives the desired conclusion, which completes the proof of Proposition \ref{p3.2}.
\end{proof}

\begin{rem} \label{r3.2}\rm
(i) It is obvious that if $T$ is a bounded $B$-sublinear
operator from $H^1(\cx)$ to $B$, then $T$ maps all
$(1,\,q)_{\rho}$-atoms when $q\in(1,\,\fz)$
and continuous $(1,\,\fz)_{\rho}$-atoms when $q=\fz$ into uniformly bounded
elements of $B$. Thus, in Proposition \ref{p3.2}, the assumption that the
uniform boundedness of $T$ on all
$(1,\,q)_{\rho}$-atoms when $q\in(1,\,\fz)$
and all continuous $(1,\,\fz)_{\rho}$-atoms when $q=\fz$
is actually necessary.

(ii) Even when $B\equiv H^1_\rho (\cx)$ or $B\equiv L^r(\cx)$ with $r\ge1$,
to apply Proposition \ref{p3.2},
it is not necessary to know the continuity of the considered
operator $T$ from any space of test functions to its dual space,
or the boundedness of $T$ in $L^2(\cx)$ or in $L^p(\cx)$
for certain $p\in (1,\fz)$, which may be convenient in applications.

(iii) Suppose that
\begin{equation}\label{3.11}
 \sup\{\|T(a)\|_B:\ a\ {\rm is\ any}\ (1,\,\fz)_{\rho}\mbox{-\rm atom}\}<\fz.
\end{equation}
Denote by $T_0$ the restriction of $T$ in $H^{1,\,\fz}_{\rho,\,\fin}(\cx)$.
Then $T_0$ satisfies \eqref{3.10}.
By Proposition \ref{p3.2}, $T_0$ has an extension, denoted by
$\wz T_0$, such that $\wz T_0$ is bounded from $H^1_\rho (\cx)$ to $L^1(\cx)$.
However, $\wz T_0$ may not coincide with $T$ on all $(1,\,\fz)_{\rho}$-atoms.
See \cite{B2,yz1,msv,gly1} for further details and examples.

(iv) As an replacement of (iii) of this remark, we point out that
if $B\equiv L^q(\cx)$ for some $q\in[1,\,\fz)$, $T$ is bounded from
$L^{p_1}(\cx)$ to $L^{q_1}(\cx)$
for some $p_1,\ q_1\in[1,\,\fz)$, and $T$ satisfies \eqref{3.11},
then $T$ and $\wz T_0$ coincide on all
$(1,\,\fz)_{\rho}$-atoms. In fact,
since $T$ is bounded from $L^{p_1}(\cx)$ to $L^{q_1}(\cx)$,
from Lemma \ref{l2.5} (i) with a $1$-$\ati$ with bounded support,
it is easy to deduce this conclusion.
Therefore, in this case, to obtain the boundedness of $T$
from $H^1_\rho (\cx)$ to $L^q(\cx)$, it is enough to prove \eqref{3.11}.
\end{rem}

As an application of Proposition \ref{p3.2},
we obtain the boundedness in $H^1_\rho (\cx)$
of certain localized singular integrals, which are closely related to $\rho$
and motivated by the Riesz transforms associated
to the Schr\"odinger operators with nonnegative potentials satisfying
the reverse H\"older inequality. In what follows,
$L_b^\fz(\cx)$ denotes the space of functions
$f\in L^\fz(\cx)$ with bounded support.

Let $T$ be a linear operator bounded on $L^q(\cx)$ for some $q\in(1,\,\fz)$.
In addition, suppose that $T$ has associated with a kernel $K$
satisfying that there exist constants $\ez\in(0,1]$ and $\wz C,\,C>0$ such that

(K1) $|K(x,\,y)|\le C\frac1{V(x,\,y)}$ for all $x,\,y\in\cx$ with $x\ne y$;

(K2) $|K(x,\,y)-K(x,\,y')|\le
C\frac1{V(x,\,y)}[\frac{d(y,\,y')}{d(x,\,y)}]^\ez$
for all $x,\,y,\,y'\in\cx$ with $x\ne y$ and $d(y,\,y')\le d(x,\,y)/2$,

(K3) for all $f\in L_b^\fz(\cx)$ and almost all $x\notin\supp f$,
\begin{equation}\label{3.12}
T(f)(x)=\int_\cx K(x,\,y)\eta\lf(\frac{d(x,\,y)}{\wz C\rho(x)}\r)f(y)\,d\mu(y),
\end{equation}
where $\eta$ is as in \eqref{2.8}.
Then we have the following result.
\begin{prop}\label{p3.3}
The operator $T$ as in \eqref{3.12} is bounded from
$H^1_\rho (\cx)$ to $L^1(\cx)$.
\end{prop}

\begin{proof} By Proposition \ref{p3.2}, to show Proposition \ref{p3.3},
it suffices to prove that for all
$(1,\,2)_{\rho}$-atoms $a$,
$\|T(a)\|_{L^1(\cx)}\ls1$.
To this end, assume that the atom $a$ is
supported in $B(x_0,\,r)$ with $r<\rho(x_0)$.
We first claim that $\supp T(a)\subset B(x_0,\,C\rho(x_0))$,
where $C\equiv1+2C_3(C_4)^{-1}(1+2\wz C )^{k_0}$.

In fact, if $x\notin B(x_0,\,C\rho(x_0))$
and $d(x,\,y)>2\wz C\rho(x)$ for all $y\in B(x_0,\,r)$,
then the support assumption of $\eta$ together with \eqref{3.12} implies that
$T(a)(x)=0$. If $x\notin B(x_0,\,C\rho(x_0))$ and
there exists some $y\in B(x_0,\,r)$ such that $d(x,\,y)\le2\wz C\rho(x)$,
then by  Lemma \ref{l2.1} (iii) and \eqref{2.2},
$$\rho(x)\le (C_4)^{-1}(1+2\wz C )^{k_0}\rho(y)\le 2(C_4)^{-1}C_3(1+2\wz C )^{k_0}\rho(x_0);$$
thus,
$$d(x_0,\,y)\ge d(x_0,\,x)-d(x,\,y)\ge C\rho(x_0)-2(C_4)^{-1}C_3(1+2\wz C )^{k_0}
\rho(x_0)\ge\rho(x_0),$$
which is a contradiction with $y\in B(x_0,\,r)$.
Thus, $T(a)(x)=0$ also in this case and this shows the claim.

If $r\ge \rho(x_0)/4$, from the H\"older inequality and
the $L^q(\cx)$-boundedness of $T$,  it then follows that
$$\|T(a)\|_{L^1(\cx)}=\|T(a)\|_{L^1(B(x_0,\,C\rho(x_0)))}
\ls[V_r(x_0)]^{1-1/q}\|a\|_{L^q(B(x_0,\,r))}\ls1.$$

If $r<\rho(x_0)/4$, then by the H\"older inequality and
the $L^q(\cx)$-boundedness of $T$, we have
$$\|T(a)\|_{L^1(B(x_0,\,Cr))}
\ls[V_r(x_0)]^{1-1/q}\|T(a)\|_{L^q(B(x_0,\,Cr))}\ls1.$$
For any $x\in (B(x_0,\,C\rho(x_0))\setminus B(x_0,\,Cr))$,
by $\int_\cx a(y)\,d\mu(y)=0$, (K1) and (K2) together
with Lemma \ref{l2.1} (i), we have
\begin{eqnarray*}
|T(a)(x)|&&=\int_\cx \lf|K(x,\,y)\eta\lf(\frac{d(x,\,y)}{\rho(x)}\r)
-K(x,\,x_0)\eta\lf(\frac{d(x,\,x_0)}{\rho(x)}\r)\r||a(y)|\,d\mu(y)\\
&&\le \int_\cx
|K(x,\,y)-K(x,\,x_0)|\eta\lf(\frac{d(x,\,y)}{\rho(x)}\r)|a(y)|\,d\mu(y)\\
&&\quad + \int_\cx
|K(x,\,x_0)| \lf|\eta\lf(\frac{d(x,\,y)}{\rho(x)}\r)-
\eta\lf(\frac{d(x,\,x_0)}{\rho(x)}\r)\r||a(y)|\,d\mu(y)\\
&&\ls \int_{B(x_0,\,r)} \frac1{V(x,\,x_0)}\lf[\frac{d(x_0,\,y)}{d(x_0,\,x)}\r]^\ez\,|a(y)|\,d\mu(y)\\
&&\quad+ \int_{B(x_0,\,r)}\frac1{V(x,\,x_0)}\frac{d(x_0,\,y)}{\rho(x)}\,|a(y)|\,d\mu(y)\\
&&\ls \frac1{V(x,\,x_0)}\lf[\frac{r}{d(x_0,\,x)}\r]^\ez+
\frac1{V(x,\,x_0)}\frac{r}{\rho(x_0)}
\ls  \frac1{V(x,\,x_0)}\lf[\frac{r}{d(x_0,\,x)}\r]^\ez.
\end{eqnarray*}
Thus, assuming that $2^{j_0}r\le \rho(x_0)<2^{j_0+1}r$ for certain $j_0\in\nn$,
we obtain
\begin{eqnarray*}
&&\int_{B(x_0,\,C\rho(x_0))\setminus B(x_0,\,Cr)}
|T(a)(x)|\,d\mu(x)\\
&&\ls \int_{B(x_0,\,C\rho(x_0))\setminus B(x_0,\,Cr)}
\frac1{V(x,\,x_0)}\lf[\frac{r}{d(x_0,\,x)}\r]^\ez\,d\mu(x)
\ls \sum_{j=0}^{j_0} 2^{-j\ez}\ls1,
\end{eqnarray*}
which completes the proof of Proposition \ref{p3.3}.
\end{proof}

\begin{rem}\label{r3.3}\rm
We should point out that Proposition \ref{p3.2} is
used in Section \ref{s5.4} to prove the boundedness on Hardy spaces of
Riesz transforms associated
to Schr\"odinger operators with potentials satisfying
the reverse H\"older inequality on connected and simply connected
nilpotent Lie groups. Moreover,
there exist many examples of such localized singular integrals as in \eqref{3.12}.
For example, if $x\in\rn$ and
$$T(f)(x)\equiv{\rm p.\,v.}\, \int_\rn \frac{x_j-y_j}
{|x-y|^{n+1}}
\eta\lf(\frac{x-y}{\rho(x)}\r)f(y)\,dy,$$
then $T$ is an operator as in \eqref{3.12}.
Let $\wz T$ be a linear operator bounded on $L^q(\cx)$
for some $q\in(1,\,\fz)$ and for all $x\in\cx$,
$$\wz T(f)(x)\equiv{\rm p.\,v.}\,\int_\cx K(x,\,y)f(y)\,d\mu(y)$$
with $K$ and $K^t$ satisfying (K1) and (K2),
where $K^t(x,\,y)=K(y,\,x)$ for all $x,\,y\in\cx$. Define  $ T$ by setting, for all $x\in\cx$,
$$T(f)(x)\equiv{\rm p.\,v.}\,\int_\cx K(x,\,y)\eta\lf(\frac{d(x,\,y)}{\rho(x)}\r)
f(y)\,d\mu(y).$$
Since the maximal operator $\wz T_\ast$, which is defined by setting, for all $x\in\cx$,
$$\wz T_\ast(f)(x)=\sup_{\ez>0}\lf|\int_{d(x,\,y)>\ez} K(x,\,y)f(y)\,d\mu(y)\r|,$$
is bounded on $L^q(\cx)$ for certain $q\in(1,\,\fz)$ (see, for example, \cite{s93}),
then $T$ is an operator as in \eqref{3.12}.
\end{rem}

\section{Radial maximal function characterizations of
$H_{\rho}^1(\cx)$}\label{s4}

In this section, we establish a radial maximal function characterization of
$H_{\rho}^1(\cx)$ as follows.

\begin{thm}\label{t4.1}
Let $\rho$ be admissible and let
$\ez_1\in(0,\,1]$,\,$\ez_2,\,\ez_3>0$,
$\ez\in(0,\,\ez_1\wedge\ez_2)$ and
$\{S_k\}_{k\in\zz}$ be an $(\ez_1,\,\ez_2,\,\ez_3)$-$\ati$.
Then $f\in H^1_\rho (\cx)$ if and only if
$f\in (\cg^\ez_0(\bz,\,\gz))'$ for some
$\bz,\,\gz\in(0,\,\ez)$ and $\|S^+_{\rho}(f)\|_{L^1(\cx)}<\fz$; moreover, there
exists a positive constant $C$ such that for all $f\in H^1_\rho (\cx)$,
\begin{equation}\label{4.1}
C^{-1}\|S^+_{\rho}(f)\|_{L^1(\cx)}\le
\|f\|_{H^1_\rho (\cx)}\le
C\|S^+_{\rho}(f)\|_{L^1(\cx)}.
\end{equation}
\end{thm}

For the sake of applications, we need the following
characterization of $H^1_\rho (\cx)$ via a
variant of the radial maximal function.

\begin{thm}\label{t4.2}
Let $\rho$ be an admissible function.
Assume that $\{T_t\}_{t>0}$ is a family of linear operators
bounded on $L^2(\cx)$ with
integrable kernels $\{T_t(x,\,y)\}_{t>0}$
satisfying that there exist a continuous
$(\ez_1,\,\ez_2,\,\ez_3)$-$\ati$
$\{\wz T_t\}_{t>0}$ for some $\ez_1\in(0,\,1]$ and $\ez_2,\,\ez_3>0$,
constants $C>0$,  $\dz_2\in(0,\,\ez_2]$ and $\dz_1,\,\dz_3>0$
such that for all $x,\,y\in\cx$,

\noindent(i) $|T_t(x,\,y)|\le C\frac1{V_t(x)+V(x,\,y)}[\frac
t{t+d(x,\,y)}]^{\dz_2}[\frac{\rho(x)}{t+\rho(x)}]^{\dz_3}$;

\noindent(ii)   $|T_t(x,\,y)-\wz T_t(x,\,y)|\le C[\frac{t}{t+\rho(x)}]^{\dz_1}
\frac1{V_t(x)+V(x,\,y)}[\frac t{t+d(x,\,y)}]^{\dz_2}.$

\noindent Then the following are equivalent:
(\rm{a}) $f\in H^1_\rho (\cx)$;  (\rm{b}) $f\in L^1(\cx)$
and $\|T^+(f)\|_{L^1(\cx)}<\fz$; (\rm{c}) $f\in L^1(\cx)$ and
 $\|T^+_{\rho}(f)\|_{L^1(\cx)}<\fz$. Moreover, for all $f\in L^1(\cx)$,
$$\|f\|_{H^1_\rho (\cx)}\sim\|T^+(f)\|_{L^1(\cx)}\sim
\|T^+_{\rho}(f)\|_{L^1(\cx)},$$
where
$T^+(f)(x)\equiv\sup_{t>0}|T_t(f)(x)|$ and
$T^+_{\rho}(f)(x)\equiv\sup_{0<t<\rho(x)}|T_t(f)(x)|$ for all $x\in\cx$.
\end{thm}

\begin{rem}\label{r4.1}\rm
(i) If $\{T_t\}_{t>0}$ and $\{\wz T_t\}_{t>0}$
satisfy the assumptions of Theorem \ref{t4.2},
then it is easy to see that for all $f\in L^1_\loc(\cx)$ and $x\in\cx$,
$$T^+(f)(x)\ls \wz T^+(f)(x)+M(f)(x)\ls M(f)(x),$$
and thus $T^+$ is bounded on $L^p(\cx)$ for $p\in(1,\,\fz]$ and
bounded from $L^1(\cx)$ to weak-$L^1(\cx)$.
Moreover, for all $f\in L^1(\cx)$, observing that for almost all $x\in\cx$,
by (ii) of Theorem \ref{t4.2},
$$|f(x)|=\lim_{t<\rho(x),\,t\to0}|\wz T_t(f)(x)|
\le T^+_{\rho}(f)(x)+C\lim_{t\to0}
\lf[\frac{t}{\rho(x)}\r]^{\dz_3}M(f)(x)
\le T^+_{\rho}(f)(x),$$
we have that $\|f\|_{L^1(\cx)}\le\|T^+_{\rho}(f)\|_{L^1(\cx)}
\le\|T^+(f)\|_{L^1(\cx)}$.

(ii)  Let $\{\wz T_t\}_{t>0}$ be as in Theorem \ref{t4.2}. Then
$\{\wz T_t(x,\,y)\chi_{\{t\le C\rho(x)\}}(x)\}_{t>0}$ satisfies
(i) and (ii) of Theorem \ref{t4.2} with  $\dz_2=\ez_2$ and
any $\dz_1,\,\dz_3>0$. Moreover, let $\dz_3>0$ and
for all $t>0$ and $x,\,y\in\cx$, define
$$T_t(x,\,y)\equiv\wz T_t(x,\,y)
\frac{[\rho(x)]^{\dz_3}}{t^{\dz_3}+[\rho(x)]^{\dz_3}}.$$
Then it is easy to verify that $\{T_t\}_{t>0}$
satisfies (i) and (ii) of Theorem \ref{t4.2} with
$\dz_1\equiv\dz_3$ and $\dz_2\equiv\ez_2$.
\end{rem}

To prove Theorem \ref{t4.1},
we need a variant of the inhomogeneous discrete
Calder\'on reproducing formula established in \cite{hmy2}.
This variant was established in \cite{gly1}.
To state this variant, we first recall the dyadic cubes on spaces of
homogeneous type constructed by Christ \cite{ch}.

\begin{lem}\label{l4.1}
Let $\cx$ be a space of homogeneous type. Then there exists a
collection $\{Q_\az^k\subset\cx:\ k\in\zz,\,\az\in I_k\}$ of open
subsets of $\cx$, where $I_k$ is some index set, and the constants
$\dz\in(0,\,1)$ and $\wz C_1,\,\wz C_2>0$ such that

(i) $\mu(\cx\setminus\cup_\az Q_\az^k)=0$ for each fixed $k$ and
$(Q_\az^k\cap Q_\bz^k)=\emptyset$ if $\az\ne\bz$;

(ii) for any $\az,\,\bz,\,k,\,\ell$ with $\ell\ge k$, either
$Q_\bz^\ell\subset Q_\az^k$ or $(Q_\bz^\ell\cap Q_\az^k)=\emptyset$;

(iii) for each $(k,\,\az)$ and $\ell<k$, there exists a unique $\bz$
such that $Q_\az^k\subset Q_\bz^\ell$;

(iv) $\diam(Q_\az^k)\le \wz C_1\dz^k$, where $\diam(Q_\az^k)\equiv{\sup}\{d(x,\,y):\ x,\,y\in Q_\az^k\}$;

(v) each $Q_\az^k$ contains some ball $B(z_\az^k,\,\wz C_2\dz^k)$, where
$z_\az^k\in\cx$.
\end{lem}

In fact, we can think of $Q_\az^k$ as being a dyadic cube with
diameter rough $\dz^k$ centered at $z_\az^k$. In what follows, for
simplicity, we always assume that $\dz=1/2$; see \cite{hmy2}
for how to remove this restriction.

For any $j\in \nn$, $k\in\zz$ and $\tau\in I_k$, denote by
$Q_\tau^{k,\,\nu}$, \, $\nu=1,\,2,\,\cdots,\,N(k,\,\tau)$, the set of
all cubes $Q_{\tau'}^{k+j}\subset Q_\tau^k$.
We also denote by $z_\tau^{k,\,\nu}$ the center of $Q_\tau^{k,\,\nu}$ and
$y_\tau^{k,\,\nu}$ any point of $Q_\tau^{k,\,\nu}$.
For $\ell\in\zz$ and $j\in\nn$, set
\begin{equation}\label{4.2}
\cd(\ell,\,j)\equiv\lf\{y_\tau^{k,\,\nu}\in Q_\tau^{k,\,\nu}:\
k=\ell,\,\cdots,\,\,\fz,\ \tau\in I_k, \ \nu=1,\,\cdots,\,
N(k,\,\tau)\r\}.
\end{equation}
In what follows, for any set $E$ and locally integrable
function $f$, set
$$m_E(f)\equiv\frac1{\mu(E)}\int_E f(z)\,d\mu(z).$$
Let $j_0\in\nn$ such that
\begin{equation}\label{4.3}
2^{-j_0}\wz C_1<1/3.
\end{equation}

The following Calder\'on reproducing formula comes from
\cite{gly1}.

\begin{lem}\label{l4.2}
Let $\ez_1\in(0,\,1]$,\,$\ez_2,\,\ez_3>0$,
$\ez\in(0,\,\ez_1\wedge\ez_2)$ and $\{S_k\}_{k\in\zz}$ be an
$(\ez_1,\,\ez_2,\,\ez_3)-\ati$. Then there exists $j_1>j_0$ with $j_0$
as in \eqref{4.3}
such that for any $\ell\in\zz$ and $\cd(\ell+1,\,j_1)$ as in \eqref{4.2},
there exist operators $\{\wz D_k\}_{k=\ell}^\fz$ with kernels $\{\wz
D_k(x,\,y)\}_{k=\ell}^\fz$ such that for any
$f\in(\cg^\ez_0(\bz,\,\gz))'$ with $\bz,\,\gz\in(0,\,\ez)$,
\begin{eqnarray*}
f(x)&&= \sum_{\tau\in I_\ell}\sum_{\nu=1}^{N(\ell,\,\tau)}
\int_{Q_\tau^{\ell,\,\nu}}\wz D_\ell(x,\,y)d\mu(y)
\, m_{Q^{\ell,\,\nu}_\tau}(S_\ell(f))\\
&&\quad+\sum_{k=\ell+1}^\fz\sum_{\tau\in
I_k}\sum_{\nu=1}^{N(k,\,\tau)} \mu(Q_\tau^{k,\,\nu})\wz
D_{k}(x,\,y^{k,\,\nu}_\tau) D_k(f)(y^{k,\,\nu}_\tau),
\end{eqnarray*}
where $D_k\equiv S_k-S_{k-1}$ for any $k\ge\ell+1$ and
the series converge in $(\cg^\ez_0(\bz,\,\gz))'$. Moreover,
for any $\ez'\in[\ez,\,\ez_1\wedge\ez_2)$, there exists a positive constant $
C_{\ez'}$ depending on $\ez'$ but not on $\ell$, $j_1$ and
$\cd(\ell+1,\,j_1)$ such that $\wz D_k$ for $k\ge\ell$
satisfies (i) and (ii) of Definition \ref{d2.3} with $\ez_1$ and
$\ez_2$ replaced by $\ez'$ and the constant $C_6$ replaced by $
C_{\ez'} $, and $\int_\cx \wz D_k(z,\,y)\,d\mu(z)=\int_\cx \wz
D_k(x,\,z)\,d\mu(z)=1$ when $k=\ell$, and $=0$ when
$k\ge\ell+1$.
\end{lem}

The following estimate is a variant of Lemma 5.3 in \cite{hmy2},
which is also used in the proof of Theorem \ref{t4.1}.

\begin{lem}\label{l4.3}
Let $\ez>0$ and $r\in(n/(n+\ez),\,1]$.
Then there exists a positive constant $C$
such that for all $k,\,k'\in\zz$,
$a^{k,\,\nu}_\tau\in\cc$,
 $y^{k,\,\nu}_\tau\in Q^{k,\,\nu}_\tau$,
 $\wz Q^{k,\,\nu}_\tau\subset Q^{k,\,\nu}_\tau$ with
$\tau\in I_k$ and $\nu=1,\,\cdots,\,N(k,\,\tau)$, and $x\in\cx$,
\begin{eqnarray*}
&&\sum_{\tau\in I_k}\sum_{\nu=1}^{N(k,\,\tau)}\mu(\wz Q^{k,\,\nu}_\tau)
 \frac{|a^{k,\,\nu}_\tau|}{V_{2^{-(k'\wedge k)}}+V(x,\,y^{k,\,\nu}_\tau)}
\lf[\frac{2^{-(k'\wedge k)}}{2^{-(k'\wedge k)}+d(x,\,y^{k,\,\nu}_\tau)}\r]^{\ez}\\
&&\quad\le C2^{[(k'\wedge k)-k]n(1-1/r)}
\lf\{M\lf(\sum_{\tau\in I_k}\sum_{\nu=1}^{N(k,\,\tau)}
|a^{k,\,\nu}_\tau|^r\chi_{\wz Q^{k,\,\nu}_\tau}\r)(x)\r\}^{1/r}.
\end{eqnarray*}
\end{lem}

We point out that if
$\wz Q^{k,\,\nu}_\tau=Q^{k,\,\nu}_\tau$,
then this is just Lemma 5.3 of \cite{hmy2}.
The proof of Lemma \ref{l4.3} is a slight modification of the proof of
\cite[Lemma 5.3]{hmy2}. We omit the details.

We point out that the following approach used in the proof of Theorem \ref{t4.1}
is totally different from that used by Dziuba\'nski and
Zienkiewicz in their papers \cite{d98,dz99,dz02,d05}
to obtain a similar result on $\rn$.
The method in \cite{d98,dz99,dz02,d05} strongly depends on
an existing theory of localized Hardy spaces $h^1$
in the sense of Goldberg \cite{g}. Our method
successfully avoids this via the
discrete Calder\'on reproducing formula,
Lemma \ref{l4.2}.

\begin{proof}[Proof of Theorem \ref{t4.1}]
By \eqref{2.9}, to prove Theorem \ref{t4.1},
we only need to prove the second inequality in \eqref{4.1}.
To this end, let  $f\in(\cg_0^\ez(\bz,\,\gz))'$
with $\ez,\,\bz,\,\gz$ as in Definition \ref{d2.8}
such that
$\|S^+_{\rho}(f)\|_{L^1(\cx)}<\fz.$
Then by the proof of Lemma \ref{l3.1},
$f\in L^1(\cx)$ in the sense of $(\cg_0^\ez(\bz,\,\gz))'$
and $\|f\|_{L^1(\cx)}\le\|S^+_{\rho}(f)\|_{L^1(\cx)}$.
For any $x\in\cx$, there exists an $\ell\in\zz$ such that
$2^{-\ell}<\rho(x)\le2^{-\ell+1}$.
We first claim that for any $\vz\in\cg^\ez_0(\bz,\,\gz)$ satisfying that
$\int_\cx \vz(x)\,d\mu(x)=0$ and $\|\vz\|_{\cg^\ez_0(x,\,2^{-k'},\,\bz,\,\gz)}\le1$
for some $k'\ge\ell+1$, we have
\begin{eqnarray}\label{4.4}
|\langle f,\,\vz\rangle|&&\ls
\lf\{M\lf(\lf[S^+_{\rho}(f)\r]^r\r)(x)\r\}^{1/r}\\
&&\quad+\lf\{M\lf(\sum_{\tau\in I_\ell}
\sum_{\nu=1}^{N(\ell,\,\tau)}\lf[m_{Q^{\ell,\,\nu}_\tau}(|S^+_{\rho}(f)|)\r]^r
\chi_{Q^{\ell,\,\nu}_\tau}\r)(x)\r\}^{1/r}\nonumber\\
&&\quad+ \int_\cx \frac1{V_{\rho(z)}(z)+V(x,\,z)}
\lf(\frac{\rho(z)}{\rho(z)+d(x,\,z)}\r)^{\gz/(1+k_0)}|f(z)|\,d\mu(z),\nonumber
\end{eqnarray}
where $r\in(n/(n+\ez),\,1)$.

Assume that the above claim holds temporarily. Notice that for any function
$\phi\in\cg^\ez_0(\bz,\,\gz)$
with $\|\phi\|_{\cg^\ez_0(x,\,2^{-k'},\,\bz,\,\gz)}\le1$
for some $k'\ge\ell+1$, we have
$$\sz\equiv\lf|\int_\cx \phi(y)\,d\mu(y)\r|\le \int_\cx \frac1{V_{2^{-k'}}(x)+V(x,\,y)}
\lf(\frac{2^{-k'}}{2^{-k'}+d(x,\,y)}\r)^\gz\,d\mu(y)\ls1.$$
Set $\vz(y)\equiv\frac1{1+\sz C_6}[\phi(y)-\sz S_{k'}(x,\,y)]$ for all $y\in\cx$.
Obviously, $\int_\cx \vz(y)\,d\mu(y)=0$
and $\|\vz\|_{\cg^\ez_0(x,\,2^{-k'},\,\bz,\,\gz)}\le 1$.
Then from \eqref{4.4},
$$|\langle f,\,\phi\rangle|
\le \sz|S_{k'}(f)(x)|+(1+\sz C_6)|\langle f,\,\vz\rangle|$$
and
$$|S_{k'}(f)(x)|\le S^+_{\rho}(f)(x)\le
\lf\{M\lf(\lf[S^+_{\rho}(f)\r]^r\r)(x)\r\}^{1/r}$$
for almost all $x\in\cx$,
it follows that \eqref{4.4} still holds with
$|\langle f,\,\vz\rangle|$ replaced by $G_{\rho}(f)$ as in Definition \ref{d2.5}.
This together with the boundedness on $L^{1/r}(\cx)$
of the Hardy-Littlewood operator $M$
and $\|f\|_{L^1(\cx)}\le\|S^+_{\rho}(f)\|_{L^1(\cx)}$ implies that
\begin{eqnarray*}
\|G_{\rho}(f)\|_{L^1(\cx)}
&&\ls\|S^+_\rho(f)\|_{L^1(\cx)}+\lf\|\sum_{\tau\in I_\ell}
\sum_{\nu=1}^{N(\ell,\,\tau)}m_{Q^{\ell,\,\nu}_\tau}(S^+_{\rho}(f))
\chi_{Q^{\ell,\,\nu}_\tau}\r\|_{L^1(\cx)}\\
&&\quad+ \int_\cx \int_\cx \frac{|f(z)|}{V_{\rho(z)}(z)+V(x,\,z)}
\lf(\frac{\rho(z)}{\rho(z)+d(x,\,z)}\r)^{\gz/(1+k_0)}\,d\mu(z)\,d\mu(x)\\
&&\ls \|S^+_{\rho}(f)\|_{L^1(\cx)}+\|f\|_{L^1(\cx)}
\ls\|S^+_{\rho}(f)\|_{L^1(\cx)},
\end{eqnarray*}
which establishes the second inequality of \eqref{4.1} in Theorem \ref{t4.1}.

To prove the claim \eqref{4.4},
by Lemma \ref{l4.2} with the same notation as there, we write
\begin{eqnarray*}
\langle f,\,\vz\rangle&&=
\sum_{\tau\in I_\ell}\sum_{\nu=1}^{N(\ell,\,\tau)}
\int_{Q_\tau^{k,\,\nu}} \wz D^\ast_{\ell}(\vz)(y)\,d\mu(y)
m_{Q_\tau^{\ell,\,\nu}}(S_\ell(f))\\
&&\hs+\sum_{k=\ell+1}^\fz
\sum_{\tau\in I_k}\sum_{\nu=1}^{N(k,\,\tau)}
\mu(Q_\tau^{k,\,\nu})\wz D^\ast_{k}(\vz)(y^{k,\,\nu}_\tau)
D_k(f)(y^{k,\,\nu}_\tau)\equiv I_1+I_2,
\end{eqnarray*}
where $\wz D_k^\ast$ denotes the integral operator with kernel
$\wz D_k^\ast(x,\,y)\equiv\wz D_k(y,\,x)$ for all $x,\,y\in\cx$.

Observe that for all $k\ge k'$,
by using $\int_\cx D_k(x,\,y)\,d\mu(y)=0$,
(i) and (ii) of $S_k$ in Definition \ref{d2.3} and the size condition of $\vz$,
we  obtain that for all $y\in\cx$,
\begin{eqnarray}\label{4.5}
&&|\wz D^\ast_{k}(\vz)(y)|
\ls 2^{-(k-k')\bz}\frac1{V_{2^{-k'}}(x)+V(x,\,y)}
\lf(\frac{2^{-k'}}{2^{-k'}+d(x,\,y)}\r)^{\gz},
\end{eqnarray}
and that for all $k<k'$, by using
$\int_\cx \vz(y)\,d\mu(y)=0$, the size condition and the regularity of $\vz$,
(i) for $D_k$ in Definition \ref{d2.3}, we have that for all $y\in\cx$,
\begin{eqnarray}\label{4.6}
&&|\wz D^\ast_{k}(\vz)(y)|
\ls 2^{-(k'-k)\gz'}\frac1{V_{2^{-k}}(x)+V(x,\,y)}
\lf(\frac{2^{-k}}{2^{-k}+d(x,\,y)}\r)^{\gz},
\end{eqnarray}
where $\gz'\in(0,\,\gz)$; see the proof of Proposition 5.7
in \cite{hmy2} for some details.

Moreover, in what follows, set
$\wz Q^{k,\,\nu}_\tau\equiv\{y\in Q^{k,\,\nu}_\tau:\ 2^{-k}<\rho(y)/2\}$
and $\overline Q^{k,\,\nu}_\tau\equiv (Q^{k,\,\nu}_\tau\setminus \wz Q^{k,\,\nu}_\tau)$.
With the subtle split of $Q^{k,\,\nu}_\tau$ together with the arbitraries of
$y^{k,\,\nu}_\tau\in Q^{k,\,\nu}_\tau$, we have
$$\inf_{y\in Q^{k,\,\nu}_\tau}|D_k(f)(y)|\ls
\inf_{y\in\wz Q^{k,\,\nu}_\tau}|S^+_{\rho}(f)(y)|,$$
and
\begin{eqnarray*}
\inf_{y\in Q^{k,\,\nu}_\tau}|D_k(f)(y)|&&\le
\inf_{y\in\overline Q^{k,\,\nu}_\tau}|D_k(f)(y)|\\
&&\ls\inf_{y\in \overline Q^{k,\,\nu}_\tau}
\int_\cx\frac{|f(z)|}{V_{2^{-k}}(y)+V(y,\,z)}
\lf(\frac{2^{-k}}{2^{-k}+d(y,\,z)}\r)^\gz\,d\mu(z).
\end{eqnarray*}
This together with \eqref{4.5} and \eqref{4.6} implies that
\begin{eqnarray*}
|I_2|&&\ls\sum_{k=\ell+1}^\fz
\sum_{\tau\in I_k}\sum_{\nu=1}^{N(k,\,\tau)}
2^{-|k'-k|(\bz\wedge\gz')}\frac{\mu(\wz Q_\tau^{k,\,\nu})}
{V_{2^{-(k'\wedge k)}}(x)+V(x,\,y)}
\lf(\frac{2^{-(k'\wedge k)}}{2^{-(k'\wedge k)}+d(x,\,y)}\r)^{\gz}\\
&&\quad\times
\inf_{y\in\wz Q^{k,\,\nu}_\tau}|S^+_{\rho}(f)(y)|\\
&&\quad+\sum_{k=\ell+1}^\fz
\sum_{\tau\in I_k}\sum_{\nu=1}^{N(k,\,\tau)}
2^{-|k'-k|(\bz\wedge\gz')}\int_{\overline Q_\tau^{k,\,\nu}}
\frac1{V_{2^{-(k'\wedge k)}}(x)+V(x,\,y)}
\lf(\frac{2^{-(k'\wedge k)}}{2^{-(k'\wedge k)}+d(x,\,y)}\r)^{\gz}\\
&&\quad\times\lf\{
\int_\cx\frac{|f(z)|}{V_{2^{-k}}(y)+V(y,\,z)}
\lf(\frac{2^{-k}}{2^{-k}+d(y,\,z)}\r)^\gz\,d\mu(z)\r\}\,d\mu(y)
\equiv I_{2,\,1}+I_{2,\,2}.
\end{eqnarray*}
We first estimate $I_{2,\,2}$ by writing
\begin{eqnarray*}
I_{2,\,2}&&\ls \int_\cx|f(z)|\lf\{\sum_{k=\ell+1}^\fz2^{-|k'-k|(\bz\wedge\gz')}
\int_{\rho(y)\le 2^{-k+1}}
\frac1{V_{2^{-(k'\wedge k)}}(x)+V(x,\,y)}\r.\\
&&\quad\quad\times\lf(\frac{2^{-(k'\wedge k)}}{2^{-(k'\wedge k)}+d(x,\,y)}\r)^{\gz}
\frac1{V_{2^{-k}}(y)+V(y,\,z)}
\lf(\frac{2^{-k}}{2^{-k}+d(y,\,z)}\r)^\gz\,d\mu(y)\Bigg\}\,d\mu(z)\\
&&\equiv \int_\cx |f(z)|I_{2,\,2}(x,\,z)\,d\mu(z).
\end{eqnarray*}
If we can show that for all $x,\,z\in\cx$,
\begin{equation}\label{4.7}
I_{2,\,2}(x,\,z)\ls\frac1{V_{\rho(x)}(x)+V(x,\,z)}
\lf(\frac{\rho(x)}{\rho(x)+d(x,\,z)}\r)^\gz,
\end{equation}
then by Lemma \ref{2.1} (ii) and \eqref{2.2} together with
$V_{\rho(x)}+V(x,\,z)\sim V_{\rho(z)}(z)+V(x,\,z)$ for all $x,\,z\in\cx$,
we have
\begin{equation}\label{4.8}
I_{2,\,2}\ls\int_\cx \frac{|f(z)|}{V_{\rho(z)}(z)+V(x,\,z)}
\lf(\frac{\rho(z)}{\rho(z)+d(x,\,z)}\r)^{\gz/(1+k_0)}\,d\mu(z),
\end{equation}
which is a desired estimate.

To see \eqref{4.7}, notice that by Lemma \ref{l2.1} (i), if $d(x,\,y)<\rho(x)$,
then there exists a positive constant $\wz C$ such that
$(\wz C)^{-1}\rho(y)<\rho(x)<\wz C\rho(y)$.
Thus if $\wz C2^{-(k'\wedge k)+1}\le\rho(x)$
and $\rho(y)\le2^{-k+1}$, then we have $d(x,\,y)\ge\rho(x)$.
From this,  it follows that $(2^{-(k'\wedge k)}+d(x,\,y))\gs \rho(x)$.
Therefore, if $d(x,\,z)<2\rho(x)$, then we have
\begin{eqnarray*}
I_{2,\,2}(x,\,z)&&\ls
\frac1{V_{\rho(x)}(x)}\sum_{k=\ell+1}^\fz 2^{-|k'-k|(\bz\wedge\gz')}
\int_\cx\frac1{V_{2^{-k}}(y)+V(y,\,z)}\\
&&\quad\times
\lf(\frac{2^{-k}}{2^{-k}+d(y,\,z)}\r)^\gz\,d\mu(y)\ls\frac1{V_{\rho(x)}(x)}.
\end{eqnarray*}
If $d(x,\,z)\ge 2\rho(x)$ and $d(x,\,y)>d(x,\,z)/2$,
similarly, we have
\begin{eqnarray*}
I_{2,\,2}(x,\,z)&&\ls
\frac1{V_{\rho(x)}(x)}\lf(\frac{\rho(x)}{\rho(x)+d(x,\,z)}\r)^\gz
\sum_{k=\ell+1}^\fz 2^{-|k'-k|(\bz\wedge\gz')}
\int_\cx\frac1{V_{2^{-(k'\wedge k)}}(y)+V(y,\,z)}\\
&&\quad\times
\lf(\frac{2^{-k}}{2^{-k}+d(y,\,z)}\r)^\gz\,d\mu(y)
\ls\frac1{V_{\rho(x)}(x)}\lf(\frac{\rho(x)}{\rho(x)+d(x,\,z)}\r)^\gz.
\end{eqnarray*}
If $d(x,\,z)\ge 2\rho(x)$ and $d(y,\,z)>d(x,\,z)/2$,
then by $2^{-(k'\wedge k)}\le\rho(x)$,
\begin{eqnarray*}
I_{2,\,2}(x,\,z)&&\ls
\frac1{V_{\rho(x)}(x)}\lf(\frac{\rho(x)}{\rho(x)+d(x,\,z)}\r)^\gz
\sum_{k=\ell+1}^\fz 2^{-|k'-k|(\bz\wedge\gz')}
\int_\cx\frac1{V_{2^{-(k'\wedge k)}}(x)+V(x,\,y)}\\
&&\quad\times
\lf(\frac{2^{-(k'\wedge k)}}{2^{-(k'\wedge k)}+d(x,\,y)}\r)^\gz\,d\mu(y)
\ls\frac1{V_{\rho(x)}(x)}\lf(\frac{\rho(x)}{\rho(x)+d(x,\,z)}\r)^\gz.
\end{eqnarray*}
Combining these estimates implies \eqref{4.7}.

On the other hand, if we choose $r\in(n/(n+\ez),\,1)$ such that
$1/r>1-(\bz\wedge\gz')$, then by Lemma \ref{l4.3}, we have
\begin{eqnarray*}
I_{2,\,1}&&\ls\sum_{k=\ell+1}^\fz 2^{-|k'-k|(\bz\wedge\gz')}
2^{[(k'\wedge k)-k]n(1-1/r)}\\
&&\quad\times\lf[M\lf(\sum_{\tau\in I_k}\sum_{\nu=1}^{N(k,\,\tau)}
\lf|\inf_{y\in\wz Q^{k,\,\nu}_\tau}S^+_{\rho}(f)(y)\r|^r
\chi_{\wz Q^{k,\,\nu}_\tau}\r)(x)\r]^{1/r}\\
&&\ls \sum_{k=\ell+1}^\fz 2^{-|k'-k|(\bz\wedge\gz')}
2^{[(k'\wedge k)-k]n(1-1/r)}\lf[M\lf(\lf[S^+_{\rho}(f)\r]^r\r)(x)\r]^{1/r}\\
&&\ls\lf[M\lf(\lf[S^+_{\rho}(f)\r]^r\r)(x)\r]^{1/r}.
\end{eqnarray*}
Combining the estimates for $I_{2,\,1}$ and $I_{2,\,2}$ yields that
$$|I_2|\ls\lf[M\lf(\lf[S^+_{\rho}(f)\r]^r\r)(x)\r]^{1/r}
+\int_\cx\frac{|f(z)|}{V_{\rho(z)}(z)+V(x,\,z)}
\lf(\frac{\rho(z)}{\rho(z)+\rho(x,\,z)}\r)^{\gz/(1+k_0)}\,d\mu(z).$$

To estimate $I_1$, set
$$\wz m_{\wz Q^{k,\,\nu}_\tau}(g)
\equiv\frac1{\mu(Q^{k,\,\nu}_\tau)}\int_{\wz Q^{k,\,\nu}_\tau}g(z)\,d\mu(z).$$
Then by $k'\ge \ell$,
Lemma \ref{4.4} and \eqref{4.6}, we have
\begin{eqnarray*}
|I_1|&&\ls\sum_{\tau\in I_\ell}\sum_{\nu=1}^{N(\ell,\,\tau)}
\frac{\mu(\wz Q_\tau^{\ell,\,\nu})}{V_{2^{-\ell}}(x)+V(x,\,y)}
\lf(\frac{2^{-\ell}}{2^{-\ell }+d(x,\,y)}\r)^{\gz}
\wz m_{\wz Q^{\ell ,\,\nu}_\tau}(|S^+_{\rho}(f)(y)|)\\
&&\quad+\sum_{\tau\in I_\ell }\sum_{\nu=1}^{N(\ell ,\,\tau)}
\int_{\overline Q_\tau^{\ell ,\,\nu}}\frac1{V_{2^{- \ell }}(x)+V(x,\,y)}
\lf(\frac{2^{-\ell }}{2^{-\ell }+d(x,\,y)}\r)^{\gz}\\
&&\quad\times\lf\{\int_\cx\frac{|f(z)|}{V_{2^{-\ell }}(y)+V(y,\,z)}
\lf(\frac{2^{-\ell }}{2^{-\ell }+d(y,\,z)}\r)^\gz\,d\mu(z)\r\}\,d\mu(y)
\equiv I_{1,\,1}+I_{1,\,2}.
\end{eqnarray*}
By an argument similar to that used in \eqref{4.8}, we have
that \eqref{4.8} still holds by replacing $I_{2,\,2}$ with $I_{1,\,2}$.
For $I_{1,\,1}$, similarly to the estimate for $I_{2,\,1}$,
by Lemma \ref{l4.3}, we have
\begin{eqnarray*}
I_{1,\,1}&&\ls\lf[M\lf(\sum_{\tau\in I_\ell }\sum_{\nu=1}^{N(\ell ,\,\tau)}
\lf[\wz m_{\wz Q^{\ell ,\,\nu}_\tau}\lf(S^+_{\rho}(f)\r)\r]^r
\chi_{\wz Q^{\ell ,\,\nu}_\tau}\r)(x)\r]^{1/r},
\end{eqnarray*}
which together with the obvious inequality
$$\wz m_{\wz Q^{\ell ,\,\nu}_\tau}\lf(S^+_{\rho}(f)\r)
\chi_{\wz Q^{\ell ,\,\nu}_\tau}\le
m_{Q^{\ell ,\,\nu}_\tau}\lf(S^+_{\rho}(f)\r)
\chi_{ Q^{\ell ,\,\nu}_\tau}$$ further implies the
desired estimate. Thus, we have
\begin{eqnarray*}
|I_1|&&\ls\lf[M\lf(\sum_{\tau\in I_\ell}\sum_{\nu=1}^{N(\ell,\,\tau)}
\lf[\wz m_{\wz Q^{\ell,\,\nu}_\tau}\lf(S^+_{\rho}(f)\r)\r]^r
\chi_{\wz Q^{\ell,\,\nu}_\tau}\r)\r]^{1/r}\\
&&+\int_\cx \frac{|f(z)|}{V_{\rho(z)}(z)+V(x,\,z)}
\lf(\frac{\rho(z)}{\rho(z)+d(x,\,z)}\r)^{\gz/(1+k_0)}\,d\mu(z).
\end{eqnarray*}
Combining the estimates for $I_1$ and $I_2$ yields \eqref{4.4} and
hence, completes the proof of Theorem \ref{t4.1}.
\end{proof}

\begin{rem}\label{r4.2}\rm
Theorem \ref{t4.1} still holds with the $(\ez_1,\,\ez_2,\,\ez_3)$-$\ati$ replaced
by the continuous $(\ez_1,\,\ez_2,\,\ez_3)$-$\ati$.
In fact, if $\{\wz S_t\}_{t>0}$ is a continuous $(\ez_1,\,\ez_2,\,\ez_3)$-$\ati$,
letting $S_k\equiv \wz S_{2^{-k}}$ for $k\in\zz$, then
$\{S_k\}_{k\in\zz}$ is an $(\ez_1,\,\ez_2,\,\ez_3)$-$\ati$
and, by Theorem \ref{t4.1},
$$\|\wz S^+_{\rho}(f)\|_{L^1(\cx)}\ls
\|G_{\rho}(f)\|_{L^1(\cx)}\ls
\|S^+_{\rho}(f)\|\ls\|\wz S^+_{\rho}(f)\|_{L^1(\cx)},$$
where
$\wz S^+_{\rho}(f)(x)\equiv\sup_{0<t<\rho(x)}|\wz S_t(f)(x)|$
for all $x\in\cx$. The above claim is true.
\end{rem}

To prove Theorem \ref{t4.2}, we need the following estimate.
Let $T_t$ and $\wz T_t$ be as in Theorem \ref{t4.2}.
For $x,\,y\in\cx$, set $E_t(x,\,y)\equiv T_t(x,\,y)-\wz T_t(x,\,y)$
and $$E^+_{\rho}(f)(x)\equiv\sup_{0<t<\rho(x)}|E_t(f)(x)|.$$

\begin{lem}\label{l4.4}
Under the same assumptions as in Theorem \ref{t4.2}, then there exists a positive
constant $C$ such that for all $f\in L^1(\cx)$,
$\|E^+_{\rho}(f)\|_{L^1(\cx)}\le C\|f\|_{L^1(\cx)}.$
\end{lem}

\begin{proof}
By Lemmas \ref{l2.3} and \ref{l2.4}, it suffices to prove that for all $\az$,
\begin{equation}\label{4.9}
 \|E^+_{\rho}(\chi_{B_\az^\ast} f)\|_{L^1(\cx)}
\ls\|\chi_{B_\az^\ast} f\|_{L^1(\cx)}.
\end{equation}
To this end, notice that for all $x,\,y\in\cx$ with $x\ne y$, by \eqref{2.1},
$$\lf[\frac {t+d(x,\,y)}{d(x,\,y)}\r]^\kz V(x,\,y)\ls
V_{t+d(x,\,y)}(x)\sim
 V_t(x)+V(x,\,y).$$
Thus for any $x\in B_\az^{\ast\ast}$ and $y\in B_\az^\ast$,
since $\rho(y)\sim\rho(x_\az)\sim\rho(x)$ via Lemma \ref{2.1} (i),
by the assumption (ii) of Theorem \ref{t4.2}, we have
\begin{eqnarray*}
|E_t(x,\,y)|&&\ls \frac1{V_t(x)+V(x,\,y)}
\lf[\frac t{t+d(x,\,y)}\r]^{\dz_2}\lf[\frac t{t+\rho(x)}\r]^{\dz_1}\\
&&\ls \frac1{V(x,\,y)}\lf[\frac {d(x,\,y)}{t+d(x,\,y)}\r]^\kz
\lf[\frac t{t+\rho(x_\az)}\r]^{\dz_1}
\ls \frac1{V(x,\,y)}\lf[\frac {d(x,\,y)}{\rho(x_\az)}\r]^{\kz\wedge \dz_1},
\end{eqnarray*}
which implies that
\begin{eqnarray*}
&&\int_{B_\az^{\ast\ast}}\sup_{0<t<\rho(x)}
|E_t(\chi_{B_\az^\ast} f)(x)|\,d\mu(x)\\
&&\quad\ls \int_{B_\az^{\ast\ast}}
\int_{B_\az^\ast}\frac1{V(x,\,y)}\lf[\frac {d(x,\,y)}{\rho(x_\az)}\r]^{\kz\wedge\dz_1}
|(\chi_{B_\az^\ast} f)(y)|\,d\mu(x)\,d\mu(y)
\ls \|\chi_{B_\az^\ast}f\|_{L^1(\cx)}.
\end{eqnarray*}
For any $x\notin B_\az^{\ast\ast}$ and  $t<\rho(x)$, it is easy to see that
$\rho(x_\az)\ls d(x,\,x_\az)\sim d(x,\,y)$ for all $y\in B_\az^\ast$,
and by \eqref{2.2},
$t<\rho(x)\ls[d(x,\,x_\az)]^{k_0/(1+k_0)}[\rho(x_\az)]^{1/(1+k_0)},$
from which it follows that
\begin{eqnarray*}
|E_t(f)(x)|&&\ls \int_{B_\az^\ast}\frac1{V_t(x)+V(x,\,y)}
\lf[\frac t{t+d(x,\,y)}\r]^{\dz_2}|f(y)|\,d\mu(y)\\
&&\ls \frac1{V(x,\,x_\az)}
\lf[\frac {\rho(x_\az)}{d(x,\,x_\az)}\r]^{\dz_2/(1+k_0)}
 \|\chi_{B_\az^\ast}f\|_{L^1(\cx)}.
\end{eqnarray*}
By this, we have
\begin{eqnarray*}
&&\int_{(B_\az^{\ast\ast})^\complement}\sup_{0<t<\rho(x)}
|E_t(\chi_{B_\az^\ast} f)(x)|\,d\mu(x)\\
&&\quad\ls \|\chi_{B_\az^\ast}f\|_{L^1(\cx)}\int_{(B_\az^{\ast\ast})^\complement}
\frac1{V(x,\,x_\az)}\lf[\frac {\rho(x_\az)}{d(x,\,x_\az)}\r]^{\dz_2/(1+k_0)}
\,d\mu(x) \ls\|\chi_{B_\az^\ast}f\|_{L^1(\cx)},
\end{eqnarray*}
which completes the proof of \eqref{4.9} and hence, the proof of Lemma \ref{l4.4}.
\end{proof}

Now we turn to the proof of Theorem \ref{t4.2}.

\begin{proof}[Proof of Theorem \ref{t4.2}]
Assume that $f\in L^1(\cx)$ and
$\|T^+_{\rho}(f)\|_{L^1(\cx)}<\fz$.
Then by Remark \ref{r4.1}, $\|f\|_{L^1(\cx)}\le\|T^+_{\rho}(f)\|_{L^1(\cx)}$.
From this, Remark \ref{r4.2} and Lemma \ref{l4.4}, it follows that
$f\in H^1_\rho (\cx)$ and
\begin{eqnarray*}
\|f\|_{H^1_\rho (\cx)}&&\ls\|\wz T^+_{\rho}(f)\|_{L^1(\cx)}
\le
\|T^+_{\rho}(f)\|_{L^1(\cx)}+\|E^+_{\rho}(f)\|_{L^1(\cx)}\\
&&\ls\|T^+_{\rho}(f)\|_{L^1(\cx)}+\|f\|_{L^1(\cx)}
\ls\|T^+_{\rho}(f)\|_{L^1(\cx)}\ls\|T^+(f)\|_{L^1(\cx)}.
\end{eqnarray*}

Conversely, we need to prove that
$T^+_{\rho}$ and $T^+$ are bounded from $H^1_\rho (\cx)$
to $L^1(\cx)$. To this end, by Proposition \ref{p3.2},
it suffices to prove that for all $(1,\,2)_{\rho}$-atoms $a$,
$$\|T^+_{\rho}(a)\|_{L^1(\cx)}+\|T^+(a)\|_{L^1(\cx)}\ls1.$$

Assume that $a$ is a $(1,\,2)_{\rho}$-atom
supported in $B(y_0,\,r)$ with $r<\rho(y_0)$.
By Theorem \ref{t4.1} and Remark \ref{r4.2},
we have $\|\wz T^+_{\rho}(a)\|_{L^1(\cx)}\ls1.$
By Lemma \ref{l4.4}, we further obtain
$\|E^+_{\rho}(a)\|_{L^1(\cx)}\ls\|a\|_{L^1(\cx)}\ls1$,
which yields $\|T^+_{\rho}(a)\|_{L^1(\cx)}\ls1$.
This also implies that, to show $\|T^+(a)\|_{L^1(\cx)}\ls1$,
it suffices to prove that $\|\sup_{t\ge\rho(\cdot)}|T_t(a)(\cdot)|\|_{L^1(\cx)}\ls1$.
To see this, by the H\"older inequality and the $L^2(\cx)$-boundedness of $T^+$
(see Remark \ref{r4.1} (i)), we have
$$\int_{B(y_0,\,2r)} \sup_{t\ge \rho(x)}|T_t(a)(x)|\,d\mu(x)\ls
\|a\|_{L^2(\cx)}[V_{2r}(y_0)]^{1/2}\ls1.$$
Since for any $x\in (B(y_0,\,4\rho(y_0))\setminus B(y_0,\,2r))$,
$\rho(x)\sim \rho(y_0)$ via Lemma \ref{l2.1} (i),
by assumption (i) of Theorem \ref{t4.2}, we have that for all
$x\in (B(y_0,\,4\rho(y_0))\setminus B(y_0,\,2r))$ and $t\ge\rho(x)$,
\begin{eqnarray*}
|T_t(a)(x)|&&\le\int_\cx |T_t(x,\,y)a(y)|\,d\mu(y)
\ls\frac1{V_t(x)}\ls \frac1{V_{\rho(y_0)}(y_0)}.
\end{eqnarray*}
This implies that
$$\int_{B(y_0,\,4\rho(y_0))\setminus B(y_0,\,2r)}
\sup_{t\ge\rho(x)}|T_t(a)(x)|\,d\mu(x)\ls1.$$
For any $x\notin B(y_0,\,4\rho(y_0))$, since \eqref{2.2} implies
$\rho(x)\ls [d(x,\,y_0)]^{k_0/(1+k_0)}[\rho(y_0)]^{1/(1+k_0)},$
by assumption (i) of Theorem \ref{t4.2}, we have that
\begin{eqnarray*}
|T_t(a)(x)|
&&\ls\frac1{V(x,\,y_0)}
\lf[\frac{\rho(x)}{d(x,\,y_0)}\r]^{\dz_2\wedge\dz_3}
\ls \frac1{V(x,\,y_0)} \lf[\frac{\rho(y_0)}{d(x,\,y_0)}\r]
^{(\dz_2\wedge\dz_3)/(1+k_0)},
\end{eqnarray*}
which implies that
\begin{eqnarray*}
&&\int_{B(y_0,\,4\rho(y_0))^\complement}
\sup_{t\ge\rho(x)}|T_t(a)(x)|\,d\mu(x)\\
&&\quad\ls
\int_{B(y_0,\,4\rho(y_0))^\complement}\frac1{V(x,\,y_0)} \lf[\frac{\rho(y_0)}{d(x,\,y_0)}\r]
^{(\dz_2\wedge\dz_3)/(1+k_0)}\,d\mu(x)\ls1.
\end{eqnarray*}
This shows that $\|\sup_{t\ge\rho(\cdot)}|T_t(a)(\cdot)|\|_{L^1(\cx)}\ls1$ and
hence, finishes the proof of Theorem \ref{t4.2}.
 \end{proof}

\section{Some applications}\label{s5}

In this section, we present several applications of results
in Sections 3 and 4.

\subsection{Schr\"odinger operators on $\rn$}\label{s5.1}

Let $n\ge 3$ and $\rn$ be the $n$-dimensional Euclidean space
endowed with the Euclidean
norm $|\cdot|$ and the Lebesgue measure $dx$.
Denote the Laplace operator $-\sum_{j=1}^n(\frac{\partial}{\partial x_j})^2$
on $\rn$ by $\Delta$ and the corresponding heat semigroup
$\{e^{-t\Delta}\}_{t>0}$ by $\{\wz T_t\}_{t>0}$. By the
Gaussian estimates for the heat kernel and the Markov property
for $\{\wz T_t\}_{t>0}$, we know that
$\{\wz T_{t^2}\}_{t>0}$ forms a continuous $(1,\,N,\,N)$-$\ati$
as in Remark \ref{r2.2} (ii) for any $N>0$.

Let $U$ be a nonnegative locally integrable function
on $\rn$, $\cl\equiv\Delta+U$ the Schr\"odinger operator
and $\{T_t\}_{t>0}\equiv\{e^{-t\cl}\}_{t>0}$  the corresponding
heat semigroup. Define
$$H^1_\cl(\rn)\equiv\{f\in L^1(\rn):\ \|f\|_{H^1_\cl(\rn)}
\equiv\|T^+(f)\|_{L^1(\rn)}<\fz\},$$
where $T^+(f)(x)\equiv\sup_{t>0}|T_t(f)(x)|$ for all $x\in\rn$.

If $q>n/2$ and $U\in \cb_q(\rn,\,|\cdot|,\,dx)$, where
$\cb_q(\rn,\,|\cdot|,\,dx)$ is the reverse H\"older
class as in Subsection \ref{s2.1},
then Dziuba\'nski and Zienkiewicz \cite{dz99} firstly established the
atomic decomposition characterizations of $H^1_\cl(\rn)$
via the auxiliary function $\rho$ defined as in \eqref{2.3}.
In fact, Dziuba\'nski and Zienkiewicz in \cite{dz99} proved that
$H^1_\cl(\rn)=H^{1,\,\fz}_{\rho}(\rn)$ with equivalent norms.
Moreover, in \cite{dz02}, for $f\in L^1(\rn)$ with compact support,
Dziuba\'nski and Zienkiewicz also proved that
$\|f\|_{H^1_\cl(\cx)}\sim \|\wz T^+_{\rho}(f)\|_{L^1(\cx)}$,
where $\wz T^+_{\rho}$ is defined as in Remark \ref{r4.2}
with $\wz S_t$ replaced by $\wz T_{t^2}$.

On the other hand, Proposition \ref{p2.1} implies that
$\rho$ defined in \eqref{2.3} is admissible,
$\{\wz T_{t^2}\}_{t>0}$ is a continuous
$(1,\,N,\,N)_{\rho}$-$\ati$ for any $N>0$,
$\{T_{t^2}\}_{t>0}$ and $\{\wz T_{t^2}\}_{t>0}$ satisfy
the assumptions (i) and (ii) of Theorem \ref{t4.2};
see \cite{dz02}. Thus, by Theorem \ref{t4.2},
$H^1_\cl(\rn)=H^1_\rho (\rn)$ with equivalent norms.
Moreover, all of the results obtained in Sections 3 and 4
are valid for $H^1_\cl(\rn)$. In particular, the results established
in Section \ref{s3} are new
compared to the results in \cite{dz99,dz02}.

\subsection{Degenerate Schr\"odinger operators on  $\rn$}\label{s5.2}

Let $n\ge 3$ and $\rn$ be the $n$-dimensional Euclidean space
endowed with the Euclidean norm $|\cdot|$ and the Lebesgue measure $dx$.
Recall that a nonnegative locally integrable function $w$ is said to
be an $\ca_2(\rn)$ weight in the sense of Muckenhoupt if
$$\sup_{B}\lf\{\frac1{|B|}\int_B w(x)\,dx\r\}^{1/2}
\lf\{\frac1{|B|}\int_B [w(x)]^{-1}\,dx\r\}^{1/2}<\fz,$$
where the supremum is taken over all the balls in $\rn$;
see \cite{m72} and also \cite{s93} for the definition of
$\ca_2(\rn)$ weights and their properties.
Observe that if we set $w(E)\equiv\int_Ew(x)dx$ for any measurable set $E$, then
there exist positive constants $C,\,Q$ and $\kz$
such that for all $x\in\rn$, $\lz>1$ and $r>0$,
$$C^{-1}\lz^\kz w(B(x,\,r))\le w(B(x,\,\lz r))\le C\lz^Qw(B(x,\,r)),$$
namely, the measure $w(x)\,dx$ satisfies \eqref{2.1}.
Thus $(\rn,\,|\cdot|,\,w(x)\,dx)$ is an RD-space.

Let $w\in\ca_2(\rn)$ and
$\{a_{i,\,j}\}_{1\le i,\,j\le n}$ be a real symmetric matrix function satisfying that
for all $x,\,\xi\in\rn$,
$$C^{-1}|\xi|^{2}\le\sum_{1\le i,\,j\le n}a_{i,\,j}(x)\xi_i\overline \xi_j
\le C|\xi|^2.$$ Then the degenerate elliptic operator
$\cl_0$ is defined by
$$\cl_0 f(x)\equiv-\frac1{w(x)}\sum_{1\le i,\,j\le n}
\partial_i(a_{i,\,j}(\cdot)\partial_j f)(x),$$
where $x\in\rn$.
Denote by $\{\wz T_t\}_{t>0}\equiv \{e^{-t\cl_0}\}_{t>0}$
the semigroup generated by $\cl_0$. We also denote the kernel of
$\wz T_t$ by $\wz T_t(x,\,y)$ for all $x,\,y\in\rn$ and $t>0$.
Then it is known that there exist positive constants $C,\,C_7,\,\wz C_7$
and $\az\in(0,\,1]$ such that
for all $t>0$ and $x,\,y\in\rn$,
\begin{equation}\label{5.1}
\frac1{V_{\sqrt t}(x)}\exp\lf\{-\frac{|x-y|^2}{\wz C_7t}\r\}\le
\wz T_t(x,\,y)\le
\frac1{V_{\sqrt t}(x)}\exp\lf\{-\frac{|x-y|^2}{C_7t}\r\};
\end{equation}
that for all $t>0$ and $x,\,y,\,y'\in\rn$ with $|y-y'|<|x-y|/4$,
\begin{equation}\label{5.2}
 |\wz T_t(x,\,y)-\wz T_t(x,\,y')|\le
\frac1{V_{\sqrt t}(x)}\lf(\frac{|y-y'|}{\sqrt t}\r)^\az\exp
\lf\{-\frac{|x-y|^2}{C_7t}\r\};
\end{equation}
and, moreover, that for all $t>0$ and $x,\,y\in\rn$,
\begin{equation}\label{5.3}
\int_\rn \wz T_t(x,\,z)\,w(z)\,dz=1=\int_\rn \wz T_t(z,\,y)\,w(z)\,dz;
\end{equation}
see, for example, Theorems 2.1, 2.7, 2.3 and 2.4, and Corollary 3.4 of \cite{hs01}.

Let $U$ be a nonnegative locally integrable function on $w(x)\,dx$.
Define the  degenerate Schr\"odinger operator by $\cl \equiv\cl_0+U.$
Then $\cl$ generates a semigroup $\{T_t\}_{t>0}\equiv\{e^{-t\cl}\}_{t>0}$ with
kernels $\{T_t(x,\,y)\}_{t>0}$ for
all $x,\,y\in\rn$. By Kato-Trotter's product formula (see \cite{k78}),
$0\le T_t(x,\,y)\le \wz T_t(x,\,y)$ for all $x,\,y\in\rn$ and $t>0$.
Define the radial maximal operator $T^+$ by $T^+(f)(x)\equiv\sup_{t>0}|e^{-t\cl}(f)(x)|$
for all $x\in\rn$.
Then $T^+$ is bounded on $L^p(w(x)\,dx)$ for $p\in(1,\,\fz]$
and from $L^1(w(x)\,dx)$ to
weak-$L^1(w(x)\,dx)$.
The Hardy space associated to $\cl$ is defined by
$$H^1_\cl(w(x)\,dx)\equiv\{f\in L^1(w(x)\,dx):\
\|f\|_{H^1_\cl(w(x),\,dx)}\equiv\|T^+(f)\|_{L^1(w(x)\,dx)}<\fz\}.$$

If $q>Q/2$ and $U\in \cb_q(\rn,\,|\cdot|,\,w(x)\,dx)$, letting
$\rho$ be as in \eqref{2.3},
then Dziuba\'nski \cite{d05} proved that there exists a positive constant
$C_7'$ such that for all $x,\,y\in\rn$,
\begin{equation}\label{5.4}
0\le T_t(x,\,y)\le \frac1{V_{\sqrt t}(x)}
\lf[\frac{\rho(x)}{\rho(x)+|x-y|}\r]^N\exp\lf\{-\frac{|x-y|^2}{C_7't}\r\},
\end{equation}
and
\begin{equation}\label{5.5}
 0\le \wz T_t(x,\,y)-T_t(x,\,y)\le
C\lf[\frac{\sqrt t}{\sqrt t+\rho(x)}\r]^{q-Q/2} \frac1{V_{\sqrt t}(x)}
\exp\lf\{-\frac{|x-y|^2}{C_7't}\r\}.
\end{equation}
By this, Dziuba\'nski \cite{d05}
proved that
$H^1_\cl(w(x)\,dx)=H^{1,\,\fz}_{\rho}(w(x)\,dx)$ with equivalent norms
via using a different theory of Hardy spaces on spaces of homogeneous type from here.

On the other hand, by Proposition \ref{p2.1}, $\rho$ is an admissible function.
From
\eqref{5.1} through \eqref{5.3}, Remark \ref{r2.2} (iii) and the semigroup property,
it follows that $\{\wz T_{t^2}\}_{t>0}$ is a
continuous $(1,\,N,\,N)_{\rho}$-$\ati$ for any $N>0$.
Observe that \eqref{5.4} and \eqref{5.5} implies that
$\{T_{t^2}\}_{t>0}$ and $\{\wz T_{t^2}\}_{t>0}$ satisfy
the assumptions of Theorem \ref{t4.2}.
Thus by Theorem \ref{t4.2} in Section 4,
$H^1_\cl(w(x)\,dx)=H^1_\rho (w(x)\,dx)$.
Moreover, all of the results in Sections 3 and 4
are valid for $H^1_\cl(w(x)\,dx)$. In particular, all the results in Section \ref{s3}
and
Theorem \ref{t4.1} are new
compared to the known results in \cite{d05}.

\subsection{Sub-Laplace Schr\"odinger operators on Heisenberg groups}\label{s5.3}

The $(2n+1)$-dimensional Heisenberg group $\hh^n$ is a
connected and simply connected Lie group with underlying manifold
$\rr^{2n}\times \rr$ and the multiplication
$$(x,\,t)(y,\,s)=\lf(x+y,\,t+s+2\sum_{j=1}^n[x_{n+j}y_j-x_jy_{n+j}]\r).$$
The homogeneous norm on $\hh^n$ is defined by
$|(x,\,t)|=(|x|^4+|t|^2)^{1/4}$ for all $(x,\,t)\in\hh^n$,
which induces a left-invariant metric
$d((x,\,t),\,(y,\,s))=|(-x,\,-t)(y,\,s)|$.
Moreover, there exists a positive constant $C$ such that
$|B((x,\,t),\,r)|=Cr^Q,$
where $Q\equiv(2n+2)$ is the homogeneous dimension of $\hh^n$ and
$|B((x,\,t),\,r)|$ is the Lebesgue measure of the
ball $B((x,\,t),\,r)$. The triplet $(\hh^n,\, d,\,dx)$ is an
RD-space.

A basis for the Lie algebra of Left-invariant vector fields on $\hh^n$ is given by
$$X_{2n+1}=\frac{\partial}{\partial t},
\quad X_j=\frac{\partial}{\partial x_j}+2x_{n+j}\frac{\partial}{\partial t},
\quad X_{n+j}= \frac{\partial}{\partial x_{n+j}}-2x_j\frac{\partial}{\partial t},\
j=1,\,\cdots,n.$$
All non-trivial commutators are
$[X_j,\,X_{n+j}]=4X_{2n+1}$,  $j=1,\,\cdots,\,n$.
The sub-Laplacian has the form
$\Delta_{\hh^n}=-\sum_{j=1}^{2n}X_j^2.$
See \cite{fs82,s93} for the theory of the Hardy spaces
associated to the sub-Laplacian $\Delta_{\hh^n}$.

Let $U$ be a nonnegative locally integrable function on $\hh^n$.
Define the sub-Laplacian Schr\"odinger operator by
$\cl\equiv\Delta_{\hh^n}+U.$
Denote by $\{T_t\}_{t>0}\equiv \{e^{-t\cl}\}_{t>0}$
the semigroup generated by $\cl$.
Define the Hardy space associated to $\cl$ by
$$H^1_\cl(\hh^n)\equiv\{f\in L^1(\hh^n):\,
\|f\|_{H^1_\cl(\hh^n)}\equiv\|T^+(f)\|_{L^1(\hh^n)}<\fz\}.$$

If $q>Q/2$ and $U\in \cb_q(\hh^n)$, then  C. Lin, H. Liu and Y. Liu \cite{lll}
proved that $H^1_\cl(\hh^n)=H^{1,\,q}_{\rho}(\hh^n)$
with equivalent norms for all $q\in(1,\,\fz]$,
where $\rho$ is as in \eqref{2.3}.

On the other hand, Proposition \ref{p2.1} implies
that $\rho$ is an admissible function.
It is easy to check that $\{\wz T_{t^2}\}_{t>0}
\equiv\{e^{-t^2\Delta_{\hh^n}}\}_{t>0}$ is a continuous
$(1,\,N,\,N)_{\rho}$-$\ati$ for any $N$; see, for example, \cite{fs82}.
C. Lin, H. Liu and Y. Liu \cite{lll} proved that
$\{T_{t^2}\}_{t>0}$ and $\{\wz T_{t^2}\}_{t>0}$ satisfy
the assumptions of Theorem \ref{t4.2}.
Thus applying Theorem \ref{t4.2}, we obtain
$H^1_\cl(\hh^n)=H^1_\rho (\hh^n)$ with equivalent norms.
Moreover, all the results in Sections 3 and 4
are valid for $H^1_\cl(\hh^n)$. In this case, comparing to the
results obtained in \cite{lll},
Theorem \ref{t3.1}, Theorem \ref{t3.2} (ii),
 Proposition \ref{p3.3} and Theorem \ref{t4.1} are new.

\subsection{Sub-Laplace Schr\"odinger operators on connected and
simply connected nilpotent Lie groups}\label{s5.4}

Let $\mathbb G$ be a connected and simply connected nilpotent Lie group.
Let $X\equiv\{X_1,\,\cdots,\,X_k\}$ be left invariant vector fields on $\bbg$
satisfying the H\"ormander condition, namely, $X$
together with their commutators of order $\le m$ generates the tangent space of $\bbg$
at each point of $\bbg$. Let $d$ be the Carnot-Carath\'eodory (control) distance
on $\bbg$ associated to $X$.
Fix a left invariant Haar measure $\mu$ on $\bbg$.
Then for all $x\in \bbg$, $V_r(x)=V_r(e)$, and moreover,
there exist $0<\kz\le D<\fz$ such that for all $x\in\bbg$,
$C^{-1} r^\kz\le V_r(x)\le Cr^\kz$ when $0<r\le1$,
and $C^{-1} r^D\le V_r(x)\le Cr^D$ when $r>1$;
see \cite{nsw85},
\cite{v88} and \cite{vsc92} for the details. Thus $({\mathbb \bbg},\,d,\,\mu)$ is
an RD-space.

The sub-Laplacian is given by
$\Delta_\bbg\equiv-\sum_{j=1}^kX_j^2.$
Denote by $\{\wz T_t\}_{t>0}\equiv\{e^{-t\Delta_\bbg}\}_{t>0}$
the semigroup generated by $\Delta_\bbg$.
Then there exist positive constants $C,\,C_8$ and $\wz C_8$
such that for all $t>0$ and $x,\,y\in\bbg$,
\begin{equation}\label{5.6}
C^{-1}\frac1{V_{\sqrt t}(x)}\exp\lf\{-\frac{[d(x,\,y)]^2}{\wz C_8t}\r\}\le
\wz T_t(x,\,y)\le C\frac1{V_{\sqrt t}(x)}\exp\lf\{-\frac{[d(x,\,y)]^2}{C_8t}\r\},
\end{equation}
that for all $t>0$ and $x,\,y,\,y'\in\cx$ with $d(y,\,y')\le d(x,\,y)/4$,
\begin{equation}\label{5.7}
 |\wz T_t(x,\,y)-\wz T_t(x,\,y')|\le C\frac{d(y,\,y')}{\sqrt t}
\frac1{V_{\sqrt t}(x)}\exp\lf\{-\frac{[d(x,\,y)]^2}{C_8t}\r\},
\end{equation}
and moreover, that for all $t>0$ and $x,\,y\in\bbg$,
\begin{equation}\label{5.8}
 \int_\bbg \wz T_t(x,\,z)\,d\mu(z)=1=\int_\bbg \wz T_t(z,\,y)\, d\mu(z);
\end{equation}
see, for example, \cite{v88} and \cite{vsc92} for the details.

Define the radial maximal operator $\wz T^+$ by
$\wz T^+(f)(x)\equiv\sup_{t>0}|\wz T_t(f)(x)|$ for all $x\in\bbg$.
Then $\wz T^+$ is bounded on $L^p(\bbg)$ for $p\in(1,\,\fz]$ and from $L^1(\bbg)$ to
weak-$L^1(\bbg)$. The Hardy space associated to $\Delta_{\bbg}$ is defined by
$$H^1(\bbg)\equiv\lf\{f\in L^1(\bbg):\
\|f\|_{H^1(\bbg)}\equiv\|\wz T^+(f)\|_{L^1(\bbg)}<\fz\r\};$$
see, for example, \cite{s86,s87,s90,hmy2}
for the theory of Hardy spaces associated with the sub-Laplace
operator $\Delta_\bbg$.

Let $U$ be a nonnegative locally integrable function on $\bbg$.
Then the sub-Laplace Schr\"odinger operator is defined by
$\cl\equiv\Delta_\bbg+U.$ The operator  $\cl$ generates a semigroup
$\{T_t\}_{t>0}\equiv\{e^{-t\cl}\}_{t>0}$, whose kernels
are denoted by $\{T_t(x,\,y)\}_{t>0}$ for all $x,\,y\in\bbg$.
By Kato-Trotter's product formula (see \cite{k78}),
$0\le T_t(x,\,y)\le \wz T_t(x,\,y)$ for all $t>0$ and $x,\,y\in\bbg$.
Define
the radial maximal operator $T^+$ by
$T^+(f)(x)=\sup_{t>0}|e^{-t\cl}(f)(x)|$ for all $x\in\bbg$.
Then $T^+$ is bounded on $L^p(\bbg)$ for $p\in(1,\,\fz]$ and from $L^1(\bbg)$ to
weak-$L^1(\bbg)$.
The Hardy space associated to $\cl$ is defined by
$$H^1_\cl(\bbg)\equiv\{f\in L^1(\bbg):\
\|f\|_{H^1_\cl(\bbg)}\equiv\|T^+(f)\|_{L^1(\bbg)}<\fz\}.$$

Let $q>D/2$,  $U\in \cb_q(\bbg,\,d,\,\mu)$ and
$\rho(x)$ for all $x\in\bbg$
be as in \eqref{2.3}. Then Li \cite{l99}
established some basic results concerning $\cl$, which
include estimates for fundamental solutions of $\cl$ and
the boundedness on Lebesgue spaces of some operators associated to $\cl$.
To apply the results obtained in Sections 3 and 4 to $\cl$,
we need the following estimates.

\begin{prop}\label{p5.1}
Let $q> D/2$ and $U\in \cb_q(\bbg,\,d,\,\mu)$. Then for each $N\in\nn$, there exists
a positive constant $C$ such that for all $f\in L^2(\bbg)$ and $x\in\bbg$,
\begin{equation}\label{5.9}
|[\rho(x)]^{-2N}\cl^{-N}f(x)|\le C M^{N}(f)(x),
\end{equation}
where $M$ denotes the Hardy-Littlewood maximal operator on $\bbg$ and
$ M^N\equiv M\circ\cdots\circ M$.
\end{prop}

\begin{proof}
Let $G_0$ be the kernel of $\cl^{-1}$.
For any fixed $N>0$, by Theorem 3.6 of \cite{l99},
there exists a positive constant $C$
such that for  all $x,\,y\in\bbg$,
\begin{equation}\label{5.10}
0\le G_0(x,\,y)\le C \frac{[d(x,\,y)]^2}{[1+d(x,\,y)/\rho(x)]^{N}}\frac1{V(x,\,y)}.
\end{equation}
By this, for all $x\in\bbg$, we have
\begin{eqnarray*}
&&|[\rho(x)]^{-2} \cl^{-1}f(x)|\\
&&\quad\ls [\rho(x)]^{-2}\int_\bbg G_0(x,\,y)|f(y)|\,d\mu(y)\\
&&\quad\ls\int_\bbg \frac1{V(x,\,y)}
\frac{[d(x,\,y)/\rho(x)]^2}{[1+d(x,\,y)/\rho(x)]^{N}}
|f(y)|\,d\mu(y)\\
&&\quad\ls \lf(\sum_{j=1}^\fz 2^{-j(N-2)}+\sum_{j=-\fz}^0 2^{2j}\r)
\frac1{V_{2^{j+1}\rho(x)}(x)}
\int_{d(x,\,y)\le 2^{j+1}\rho(x)}
|f(y)|\,d\mu(y)\ls M(f)(x),
\end{eqnarray*}
where $N>2$.
Since \eqref{2.2} and Lemma \ref{l2.1} imply that
$$\frac1{\rho(y)}\gs \frac1{\rho(x)}\lf[1+\frac{d(x,\,y)}{\rho(x)}\r]^{-k_0/(1+k_0)},$$
 thus, by \eqref{5.10}, we have
\begin{eqnarray*}
|[\rho(x)]^{-4} \cl^{-2}f(x)|
&&\ls [\rho(x)]^{-4}\int_\bbg G_0(x,\,y)|\cl^{-1}f(y)|\,d\mu(y)\\
&&\ls\int_\bbg \frac1{V(x,\,y)}
\frac{[d(x,\,y)/\rho(x)]^2}{[1+d(x,\,y)/\rho(x))]^{N}}
|M(f)(y)|\,d\mu(y)\ls M^2(f)(x).
\end{eqnarray*}
Repeating the above arguments then completes the proof of Proposition \ref{p5.1}.
\end{proof}

\begin{prop}\label{p5.2}
If $q>D/2$ and $U\in \cb_q(\bbg,\,d,\,\mu)$,
then for any $N>0$, there exist positive constants $C$ and $C_8'$ such that
for all $t>0$ and $x,\,y\in\bbg$,
$$0\le T_t(x,\,y)\le C\frac1{V_{\sqrt t}(x)}
\lf[\frac{\rho(x)}{\rho(x)+d(x,\,y)}\r]^N\exp\lf\{-\frac{[d(x,\,y)]^2}{C_8't}\r\}.$$
\end{prop}

\begin{proof}
By \eqref{5.6} and $0\le T_t(x,\,y)\le\wz T_t(x,\,y)$
for all $t>0$ and $x,\,y\in\bbg$, we have
\begin{equation}\label{5.11}
0\le T_t(x,\,y)\le\wz T_t(x,\,y)\ls
\frac1{V_{\sqrt t}(x)}
\exp\lf\{-\frac{[d(x,\,y)]^2}{C_8t}\r\}.
\end{equation}
To prove Proposition \ref{p5.2}, it suffices to prove that for $t\ge C[\rho(x)]^2$,
\begin{equation}\label{5.12}
T_t(x,\,y)\ls
\frac1{V_{\sqrt t}(x)} \lf[\frac{\rho(x)}{\sqrt t}\r]^N.
\end{equation}
In fact, if this holds, then for
$t\ge C[\rho(x)]^2$,
\begin{eqnarray*}
T_t(x,\,y)&&\ls
\frac1{V_{\sqrt t}(x)}\lf[\frac{\rho(x)}{\sqrt t}\r]^N
\ls \frac1{V_{\sqrt t}(x)}\lf[\frac{\rho(x)}{\rho(x)+d(x,\,y)}\r]^N
\lf[\frac{\sqrt t+d(x,\,y)}{\sqrt t}\r]^{N},
\end{eqnarray*}
which together with \eqref{5.11} via the geometric mean yields the desired conclusion.
For $t\le C[\rho(x)]^2$,
since the function $f(t)=\frac{t}{t+a}$ is increasing in $t$,
by \eqref{5.11}, we have
\begin{eqnarray*}
T_t(x,\,y)
 \ls
\frac1{V_{\sqrt t}(x)}
\lf[\frac{\rho(x)}{\rho(x)+d(x,\,y)}\r]^N\exp\lf\{-\frac{[d(x,\,y)]^2}{Ct}\r\}.
\end{eqnarray*}

To prove \eqref{5.12}, observe that $\cl$ is self-adjoint.
For any $f\in L^2(\bbg)$, by the well-known spectral theorem, we have
\begin{equation}\label{5.13}
\|\partial^N_tT_t(f)\|_{L^2(\bbg)}= t^{-N} \|(t\cl)^Ne^{-t\cl}f\|_{L^2(\bbg)}
\le C(N)t^{-N}\|f\|_{L^2(\bbg)}.
\end{equation}
Set $T^{(N)}_t(x,\,y)\equiv\partial^N_s T_s(x,\,y)\Big|_{s=t}$
for all $x,\,y\in\bbg$.
Notice that
\begin{eqnarray*}
T_{2t}^{(N)}(x,\,y)&&=\partial^N_s T_s(x,\,y)\Big|_{s=2t}=
\partial^N_sT_{t+s}(x,\,y)\Big|_{s=t}=
\partial_sT_s(T_t(x,\,\cdot))(y)\Big|_{s=t}\\
&&=
\lf(\partial^N_sT_s\Big|_{s=t}\r)(T_t(x,\,\cdot))(y)
=(\partial^N_tT_t)(T_t(x,\,\cdot))(y),
\end{eqnarray*}
which together with \eqref{5.13} and \eqref{5.11} implies that
\begin{eqnarray*}
\|T_{2t}^{(N)}(x,\,\cdot)\|_{L^2(\bbg)}&&\ls \frac1{t^N}\|T_t(x,\,\cdot)\|_{L^2(\bbg)}\\
&&\ls \frac1{t^N}\lf\{\int_\bbg \frac1{[V_{\sqrt t}(x)]^2}
\exp\lf\{-\frac{[d(x,\,y)]^2}{ct}\r\}\,d\mu(y)\r\}^{1/2}\\
&&\ls \frac1{t^N}\frac1{[V_{\sqrt t}(x)]^{1/2}}
\lf\{\sum_{j=0}^\fz 2^{jD/2}\exp\{-2^j\}\r\}^{1/2}
\ls \frac1{t^N}\frac1{[V_{\sqrt t}(x)]^{1/2}}.
\end{eqnarray*}
This together with the H\"older inequality, \eqref{5.13} and \eqref{5.11}
again further yields that
\begin{eqnarray*}
|T_{2t}^{(N)}(x,\,y)|&&\ls\|T^{(N)}_t(x,\,\cdot)\|_{L^2(\bbg)}
\|T_t(x,\,\cdot)\|_{L^2(\bbg)}
\ls \frac1{ t^N}\frac1{V_{\sqrt t}(x)}.
\end{eqnarray*}
Thus, by Proposition \ref{p5.1}, we have
\begin{eqnarray*}
T_t(x,\,y)&&=[\rho(x)]^{2N}\lf|[\rho(x)]^{-2N}\cl_y^{-N}\cl_y^N(T_t(x,\,\cdot))(y)\r|\\
&&=[\rho(x)]^{2N}\lf|[\rho(x)]^{-2N}\cl_y^{-N}\partial_t^{(N)}(T_t(x,\,\cdot))(y)\r|\\
&&\ls[\rho(x)]^{2N}M^N\lf(\partial_t^{(N)} T_t(x,\,\cdot)\r)(y)\\
&&\ls[\rho(x)]^{2N}\lf\|\partial_t^{(N)} T_t(x,\,\cdot)\r\|_{L^\fz(\bbg)}
\ls \frac1{V_{\sqrt t}(x)}\lf[\frac{\rho(x)}{\sqrt t}\r]^{2N},
\end{eqnarray*}
which implies \eqref{5.12} and hence, completes the proof of Proposition \ref{p5.2}.
\end{proof}

For $t\ge 0$, set $E_t\equiv\wz T_t-T_t.$
Denote also by $E_t$ the kernel of $E_t$. Then for all $x,\,y\in\bbg$,
\begin{equation}\label{5.14}
E_t(x,\,y)=\wz T_t(x,\,y)-T_t(x,\,y)=
\int_0^t\int_\bbg \wz T_s(x,\,z)U(z) T_{t-s}(z,\,y)\,d\mu(z)\,ds;
\end{equation}
see, for example, \cite{d05}. To estimate $E_t$, we need the following
estimate.

\begin{lem}\label{l5.1}
If $q>D/2$ and $U\in \cb_q(\bbg,\,d,\,\mu)$, then for any
positive constants $\wz C$ and $C'$, there exists a
positive constant $C$ such that for all
$x\in\bbg$ and $t>0$ with $\sqrt t<\wz C\rho(x)$,
$$\int_\bbg\frac{U(z)}{V_{\sqrt t}(x)}
\exp\lf\{-\frac{[d(x,\,z)]^2}{C't}\r\}\,d\mu(z)
\le C \frac1t\lf[\frac{\sqrt t}{\rho(x)}\r]^{2-D/q}.$$
\end{lem}

\begin{proof} We first recall that Li in \cite[Lemma 2.8]{l99} proved that
there exists a positive constant $\ell$ such that
for all $x\in\bbg$ and $R\ge\rho(x)$,
\begin{equation*}
\frac{R^2}{V_R(x)}\int_{B(x,\,R)}U(z)\,d\mu(z)\ls\lf[\frac R{\rho(x)}\r]^\ell,
\end{equation*}
which is also easy to be deduced from \eqref{2.5} and \eqref{2.6}.
By this, \eqref{2.4} and \eqref{2.6},
letting $j_0\in\nn$ such that $2^{j_0-1}\le \wz C\rho(x)/\sqrt t<2^{j_0}$,
we then have
\begin{eqnarray*}
&&\int_\bbg\frac{U(z)}{V_{\sqrt t}(x)}
\exp\lf\{-\frac{[d(x,\,z)]^2}{C't}\r\}\,d\mu(z)\\
&&\quad\ls\sum_{j=0}^\fz\frac1{V_{\sqrt t}(x)} e^{-2^j}
\int_{d(x,\,z)<\sqrt {2C'2^jt}} U(z)\,d\mu(z)\\
&&\quad\ls\sum_{j=1}^{j_0}\frac1{2^jt}e^{-2^j}\lf[\frac{\sqrt{2^jt}}{\rho(x)}\r]^{2-D/q}
+ \sum_{j=j_0+1}^\fz\frac1{2^jt}e^{-2^j}
\lf[\frac {\sqrt{2^jt}}{\rho(x)}\r]^\ell
\ls \frac1t\lf[\frac{\sqrt t}{\rho(x)}\r]^{2-D/q},
\end{eqnarray*}
which completes the proof of Lemma \ref{l5.1}.
\end{proof}

\begin{prop}\label{p5.3}
If $q>D/2$ and $U\in \cb_q(\bbg,\,d,\,\mu)$, then
for each $N>0$, there exist positive constants $C$ and $C_9$ such that
for all $t>0$ and $x,\,y\in \bbg$,
\begin{equation}\label{5.15}
0\le E_t(x,\,y)\le
C\lf[\frac{\sqrt t}{\sqrt t+\rho(x)}\r]^{2-D/q} \frac1{V_{\sqrt t}(x)}
\exp\lf\{-\frac{[d(x,\,y)]^2}{C_9t}\r\}.
\end{equation}
\end{prop}

\begin{proof}
By \eqref{5.11}, we have
$$0\le E_t(x,\,y)\ls \frac1{V_{\sqrt t}(x)}
\exp\lf\{-\frac{[d(x,\,y)]^2}{Ct}\r\}.$$
Thus, if $\sqrt t\ge C\rho(x)$, then \eqref{5.15} follows from this estimate.
If $\sqrt t\ge C\rho(y)$ and $\sqrt t\le C\rho(x)$,  then by \eqref{2.2} and
Lemma \ref{l2.1} (ii), we obtain
\begin{equation*}
1\ls \frac{\sqrt t}{\rho(y)}\ls \frac{\sqrt t}{\rho(x)}
\lf[\frac{\rho(x)+d(x,\,y)}{\rho(x)}\r]^{k_0}
\ls \frac{\sqrt t}{\rho(x)} \lf[\frac{\sqrt t+d(x,\,y)}{\sqrt t}\r]^{k_0},
\end{equation*}
which implies that
\begin{eqnarray*}
E_t(x,\,y)&&\ls\lf[\frac{\sqrt t}{\rho(x)}\r]^{2-D/q}\frac1{V_{\sqrt t}(x)}
\exp\lf\{-\frac{[d(x,\,y)]^2}{Ct}\r\}.
\end{eqnarray*}
If $t\le C[\rho(x)\wedge\rho(y)]$, then set
$W_1\equiv\{z\in \bbg:\ d(z,\,x)\ge d(x,\,y)/2\}$ and
$W_2\equiv\bbg\setminus W_1$. By \eqref{5.14} and \eqref{5.11}, we have
\begin{eqnarray*}
E_t(x,\,y)&&\le \lf\{\int_0^{t/2}\int_{W_1}+
\int_0^{t/2}\int_{W_2}
+\int_{t/2}^t\int_{W_1}
+\int_{t/2}^t\int_{W_2}\r\}
\frac1{V_{\sqrt s}(x)}\\
&&\quad\times
\exp\lf\{-\frac{[d(x,\,z)]^2}{Cs}\r\}
\frac{U(z)}{V_{\sqrt{t-s}}(y)}
\exp\lf\{-\frac{[d(z,\,y)]^2}{C(t-s)}\r\}
\,d\mu(z)\,ds\\
&&\equiv Z_1+Z_2+Z_3+Z_4.
\end{eqnarray*}
Notice that if $0<s<t/2$, then $t-s\sim t$. Then, by Lemma \ref{l5.1},
we obtain
\begin{eqnarray*}
Z_1&&\ls\frac1{V_{\sqrt t}(y)}\exp\lf\{-\frac{[d(x,\,y)]^2}{Ct}\r\}\int_0^{t/2}
\int_{W_1}\frac{U(z)}{V_{\sqrt s}(x)}
\exp\lf\{-\frac{[d(x,\,z)]^2}{Cs}\r\}\,d\mu(z)\,ds\\
 &&\ls\frac1{V_{\sqrt t}(y)}\exp\lf\{-\frac{[d(x,\,y)]^2}{Ct}\r\}
 \int_0^{t/2}\frac1s\lf[\frac{\sqrt s}{\rho(x)}\r]^{2-D/q}\,ds\\
 &&\ls \lf[\frac{\sqrt t}{\rho(x)}\r]^{2-D/q}\frac1{V_{\sqrt t}(x)}
\exp\lf\{-\frac{[d(x,\,y)]^2}{Ct}\r\}.
 \end{eqnarray*}
If $0<s<t/2$ and $z\in W_2$, then $t\sim t-s$
and $d(y,\,z)\ge d(x,\,y)-d(x,\,z)\ge d(x,\,y)/2$,
which together with Lemma \ref{l5.1} imply that
\begin{eqnarray*}
Z_2&&\ls\frac1{V_{\sqrt t}(y)}
\exp\lf\{-\frac{[d(x,\,y)]^2}{Ct}\r\}\int_0^{t/2}\int_{W_2}\frac{U(z)}{V_{\sqrt s}(x)}
\exp\lf\{-\frac{[d(x,\,z)]^2}{Cs}\r\}\,d\mu(z)\,ds\\
&&\ls\lf[\frac{\sqrt t}{\rho(x)}\r]^{2-D/q}\frac1{V_{\sqrt t}(x)}
\exp\lf\{-\frac{[d(x,\,y)]^2}{Ct}\r\}.
\end{eqnarray*}
The estimates for $Z_3$ and $Z_4$ are similar and
we omit the details, which completes the proof of Proposition \ref{p5.3}.
\end{proof}

By Proposition \ref{p2.1}, $\rho$ as in \eqref{2.3} is an admissible function.
From \eqref{5.6}, \eqref{5.7}, \eqref{5.8} and Remark \ref{r2.2} (iii) together with
the semigroup property of $\{\wz T_t\}_{t>0}$, it follows that, for any $N>0$,
$\{\wz T_{t^2}\}_{t>0}$ is a $(1,\,N,\,N)$-$\ati$.
Moreover, Propositions \ref{p5.2} and \ref{p5.3} imply that
the assumptions (i) and (ii) of
Theorem \ref{t4.2} hold for $\{\wz T_{t^2}\}_{t>0}$ and $\{T_{t^2}\}_{t>0}$.
Applying the results obtained in Sections \ref{s3} and \ref{s4}
directly to $\cl$, we have the following conclusions.

\begin{thm}\label{t5.1}
Let $q>D/2$ and $U\in \cb_q(\bbg,\,d,\,\mu)$. If $f\in H^1_\cl(\bbg)$,
then $f\in L^1(\bbg)$ and  $K_\rho (f)-f\in H^1(\bbg)$;
moreover, there exists a positive constant $C$ such that for all
$f\in H^1_\cl(\bbg)$,
$\|K_\rho (f)-f\|_{H^1(\bbg)}\le C\|f\|_{H^1_\rho (\bbg)}.$
 \end{thm}

\begin{thm}\label{t5.2}
If $q>D/2$ and $U\in \cb_q(\bbg,\,d,\,\mu)$, then
the following are equivalent:

\noindent (i) $f\in H^1_\cl(\bbg)$;

\noindent(ii) $f,\, T^+_{\rho}(f)\in L^1(\bbg)$, where $T^+_{\rho}$
is defined as in Remark \ref{r4.2} with $\wz S_t$ replaced by $T_{t^2}$;

\noindent(iii) $f\in H^1_\rho (\bbg)$;

\noindent(iv) there exists $r\in(1,\,\fz]$ such that
$f\in H^{1,\,r}_{\rho}(\bbg)$;

\noindent(v) there exist $\ez\in(0,\,1)$ and $\bz,\,\gz\in(0,\,\ez)$ such that
$f\in (\cg^\ez_0(\bz,\,\gz))'$ and
$\wz T^+_{\rho}(f)\in L^1(\bbg)$,
where $\wz T^+_{\rho}$
is defined as in Remark \ref{r4.2} with $\wz S_t$ replaced by $\wz T_{t^2}$.

\noindent Moreover, if $r\in(1,\,\fz]$, then for all $f\in L^1(\bbg)$,
$$\|f\|_{H^1_\cl(\bbg)}\sim \|T^+_{\rho}(f)\|_{L^1(\bbg)}
\sim\|\wz T^+_{\rho}(f)\|_{L^1(\bbg)}\sim
\|G_{\rho}(f)\|_{L^1(\bbg)}\sim \|f\|_{H^{1,\,r}_{\rho}(\bbg)}.$$
\end{thm}

\begin{thm}\label{t5.3}
Let $q>D/2$ and $U\in \cb_q(\bbg,\,d,\,\mu)$. If $r\in(1,\,\fz]$, then
$\|\cdot\|_{H^{1,\,r}_{\rho,\,\fin}(\bbg)}$ and $\|\cdot\|_{H^1_\cl(\bbg)}$
are equivalent on $H^{1,\,r}_{\rho,\,\fin}(\bbg)$.
\end{thm}

Moreover, applying Proposition \ref{p3.2} to the Riesz transforms $\nabla\cl^{-1}$,
we have the following conclusion.

\begin{thm}\label{t5.4}
If $q\in(D/2,\,D)$ and $U\in \cb_q(\bbg,\,d,\,\mu)$, then
Riesz transforms $\nabla \cl^{-1/2}$ are bounded from $H^1_\cl(\bbg)$ to $L^1(\bbg)$.
\end{thm}

\begin{proof}
It has been proved in \cite[Theorem C]{l99} that
$\nabla \cl^{-1/2}$ is bounded on $L^{p_1}(\bbg)$
for any $p_1\in(1,\,p)$, where $1/p=1/q-1/D$.
By Proposition \ref{p3.2}, it suffices to prove that for all
$(1,\,2)_{\rho}$-atoms $a$,
$\|\nabla\cl^{-1/2}(a)\|_{L^1(\bbg)}\ls1.$

Let $K$ and $\wz K$ be the integral kernels of $\nabla\cl^{-1/2}$
and $\nabla\Delta_\bbg^{-1/2}$, respectively.
 Let $1/p_1+1/p_1'=1$.
Then Li \cite{l99} proved that for all $f\in L^{p_1'}_\loc(\bbg)$ and $x\in\bbg$,
\begin{equation}\label{5.16}
 \int_{d(x,\,y)>\rho(x)} |K(y,\,x)||f(y)|\,d\mu(y)\ls [M(|f|^{p_1'})(x)]^{1/p_1'}
\end{equation}
(see Lemma 6.1, Corollary 6.2 and their proofs therein), and that
\begin{equation}\label{5.17}
\int_{d(x,\,y)\le \rho(x)} |K(y,\,x)-\wz K(y,\,x)||f(y)|\,d\mu(y)
\ls [M(|f|^{p_1'})(x)]^{1/p_1'}
\end{equation}
(see Lemma 6.4 and its proofs therein).
Let $C$ be a positive constant such that
$1/2\le C_3C^{-1/(1+k_0)}(1+1/C)^{k_0/(1+k_0)}<1$.
If $d(x,\,y)>C\rho(x)$, then by \eqref{2.2}, we have
\begin{equation}\label{5.18}
 \rho(y)\le C_3C^{-1/(1+k_0)}(1+1/C)^{k_0/(1+k_0)}d(x,\,y)< d(x,\,y).
 \end{equation}
Let $\eta$ be as in \eqref{2.8} and for all $x\in\bbg$, set
\begin{eqnarray*}
A_1(f)(x)&&\equiv\int_\bbg K(x,\,y)\lf[1-\eta\lf(\frac{d(x,\,y)}{C\rho(x)}\r)\r]f(y)\,d\mu(y),\\
A_2(f)(x)&&\equiv\int_\bbg[K(x,\,y)-\wz K(x,\,y)]\eta\lf(\frac{d(x,\,y)}{C\rho(x)}\r)f(y)\,d\mu(y),
\end{eqnarray*}
and
$$A_3(f)(x)\equiv T(f)-A_1(f)(x)-A_2(f)(x)\\
\equiv\int_\bbg\wz K(x,\,y)\eta\lf(\frac{d(x,\,y)}{C\rho(x)}\r)f(y)\,d\mu(y).$$
Then by \eqref{5.18},
$$|A_1(f)(x)|\le \int_{d(x,\,y)>C\rho(x)} |K(x,\,y)|
|f(y)|\,d\mu(y)\le \int_{d(x,\,y)>\rho(y)} |K(x,\,y)|
|f(y)|\,d\mu(y),$$
which together with the
 duality, \eqref{5.16} and the boundedness of the Hardy-Littlewood maximal operator
 $M$ implies that $\|A_1(f)\|_{L^{p_1}(\bbg)}\ls \|f\|_{L^{p_1}(\cx)}$
for any $p_1\in(1,\,p)$. Moreover,
for all $(1,\,2)_\rho$-atoms $a$, by \eqref{5.16},
\begin{eqnarray*}
\|A_1(a)\|_{L^1(\cx)}\ls
\int_\bbg |a(y)|\int_{d(x,\,y)>\rho(y)} |K(x,\,y)|\,d\mu(x)\,d\mu(y)
\ls\|a\|_{L^1(\bbg)}\ls1,
\end{eqnarray*}
which together with Proposition \ref{p3.2} implies that $A_1$ is bounded from
$H^1_\rho (\bbg)$ to $L^1(\bbg)$.

Similarly, we have
\begin{eqnarray*}
|A_2(f)(x)|&&\le \int_{d(x,\,y)\le\rho(y)} |K(x,\,y)-\wz K(x,\,y)|
|f(y)|\,d\mu(y)\\
&&\quad + \int_{\rho(y)\le d(x,\,y)<2C\rho(x)} |K(x,\,y)||f(y)|\,d\mu(y)\\
&&\quad+\int_{\rho(y)\le d(x,\,y)<2C\rho(x)}\frac1{V(x,\,y)}|f(y)|\,d\mu(y).
\end{eqnarray*}
By duality, the boundedness of $M$ and $\rho(x)\sim \rho(y)$ when $d(x,\,y)<\rho(x)$,  we have
$\|A_2(f)\|_{L^{p_1}(\bbg)}\ls \|f\|_{L^{p_1}(\cx)}$ for any $p_1\in(1,\,p)$.
Moreover, for all $(1,\,2)_{\rho}$-atoms $a$, by Lemma \ref{l2.1} (i),
\eqref{5.15} and \eqref{5.17}, we have
\begin{eqnarray*}
\|A_2(a)\|_{L^1(\cx)}&&\ls
\int_\bbg |a(y)|\int_{d(x,\,y)<\rho(y)} |K(x,\,y)-\wz K(x,\,y)|\,d\mu(x)\,d\mu(y)\\
&&\quad+\int_\bbg |a(y)|\int_{\rho(y)\le d(x,\,y)<2C\rho(x)}
|K(x,\,y)|\,d\mu(x)\,d\mu(y)\\
&&\quad+\int_\bbg |a(y)|\int_{\rho(y)\le d(x,\,y)<2C\rho(x)}
\frac1{V(x,\,y)}\,d\mu(x)\,d\mu(y)
\ls\|a\|_{L^1(\bbg)}\ls1,
\end{eqnarray*}
 which together with Proposition \ref{p3.2} implies that $A_2$ is bounded from
$H^1_\rho (\bbg)$ to $L^1(\bbg)$.
Here we used the fact that $\wz K$ satisfies (K1); see \cite{lv85,a92}.

Obviously,  for any $p_1\in(1,\,p)$, $A_3$ is bounded on $L^{p_1}(\cx)$. Moreover,
for any $f\in L^\fz_b(\bbg)$ and $x\notin\supp f$,
$$A_3(f)(x)
=\int_\bbg \wz K(x,\,y)\eta\lf(\frac{d(x,\,y)}{C\rho(x)}\r)f(y)\,d\mu(y).$$
Since $\wz K$ satisfies the conditions (K1) and (K2) (see \cite{lv85,a92} again),
by Proposition \ref{p3.3}, $A_3$ is also bounded from $H^1_\rho (\bbg)$
to $L^1(\bbg)$. Thus, by Theorem \ref{5.2}, $\nabla\cl^{-1/2}$ are bounded
from $H^1_\cl(\bbg)$ to $L^1(\bbg)$,
which completes the proof of Theorem \ref{t5.4}.
\end{proof}

\smallskip

{\bf Acknowledgements.} Dachun Yang would like to thank Professor
Jacek Dziuba\'nski for some stimulating conversations on this subject
and Professor Heping Liu for providing his preprint \cite{lll}
and some useful discussions. Both authors would like to thank Professor
Wengu Chen for some useful advices on this subject and also
Doctor Liguang Liu for her improvement on Proposition \ref{p3.1}.
They would also like to thank the referee for
his many valuable remarks which improve the presentation of this
article.

\bigskip

\noindent Dachun Yang and Yuan Zhou

\medskip

\noindent School of Mathematical Sciences,
 Beijing Normal University,
 Laboratory of Mathematics and Complex Systems, Ministry of Education,
 Beijing 100875, People's Republic of China

\noindent{\it E-mail addresses}:
\texttt{dcyang@bnu.edu.cn} and \texttt{yuanzhou@mail.bnu.edu.cn}

\end{document}